\documentclass[12pt]{article}
\usepackage[utf8]{inputenc}
\usepackage[english]{babel}

\usepackage{hyperref}
\usepackage{amsmath,amssymb,amsfonts,amsthm,amsbsy, epsfig, psfrag,epsf}
\usepackage{graphicx}
\usepackage{xr}
\usepackage{enumerate}
\usepackage{natbib}
\usepackage{url} 
\usepackage{rotating}
\usepackage{boxedminipage}
\usepackage{algorithm}
\usepackage{subfigure}
\usepackage{multirow}
\usepackage{enumerate}
\usepackage{array}
\usepackage{xcolor}
\usepackage{algpseudocode}
\usepackage{pifont}
\usepackage{fullpage}
\usepackage{epstopdf}

\hypersetup{colorlinks=true, linkcolor=blue, urlcolor=blue, citecolor=blue}
\renewcommand{\baselinestretch}{1.43}

\newtheorem{theorem}{Theorem}[section]
\newtheorem{proposition}[theorem]{Proposition}
\newtheorem{lemma}[theorem]{Lemma}
\newtheorem{corollary}[theorem]{Corollary}
\newtheorem{remark}{Remark}[section]
\newtheorem{example}[theorem]{Example}
\newtheorem{definition}[theorem]{Definition}

\newcommand{\be}{\boldsymbol{\epsilon}}
\newcommand{\bt}{{\boldsymbol{\theta}}}
\newcommand{\bbt}{{\boldsymbol{\beta}}}
\newcommand{\hbbt}{\hat{\boldsymbol{\beta}}}
\newcommand{\btheta}{{\boldsymbol{\theta}}}

\newcommand{\bg}{\boldsymbol{\gamma}}
\newcommand{\bxi}{\boldsymbol{\xi}}

\newcommand{\E}{\mathbb{E}}

\newcommand{\R}{\mathbb{R}}

\newcommand{\G}{\mathcal{G}}

\newcommand{\by}{\mathbf{y}}
\newcommand{\bx}{\mathbf{x}}
\newcommand{\bz}{\mathbf{z}}
\newcommand{\bu}{\mathbf{u}}
\newcommand{\bv}{\mathbf{v}}
\newcommand{\bw}{\mathbf{w}}
\newcommand{\ba}{\mathbf{a}}
\newcommand{\bb}{\mathbf{b}}
\newcommand{\bc}{\mathbf{c}}
\newcommand{\bd}{\mathbf{d}}
\newcommand{\bh}{\mathbf{h}}
\newcommand{\X}{\mathcal{X}}
\newcommand{\C}{\mathcal{C}}
\newcommand{\Q}{\mathcal{Q}}

\newcommand{\tbt}{\widetilde{\boldsymbol{\theta}}}

\newcommand{\Proj}{\text{Proj}}
\newcommand{\df}{\text{df}}
\newcommand{\hbt}{{\hat{\boldsymbol{\theta}}}}
\newcommand{\hbxi}{\hat{\boldsymbol{\xi}}}

\renewcommand{\df}{{\mathrm{df}}}
\renewcommand{\df}{{\mathrm{df}}}

\DeclareMathOperator*{\argmin}{arg\,min}
\DeclareMathOperator*{\argmax}{arg\,max}

\let\hat\widehat

\makeatletter

\newcommand{\Rmnum}[1]{\expandafter\@slowromancap\romannumeral #1@}
\makeatother

\definecolor{DSgray}{cmyk}{0,1,0,0}

\newcounter{rcnt}[section]
\renewcommand{\thercnt}{(\roman{rcnt})}
\newcommand{\red}{\color{red}}

\newcommand{\blind}{0}

\begin{document}

\def\spacingset#1{\renewcommand{\baselinestretch}%
{#1}} \spacingset{1.4}

 \title{\bf On Degrees of Freedom of Projection Estimators with Applications to Multivariate Nonparametric Regression}
\author{Xi Chen \footnote{Supported by Alibaba Innovation Research Award and Bloomberg Data Science Research Award; e-mail: xchen3@stern.nyu.edu}  \\
	Stern School of Business, New York University \\
	and \\
	Qihang Lin \footnote{e-mail: qihang-lin@uiowa.edu} \\ 
	Tippie College of Business, University of Iowa \\
	and \\
	Bodhisattva Sen \footnote{Supported by NSF grants DMS-1712822 and AST-1614743; e-mail: bodhi@stat.columbia.edu}  \\
	Department of Statistics, Columbia University
}
\date{}
\maketitle

\begin{abstract}
In this paper, we consider the nonparametric regression problem with multivariate predictors. We provide a characterization of the degrees of freedom and divergence for estimators of the unknown regression function, which are obtained as outputs of linearly constrained quadratic optimization procedures; namely, minimizers of the least squares criterion with linear constraints and/or quadratic penalties. 
As special cases of our results, we derive explicit expressions for the degrees of freedom in many   nonparametric regression problems, e.g., bounded isotonic regression, multivariate (penalized) convex regression, and additive total variation regularization. Our theory also yields, as special cases, known results on the degrees of freedom of many well-studied estimators in the statistics literature, such as ridge regression, Lasso and generalized Lasso. Our results can be readily used to choose the tuning parameter(s) involved in the estimation procedure by minimizing the Stein's unbiased risk estimate. 
As a by-product of our analysis we derive an interesting connection between bounded isotonic regression and isotonic regression on a general partially ordered set, which is of independent interest.
\end{abstract}

\noindent%

\section{Introduction}
Consider the problem of nonparametric  regression with observations $\{(\bx_i,y_i): i=1, \ldots, n\}$ satisfying
\begin{equation}\label{eq:RegMdl2}
	y_i = f(\bx_i) + \epsilon_i, \qquad \mbox{for } i = 1,\ldots, n,
\end{equation}
where $\epsilon_1,\ldots, \epsilon_n$ are i.i.d.~$N(0,\sigma^2)$ (unobserved) errors, $\bx_1, \ldots, \bx_n$ are design points in $\R^d$ ($d \ge 1$) and the regression function $f$ is unknown. In this paper we study the degrees of freedom and divergence of nonparametric estimators of $f$ that are obtained as outputs of linearly constrained quadratic optimization procedures, namely, minimizers of the least squares criterion with linear constraints and/or quadratic penalties. Letting $\bt^* := (f(\bx_1),\ldots, f(\bx_n))$, these problems are characterized by constraints on $\bt^*$ whereby $\bt^* \in \C$ for some suitable closed convex set $\C \subset \R^n$. We briefly introduce the three main examples we will study in detail in this paper, namely isotonic regression, convex regression, and additive total variation regularization. \newline

\noindent {\bf Example 1} (Isotonic regression) If $f$ is assumed to be nondecreasing and the $x_i$'s are univariate and ordered (i.e., $x_1< x_2 < \cdots < x_n$), then $\bt^* \in \mathcal{M}$, where
\begin{equation}\label{eq:IsoReg}
\mathcal{M} := \{\btheta \in \R^n: \theta_1 \le \theta_2 \le \ldots \le \theta_n\}.
\end{equation}
Isotonic regression has  a long history in statistics; see e.g.,~\citet{Brunk55}, \citet{AyerEtAl55}, and~\citet{vanEeden58}. Isotonic regression can be easily extended to the setup where the predictors take values in any space with a partial order; see Section~\ref{sec:DFIso} for the details.

The isotonic least squares estimator (LSE) $\hbt(\by)$, which is defined as the Euclidean  projection of $\by := (y_1,\ldots, y_n)$ onto $\mathcal{M}$, i.e.,
\begin{equation}\label{eq:Proj-Iso}
\hat \bt(\by) := \argmin_{\bt \in \mathcal{M}} \|\by - \bt \|_2^2
\end{equation}
(here $\|\cdot\|_2$ denotes the usual Euclidean norm) is a natural estimator in this problem and has many desirable properties (see e.g., \cite{GJ14}). However, it suffers from the ``spiking'' effect \citep{WS93,Pal08}, i.e., it is inconsistent at the boundary of the covariate domain.  For multivariate predictors, this over-fitting of the LSE can be even more pronounced and some recent research has focused on studying the regularized isotonic LSE (see e.g.,~\cite{LRS12,LR14,WMO15}). A natural way to regularize the model complexity would be  to consider {\it bounded isotonic} regression: $\bt^*$ is assumed to be nondecreasing and the range of $\bt^*$ is assumed to be bounded by $\lambda$, for $\lambda >0$. In Section~\ref{sec:DFIso}, we show that for bounded isotonic regression,  $\bt^* = (f(\bx_1),\ldots, f(\bx_n))$ belongs to a closed polyhedral set $\mathcal{C}$ (i.e., an intersection of finitely many hyperplanes) that can be expressed in the general form as $ \vspace{-.03in}$
\begin{equation}\label{eq:CvxPoly}
\mathcal{C} =  \{\btheta \in \R^n: B \btheta \le \mathbf{c} \}  \vspace{-.03in}
\end{equation}
for some suitable matrix $B\in \mathbb{R}^{m \times n}$ and a vector $ \mathbf{c} \in \mathbb{R}^{m \times 1}$; here the inequality between vectors is understood in a component-wise sense. \newline



\noindent {\bf Example 2} (Convex regression) In convex regression (see e.g.,~\citet{Hildreth54}, \citet{K08}, \citet{SS11}, \citet{LG12}, \citet{Xu14Convex}, \citet{HanWellner:16}) $f: \R^d \to \R$  is known to be a convex function  (see~\eqref{eq:RegMdl2}) and $\bx_1,\ldots, \bx_n$ are the design points in $\R^d$, $d \ge 1$. Letting $\bt^* := (f(\bx_1),\ldots, f(\bx_n))$, it can be shown that the convexity of $f$ is equivalent to $\bt^*$ belonging to a convex polyhedral set $\C$. For example, when $d=1$ and the $x_i$'s are ordered, $\C$ has the following simple characterization:
\begin{equation}\label{eq:CvxReg}
\C = \left\{\btheta \in \R^n: \frac{\theta_2 - \theta_1}{x_2 - x_1} \le \ldots \le \frac{\theta_n - \theta_{n-1}}{x_n - x_{n-1}} \right\}.
\end{equation}
However, for $d \ge 2$, the characterization of the underlying convex set $\C$ is more complex. In this case, there must exist a auxiliary vector $\bxi:=[\bxi_1^\top, \ldots, \bxi_n^\top]^\top \in\mathbb{R}^{dn}$  representing the subgradient of $f(\bx_j)$, for $j=1, \ldots, n$, such that $\left\langle\bxi_j,\bx_i-\bx_j \right\rangle\leq \theta_i-\theta_j$, for $i,j=1,\dots,n$. Thus, $\C$ can be expressed as the {\it projection} of the higher-dimensional polyhedron
\begin{equation}\label{eq:Q_cvxreg}
\left\{(\bxi,\bt) \in \R^{dn +n} :\bxi=[\bxi_1^\top, \ldots, \bxi_n^\top]^\top, \left\langle\bxi_j,\bx_i-\bx_j \right\rangle\leq \theta_i-\theta_j,\forall \;i,j=1,\dots,n \right\},
\end{equation}
onto the space of $\bt$. Although the projection of a polyhedron is still a polyhedron, it is difficult to express $\C$ in the form  of~\eqref{eq:CvxPoly} explicitly.

As before, a natural estimator of $\bt^*$ in this problem is the LSE defined as in~\eqref{eq:Proj-Iso} with $\mathcal{M}$ replaced by $\C$. For multivariate designs, the classical convex LSE tends to overfit the data, especially near the boundary of the convex hull of the design points. To avoid this over-fitting, \cite{SM13} and \cite{Lim14} propose a  regularization technique using the norm of the subgradients, which leads to penalized convex regression (see Section \ref{sec:DFCvxReg} for the details). \newline

\noindent {\bf Example 3} (Additive total variation regression) Suppose that $d=1$ and $f$ (as defined in~\eqref{eq:RegMdl2}) is a function of bounded variation. In this case a popular estimator of $f$ is to consider the total variation (TV) regularized regression (\cite{RudinEtAl92}; also see~\cite{MV97}) which can be expressed as $ \vspace{-.03in}$
\begin{equation}\label{eq:TV}
\hat \bt(\by) = \arg \min_{\bt \in \R^n} \sum_{i=1}^n (y_i - \theta_i)^2 + \lambda \sum_{i=2}^n |\theta_i - \theta_{i-1}| \vspace{-.03in}
\end{equation}
where $\lambda >0$ is a tuning parameter. The presence of the $\ell_1$-norm in the penalty term in~\eqref{eq:TV} ensures sparsity of the vector $(\hat \theta_2 - \hat \theta_1,\ldots, \hat \theta_n - \hat \theta_{n-1})$; thus $\hat \bt(\by)$ is piecewise constant with adaptively chosen break-points. The motivation for using~\eqref{eq:TV} to estimate $\bt^* := (f(\bx_1),\ldots, f(\bx_n))$ comes from the belief that $\bt^*$ lies in the closed convex set  $\C = \big \{\bt \in \R^n: \sum_{i=2}^n |\theta_i - \theta_{i-1}| \le V \big\}$ for some $V >0$; indeed~\eqref{eq:TV} expresses the above constraint in the penalized form. TV regularization has many important applications, especially in image processing; also see the closely related method of fused Lasso (\cite{Fused05}).

When we have multidimensional predictors, i.e., $d > 1$, to alleviate the curse of dimensionality, it is useful to consider an additive model of the form $f(x_{1}, \ldots, x_{d}) := \sum_{j=1}^d f_j(x_{j})$, where each $f_j(\cdot)$ is assumed to be of bounded variation. A natural estimator in this scenario, which is an extension of~\eqref{eq:TV}, is the additive TV regression (\cite{Petersen:16}), where we minimize the sum of squared errors constraining the sum of the variations for each $f_j(\cdot)$. We study this estimator in Section~\ref{sec:AddMdl}. In fact, we consider a more general setup where each $f_j(\cdot)$ can have different degrees of ``smoothness''.   \newline

All the above three examples can be succinctly expressed in the Gaussian sequence model:
\begin{eqnarray}\label{eq:SeqMdl}
	\by = \bt^* + \be,
\end{eqnarray}
where we observe $\by = (y_1,\ldots, y_n)\in \R^n$, $\bt^* = (\theta_1^*,\ldots, \theta_n^*) \in \R^n$ is the unknown parameter of interest known to belong to a given closed convex set $\C \subseteq \R^n$ (recall that $\bt^*$ corresponds to $(f(\bx_1),\ldots, f(\bx_n))$), and $\be \sim N(\mathbf{0},\sigma^2 I_n)$ ($I_n$ is the $n \times n$ identity matrix) is the unobserved error. Let $\hat \bt(\by) := (\hat \theta_1,\ldots, \hat \theta_n)$ be an estimator of $\bt^*$. The ``degrees of freedom'' of $\hat \bt(\by)$ (see \citet{Efron04}) is defined as
\begin{equation}\label{eq:df_cov}
\mbox{df}(\hat \bt(\by)) := \frac{1}{\sigma^2} \sum_{i=1}^n \mbox{Cov}(\hat \theta_i,y_i).
\end{equation}
Degrees of freedom (DF) is an important concept in statistical modeling and is often used to quantify the model complexity of a statistical procedure; see e.g., \citet{MW00},~\citet{Zou07},~\citet{TT12},  and the references therein. Intuitively, the quantity $\mbox{df}(\hat \bt(\by))$ reflects the effective number of parameters used by $\hat \bt(\by)$ in producing the fitted output, e.g., in linear regression, if $\hat \bt(\by)$ is the LSE of $\by$ onto a subspace of dimension $d<n$, the DF of $\hat \bt(\by)$ is simply $d$. Using Stein's lemma it follows that (see~\citet{MW00} and~\citet{TT12})
$$ \mbox{df}(\hat \bt(\by)) = \E [D(\by)]$$ where
\begin{equation}\label{eq:Div}
	D(\by) = \mbox{div}(\hat \bt(\by)) := \sum_{i=1}^n \frac{\partial} {\partial y_i} \hat \theta_i(\by) = \nabla_{\by} \hat{\bt}(\by)
\end{equation}
is called the {\it divergence} of $\hat \bt(\by)$. Thus, $D(\by)$ is an unbiased estimator of df$(\hat \bt(\by))$. This has many important implications, e.g., Stein's unbiased risk estimate (SURE); see~\citet{S81}. Aside from plainly estimating the risk of an estimator, one could also use SURE for model selection purposes: if the estimator depends on a tuning parameter, then one could choose this parameter by minimizing SURE. This has been successfully used in many statistical problems, see e.g.,~\citet{DJ94},~\citet{XieEtAl12},~\citet{Candes13}, and~\citet{FengZou12} for applications in wavelet denoising,   heteroscedastic hierarchical models, singular value thresholding, and bandable covariance matrices, respectively. We elaborate on this connection in Section~\ref{sec:SURE}.


In this paper we develop a theoretical framework to evaluate the divergence (as defined in~\eqref{eq:Div}) for a broad class of (nonparametric) regression estimators that are minimizers of the least squares criterion with linear constraints and/or quadratic penalties. Our theory also recovers many existing results (see Section \ref{sec:proof_other_app} in the supplementary material), which include the exact expressions for divergence for ridge regression (see~\citet{Li86}) and the active set representation of the divergence for Lasso and generalized Lasso (see~\citet{Zou07} and \citet{TT12}).  


In the following we motivate the general form of the estimators we study in this paper. In many  regression problems, $\bt^* \in \C\subset \R^n$ where $\C$ is a polyhedron. Moreover, in many of these problem (e.g., convex regression) $\C$ is not easily expressible in the form~\eqref{eq:CvxPoly}, but can be described as the projection of a higher-dimensional polyhedron of $(\bxi,\bt)$ onto the space of $\bt$ (see e.g.,~\eqref{eq:Q_cvxreg}). In particular, this higher-dimensional polyhedron can, in general, be represented as
\begin{equation}\label{eq:Q}
\Q:=\{(\bxi,\bt) \in \R^{p +n} : A \bxi + B \bt \leq \bc\}
\end{equation}
where $\bxi \in \R^p$ is the auxiliary variable and $A \in \mathbb{R}^{m \times p}$, $B \in \mathbb{R}^{m \times n}$ and $\bc \in \mathbb{R}^{m}$ are suitable matrices. The true parameter $\bt^*$ thus belongs to the set $\C := \Proj_{\bt}(\Q)$  defined as
\begin{equation}\label{eq:proj_C}
\Proj_{\bt}(\Q):=\{\bt \in \R^n: \exists \; \bxi \in \R^{p} \; \text{such that} \; (\bxi, \bt) \in \Q\}.
\end{equation}
{A natural estimator of $\bt^*$ in this situation is the LSE $\hbt(\by) :=\argmin_{\bt \in \Proj_{\bt}(\Q)} \frac{1}{2} \|\bt-\by\|_2^2,$ which is equivalent to  $(\hbt(\by), \hbxi(\by)) \in \argmin_{(\bt, \bxi) \in \Q}\frac{1}{2} \|\bt-\by\|_2^2$. Instead of considering this partially projected LSE, we study a more general formulation by adding \emph{linear} and \emph{quadratic perturbations} in the objective function to accommodate more applications:}
\begin{eqnarray}\label{eq:LSE}
  (\hbt(\by), \hbxi(\by)) & \in &  \argmin_{\bt, \bxi}\frac{1}{2} \|\bt-\by\|_2^2 +\bd^\top\bxi+ \frac{\lambda}{2}\|\bxi\|_2^2\\
            & & \;\;\; \text{s.t.} \;   A\bxi + B\bt \leq \bc, \nonumber
\end{eqnarray}
where $A=[\ba_1, \ldots, \ba_m]^\top \in \mathbb{R}^{m \times p}$, $B=[\bb_1, \ldots, \bb_m]^\top \in \mathbb{R}^{m \times n}$, $\bc \in \mathbb{R}^{m}$, $\bd \in \mathbb{R}^{p}$ and $\lambda \geq 0$ is a regularization parameter. As we will show below~\eqref{eq:LSE} finds many statistical applications beyond the examples described above.  Note that the objective function in~\eqref{eq:LSE} is strongly convex in $\bt$ and convex in $\bxi$; moreover, if $\lambda > 0$, it is strongly convex in both $\bt$ and $\bxi$.

Formulation~\eqref{eq:LSE} covers a wide range of useful estimators in shape-restricted nonparametric regression, additive total variation regression, and Lasso-related problems. For example, when $\bd=\mathbf{0}$, $\lambda=0$ but $A$ is not a zero matrix,~\eqref{eq:LSE} becomes
\begin{eqnarray}
     \label{eq:LSE_mul_cvx}
          (\hbt(\by), \hbxi(\by))  = \argmin_{(\bt, \bxi)\in \mathcal{Q}} \frac{1}{2}\|\bt-\by\|_2^2,
\end{eqnarray}
where $\mathcal{Q}$ is defined  in~\eqref{eq:Q}. This formulation can also be viewed as the projection of $\by$ onto a polyhedron $\Proj_{\bt}(\Q)$ defined in \eqref{eq:proj_C}. This class of problems include the LSE in multivariate convex regression for which DF has not been studied before (see Section~\ref{sec:DFCvxReg} for the details). Based on \eqref{eq:LSE_mul_cvx},  if we further have $\bd \neq \mathbf{0}$, then \eqref{eq:LSE} reduces to
\begin{equation}\label{eq:linear_per}
 (\hbt(\by), \hbxi(\by))  = \argmin_{(\bt, \bxi)\in \mathcal{Q}} \frac{1}{2}\|\bt-\by\|_2^2+\bd^\top\bxi.
\end{equation}
This formulation includes many  examples in statistics, such as additive TV regression  (see Example 3 above) and $\ell_\infty$-regularized group Lasso (see Section \ref{sec:other_app}). Moreover, when $\bd=\mathbf{0}$ and $\lambda>0$ in \eqref{eq:LSE}, the corresponding optimization problem becomes
       \begin{eqnarray}
      \label{eq:LSE_Lip_multi_conv}
              (\hbt(\by), \hbxi(\by))  = \argmin_{(\bt, \bxi)\in \mathcal{Q}} \frac{1}{2}\|\bt-\by\|_2^2+\frac{\lambda}{2}\|\bxi\|_2^2,
        \end{eqnarray}
which includes the example of {penalized  multivariate convex regression}, where the norm of the subgradient $\bxi$ is penalized.

In the following we briefly describe some of the main contributions of this paper.
\begin{enumerate}
  \item  We characterize the divergence and DF of $\hat \bt(\by)$, as defined in \eqref{eq:LSE}, by providing easy-to-compute formulas. Our main result, Theorem \ref{thm:div}, can be used to compute the divergence and DF in any statistical regression problem where the estimator can be expressed in the form~\eqref{eq:LSE}. A special case of \eqref{eq:LSE} --- projection onto a convex polyhedron --- has been studied in the literature \citep{Kato09DF, TT12} where
      \begin{eqnarray}
       \label{eq:Proj}
       \hbt(\by)=P_\C(\by):=\argmin_{\bt\in \mathcal{C} }\frac{1}{2} \|\bt-\by\|_2^2, \vspace{-.07in}
       \end{eqnarray}
       and $\mathcal{C} =  \{\btheta \in \R^n: B \btheta \le \mathbf{c} \}$ is as defined in \eqref{eq:CvxPoly}.
Our main theorem generalizes these previous results. In particular, when $\bd\neq 0$ and $\lambda=0$ in \eqref{eq:LSE}, the problem is challenging as now $\hbt(\by)$ cannot be written as a projection estimator. When $\lambda>0$, although \eqref{eq:LSE} can be viewed as a projection problem in a higher dimensional space, the previous results on the projection estimator  cannot be directly applied to obtain the divergence of $\hat \bt(\by)$ (see Remark \ref{rem:nontrivial} for details).


  \item Using our main result we derive the DF for many estimators, including multivariate convex regression, penalized convex regression, (bounded) isotonic regression, additive TV regression, $\ell_\infty$-regularized group Lasso, etc. Note that although the divergences and DF for Lasso and generalized Lasso have been characterized in \citet{Zou07} and~\citet{TT12} we demonstrate that we recover their results (in the active set representation) as straightforward consequences of Theorem~\ref{thm:div}; see Section \ref{sec:proof_other_app} in the supplement for the details.

  \item {For bounded isotonic regression  where the design points are allowed to belong to any partially ordered set, we establish the equivalence between the divergence of the isotonic LSE and the number of connected components of the graph induced by the LSE (see Proposition \ref{prop:uni_bound_iso}). This result is not only theoretically interesting but also provides a fast algorithm for computing the divergence in this problem.  Moreover, we establish a  connection between the LSE for bounded isotonic regression and that for unbounded isotonic regression, a result which is of independent interest. In particular, we show that the bounded isotonic LSE  can be easily obtained by appropriately thresholding the unbounded isotonic LSE (see Proposition \ref{prop:threshold_gen}). Further, using this property, we show the monotonicity of divergence (and DF) as a function of the model complexity parameter  --- this shows that DF indeed characetrizes model complexity --- for bounded isotonic regression.}

\end{enumerate}

In the following we compare and contrast our results with some of the recent work on divergence and DF for projection estimators. {\cite{Kato09DF} characterizes the DF in shrinkage regression where the coefficients belong to a closed convex set. The estimation problem considered by \cite{Kato09DF} contains \eqref{eq:LSE_mul_cvx} as a special case but his result cannot be directly applied to \eqref{eq:linear_per} when $\bd\neq\mathbf{0}$. As a consequence, \cite{Kato09DF} can characterize DF for generalized Lasso expressed in a constrained form while we can characterize the DF in the penalized form (as described in Section~\ref{sec:proof_other_app} of the supplementary file).}
\citet{Hansen14DF} consider the closed constraint set $\C=\zeta(\mathcal{B})$ where $\mathcal{B} \subseteq \R^p$ is a closed set and $\zeta: \R^p \rightarrow \R^n$ is a (possibly non-linear) map satisfying some regularity conditions. 
Their main result (Theorem 3) requires the optimal solution $\hat{\bbt}$ to be in the \emph{interior} of $\mathcal{B}$ (which is almost never the case in the examples of interest to us) and a variant of the Hessian matrix of $\zeta(\hat{\bbt})$ to be full rank (e.g., when $\zeta(\bbt)=X\bbt$, it requires that $X^\top X$ is full rank). The results in \cite{Hansen14DF} can only deal with a constraint set that can be explicitly written as a set of inequalities (e.g., the general projected polyhedron $\Proj_\theta(\Q)$ in \eqref{eq:proj_C} is not allowed) and cannot be applied to regularized estimators (e.g.,~generalized Lasso as described in Section~\ref{sec:proof_other_app} of the supplementary file and penalized multivariate convex regression as described in Section \ref{sec:DFCvxReg}). \citet{Vaiter14DF} study DF for a class of regularized regression problems that include Lasso and group Lasso as special cases. However, their paper does not consider constrained formulations and thus cannot be applied to shape restricted regression problems. \citet{Mikkelsen:16} provide a characterization of DF for a class of estimators which are locally Lipschitz continuous on each of a finite number of open sets that cover $\mathbb{R}^n$.
\citet{Rueda13DF} utilize the results of \citet{MW00} to study the DF for the specific problem of semiparametric additive (univariate) monotone regression.

In the recent papers \citet{Kaufman14} and \citet{Janson15} the authors argue that in many problems DF might not be an appropriate notion for characterizing model complexity. They provide counter examples of situations where DF is not monotone in the model complexity parameter (or DF is unbounded). However, most of these counter examples either involve nonconvex constraints or non-Gaussian or heteroscedastic noise --- in \cite{Janson15} it is argued that such irregular behavior happens ``whenever we project onto a nonconvex model". Nevertheless, some of the main applications in our paper, namely, bounded isotonic regression and additive total variation regression, correspond to projections onto polyhedral convex sets with i.i.d.~Gaussian noise so the irregular behavior of DF, observed in some of the counter examples, may not occur here. In fact, in Theorem~\ref{thm:iso_bounded_monotone}  we prove that for bounded isotonic regression, DF is indeed monotone in the model complexity parameter. 

The paper is organized as follows. 
In Section~\ref{sec:DF} we provide some basic results on the divergence of projection estimators. In Section~\ref{sec:main} we state our main result. In Sections~\ref{sec:DFCvxReg}, \ref{sec:DFIso}, and \ref{sec:other_app}, we discuss many applications of our main result to different regression problems. In Section~\ref{sec:SURE} we discuss how the characterization of divergence of estimators (computed in the paper) can be useful in model selection (choice of tuning parameter) based on SURE, and illustrate this for bounded isotonic regression and penalized multivariate convex regression.   
We relegate all the technical proofs, graphical illustrations, as well as the derivation of some existing results (such as generalized Lasso) using our main theorem to the supplementary material.


\section{An Existing Result on DF}\label{sec:DF}

DF is an important concept in statistical modeling as it provides a quantitative description of the amount of fitting performed by a given procedure. Despite its fundamental role in statistics, its behavior is not completely well-understood, even for widely used estimators.

In this section we review an important known result on DF and the divergence of the projection estimator $\hat \bt(\by)$ when $\C$ is a convex polyhedron as defined in \eqref{eq:CvxPoly}; see~\eqref{eq:Proj}. We will assume that the reader is familiar with basic concepts from convex analysis (see Section~\ref{sec:convex} in the supplementary material where we provide a review of some basic concepts: polyhedron, cone, normal cone, affine hull, interior, boundary, relative interior, relative boundary, etc).

 {The following result, due to \citet[Lemma 3.2]{Kato09DF}\footnote{In fact, Lemma 3.2 in~\cite{Kato09DF} provides a more general result about the divergence of the projection estimator $\hat \bt(\by)$ when $\C$ is a closed convex set with piecewise smooth boundary.} and \citet[Lemma 2]{TT12}, shows that the divergence of the projection estimator $\hat \bt(\by)$ onto a convex polyhedron as described in \eqref{eq:CvxPoly}  can be calculated as the dimension of the affine space that $\hat \bt(\by)$ lies on.}

\begin{proposition}
\label{prop:div_poly}
Suppose that the projection estimator $\hat \bt(\by)$ is defined in~\eqref{eq:Proj} where $\C$ is a convex polyhedron as defined in~\eqref{eq:CvxPoly}. Then the components of $\hat \bt(\by)$ are almost differentiable, and $\nabla \hat \theta_i$ ($i$-th entry of $\nabla \hat \bt(\by)$) is an essentially bounded function, for $i=1,\ldots, n$. Let  $J_\by$ be the set of indices for all the binding constraints of $\hat \bt(\by)$, i.e.,
\begin{eqnarray}\label{eq:J_y}
 J_{\by}:=\{1 \leq i \leq m: \langle \bb_i, \hat{\bt}(\by) \rangle =c_i \}.
\end{eqnarray}
Then, for a.e.~$\by \in \mathbb{R}^n$, there is a neighborhood $U$ of $\by$, such that for every $\bz \in U$,
\begin{eqnarray}\label{eq:proj_aff}
\hat{\bt}(\bz)& = & \argmin_{\bt \in H} \frac{1}{2}\|\bt-\bz\|_2^2
\end{eqnarray}
where $H=\{\bt: B_{J_\by}\bt=\bc_{J_\by}\}$ is an affine space, $J_\by$ is defined in \eqref{eq:J_y} and $B_{J_{\by}}$ is the submatrix of $B$ with rows indexed by $J_{\by}$. As a consequence,
\begin{equation}\label{eq:div_poly}
D(\by)=n- \mathrm{rank}(B_{J_{\by}}), \qquad \qquad \mbox{for a.e.~$\by \in \mathbb{R}^n$},
\end{equation}
Thus, $\mathrm{df}(\hat \bt(\by))= n-\E\left[ \mathrm{rank}(B_{J_{\by}})\right]$.
\end{proposition}
Note that \emph{a.e.} in \eqref{eq:div_poly} stands for ``almost everywhere'', which means that \eqref{eq:div_poly} holds for all $\by$ except on a measure-zero set. 
{Note that, by an almost differentiable function $f:\mathbb{R}^n\rightarrow\mathbb{R}$ we mean that $f$ is differentiable everywhere except on a measure-zero set (see~\cite{MW00} for a precise definition); $f$ is essentially bounded if there exists an constant $c$ such that $f^{-1}((c,+\infty))$ is a measure-zero set.}

\section{Main Result}
\label{sec:main}

In this section we consider the estimator $\hbt(\by)$ obtained from the optimization problem \eqref{eq:LSE} with the auxiliary variable $\bxi \in \mathbb{R}^{p}$. 
When $\lambda=0$ and $\mathbf{d} \neq \mathbf{0}$, the optimization problem \eqref{eq:LSE} may have an unbounded optimal value depending on $\bd$. The following result gives the necessary and sufficient condition for \eqref{eq:LSE} to be bounded.
\begin{lemma}\label{lem:bounded}
	When $\lambda=0$, the optimization problem in~\eqref{eq:LSE} has a bounded optimal value if and only if $-\bd=A^\top\bu$ for some $\bu\geq\mathbf{0}$.
\end{lemma}
The proof of Lemma \ref{lem:bounded} is based on Farkas's lemma (see e.g.,~\citet[Corollary 22.3.1]{Rockafellar70book}) and is provided in Section~\ref{sec:proof_bounded} of the supplementary material. Based on the above lemma, for the rest of the paper, we will assume that $-\bd=A^\top\bu$ for some $\bu\geq\mathbf{0}$ so that \eqref{eq:LSE} is bounded. When $\bd=\mathbf{0}$ such an assumption trivially holds for $\bu=\mathbf{0}$. For applications with $\bd \neq \mathbf{0}$, e.g., additive model, generalized Lasso, and $\ell_\infty$-regularized group Lasso, we will show that this assumption always holds.

The divergence of $\hbt(\by)$, as the solution~\eqref{eq:LSE}, is characterized by the following theorem, which is the main result of the paper.
\begin{theorem}\label{thm:div}
	Suppose that $-\bd=A^\top\bu$ for some $\bu\geq\mathbf{0}$ whenever $\lambda=0$ in \eqref{eq:LSE}. For any
	$\by \in \R^n$, let $(\hbt(\by), \hbxi(\by)) $ be any solution for \eqref{eq:LSE} and let
	\begin{eqnarray}\label{eq:J_y_2}
	J_{\by}:=\{1 \leq i \leq m: \langle \ba_i, \hat{\bxi}(\by) \rangle +  \langle \bb_i, \hat{\bt}(\by) \rangle =c_i \},
	\end{eqnarray}
	and $A_{J_\by}$ and $B_{J_\by}$ be  the submatrices of $A$ and $B$ with rows in the set $J_\by$. Let $I_{\by} \subseteq J_\by$ be the index set of maximal independent rows of the matrix $[A_{J_\by}, B_{J_\by}]$, i.e., the set of vectors $\{[\ba_i^\top, \bb_i^\top], i \in I_\by\}$ are linearly independent.
	Then, the following statements hold:
	\begin{enumerate}
		\item[(i)] The optimal solution $(\hbt(\by), \hbxi(\by))$ of~\eqref{eq:LSE} has unique components $\hbt(\by)$. The components of $\hat \bt(\by)$ are almost differentiable in $\by$ and  $\nabla \hat \theta_i(\by)$  is an essentially bounded function for each $i=1,\ldots, n$.
		\item[(ii)] For a.e. $\by$,
		{\small \begin{equation}\label{eq:LSE_div}
			D(\by)=\begin{cases} n - \mathrm{trace}\left(B_{I_\by}^\top \left(B_{I_\by}B_{I_\by}^\top+ \frac{1}{\lambda} A_{I_\by} A_{I_\by}^\top\right)^{-1}B_{I_\by} \right), &\text{ if }\lambda>0,\\
			n- |I_\by| + \mathrm{rank}(A_{I_\by}),&\text{ if }\lambda=0,
			\end{cases}
			\end{equation}}
		and $\df(\hat\bt(\by))=\E[D(\by)]$ (note that the index set $I_\by$ is random).
	\end{enumerate}
\end{theorem}
First note that any solution $(\hbt(\by), \hbxi(\by)) $ of~\eqref{eq:LSE} depends on $\bd$ and so do $J_\by$ and $I_\by$. Hence, $D(\by)$ given in~\eqref{eq:LSE_div} depends on $\bd$ implicitly. To simplify notation, we suppress the dependence of $J_\by$, $I_\by$ and $D(\by)$ on $\bd$. {The divergence in \eqref{eq:LSE_div} holds for any given $\mathbf{d} \in \R^p$ and for every $\by \in \mathbb{R}^n$ expect for a measure-zero set in $\mathbb{R}^n$. The explicit form of this measure zero set is provided in our proof (see~\eqref{eq:y_bd_lem2} in the supplementary file for the case $\lambda=0$  and~\eqref{eq:y_bd_lem3} when $\lambda>0$).}

We also note that when $\lambda>0$, $B_{I_\by}B_{I_\by}^\top+ \frac{1}{\lambda} A_{I_\by} A_{I_\by}^\top$ is invertible. To see this observe that, from the definition of $I_{\by}$, the rows of $V := [\frac{1}{\sqrt{\lambda}}A_{I_\by}, B_{I_\by}]$ are linearly independent. Therefore, $B_{I_\by}B_{I_\by}^\top+ \frac{1}{\lambda} A_{I_\by} A_{I_\by}^\top
= V V^\top$
is invertible.
Further, as a simple sanity check of Theorem \ref{thm:div}, we show in Lemma~\ref{lem:SanityCheck} (see Section~\ref{sec:SanityCheck} of the supplementary file) that $D(\by)$, as defined in~\eqref{eq:LSE_div}, is always nonnegative.
A few important remarks are in order now.
{
\begin{remark}\label{rem:nontrivial}
\emph{
When $\lambda>0$,  we can define $\bd_{\lambda}:=\frac{-\bd}{\sqrt{\lambda}}$ and can reformulate~\eqref{eq:LSE} as a projection problem
\begin{eqnarray}\label{eq:reform_Lip_multi_conv_lambdanonzero}
(\hbt(\by,\bd_{\lambda}), \hat{\bg}(\by,\bd_{\lambda})) & = &  \argmin_{\bt, \bg}\frac{1}{2} \left\|(\bt,\bg)-(\by,\bd_{\lambda})\right\|_2^2\\
& \text{s.t.} &   \frac{1}{\sqrt{\lambda}}A\bg + B\bt \leq \bc. \nonumber
\end{eqnarray}
It is easy to verify that  $\hat{\bg}=\sqrt{\lambda}~\hat{\bxi}$ and that \eqref{eq:reform_Lip_multi_conv_lambdanonzero} is just an instance of \eqref{eq:Proj} in $\mathbb{R}^{p+n}$ by viewing $(\hbt,\hat{\bg})$, $(\by,\bd_{\lambda})$ and the feasible domain $\{(\bt,\bg) \in \R^{p+n} : \frac{1}{\sqrt{\lambda}}A \bg + B \bt \leq \bc\}$ in \eqref{eq:reform_Lip_multi_conv_lambdanonzero} as $\hbt$, $\by$ and $\mathcal{C}$ in \eqref{eq:Proj}, respectively. Hence, by applying Proposition \ref{prop:div_poly} to \eqref{eq:reform_Lip_multi_conv_lambdanonzero},  we can show that, for a.e.~$(\by,\bd_{\lambda}) \in \mathbb{R}^{p+n}$, there is a neighborhood $U$ of $(\by,\bd_{\lambda})$, such that for every $(\bz,\bb) \in U$,  the solution $(\hbt(\bz,\bb), \hat{\bg}(\bz,\bb))$ defined in \eqref{eq:reform_Lip_multi_conv_lambdanonzero} is the projection of $(\bz,\bb)$ to the affine space $\{(\bt, \bg):  \frac{1}{\sqrt{\lambda}}A_{I_\by}\bg + B_{I_\by}\bt = \bc_{I_\by} \}$ with $I_\by$ defined the same as in Theorem~\ref{thm:div}. In other words, for every $(\bz,\bb) \in U$,
$$
\left[
\begin{array}{c}
\hbt(\bz,\bb)\\
\hat{\bg}(\bz,\bb)
\end{array}
\right]
=(I-P)
\left[
\begin{array}{c}
\bz\\
\bb
\end{array}
\right],~
\text{ where } P=\left[
	\begin{array}{c}
	B_{I_\by}^\top\\
	\frac{1}{\sqrt{\lambda}}A_{I_\by}^\top
	\end{array}
	\right] \left(B_{I_\by}B_{I_\by}^\top+ \frac{1}{\lambda} A_{I_\by} A_{I_\by}^\top\right)^{-1}\left[B_{I_\by},\frac{1}{\sqrt{\lambda}}A_{I_\by}\right].
$$
Therefore, for a.e.~$(\by,\bd_{\lambda}) \in \mathbb{R}^{p+n}$, the matrix $I-P$ is the Jacobian matrix of $(\hbt(\by,\bd_{\lambda}), \hat{\bg}(\by,\bd_{\lambda}))$ and
we obtain \eqref{eq:LSE_div} for $\lambda>0$ by taking the trace of the $n \times n$ top-left block of $I-P$.
}

\emph{
Unfortunately, this argument cannot serve as a proof for Theorem~\ref{thm:div} when $\lambda>0$ as the above argument only holds for almost every $(\by,\bd_{\lambda})$ in $\mathbb{R}^{p+n}$ but \emph{not necessarily} for almost every $\by$ in $\mathbb{R}^n$ for a given $\bd_{\lambda}$. This is because the projection of a zero-measure set  in $\mathbb{R}^{p+n}$ (i.e., the set of $(\by,\bd_{\lambda})$'s) onto the space of $\by$ is not necessarily a zero-measure set in $\mathbb{R}^n$. But our main result in Theorem~\ref{thm:div} shows that \eqref{eq:LSE_div} holds for almost every $\by \in \mathbb{R}^n$ and any given $\bd_{\lambda}\in \mathbb{R}^p$. In Section \ref{sec:supp_nontrivial} in the supplementary material, we present a concrete example which shows that the  entire set of $(\by,\bd_{\lambda})$ with a given $\bd_{\lambda}$ falls into the measure-zero part on which the previous results from  \cite{Kato09DF} and \cite{TT12} fail.}
\end{remark}
}

\begin{remark}\label{rem:nontrivialzer0}
	\emph{
		When $\lambda=0$,
		using the strong duality of linear programming, we can reformulate \eqref{eq:LSE} and $\hbt(\by)$ as follows:
	}
	\begin{eqnarray}\label{eq:proximalmapping}
	\hbt(\by) & \in &  \argmin_{\bt}\frac{1}{2} \|\bt-\by\|_2^2 +g(\bt),
	\end{eqnarray}
	\emph{where $g(\bt)$ is a piece-wise linear convex function:}
	\begin{eqnarray}\label{eq:prox_term}
	g(\bt)&:=&\left\{
	\begin{array}{ll}
	\min\limits_{\bxi}\bd^\top\bxi	\text{ s.t. }A\bxi \leq \bc-B\bt &\text{ if  }\{\bxi|A\bxi \leq \bc-B\bt\}\neq\emptyset\\
	+\infty &\text{ if }\{\bxi|A\bxi \leq \bc-B\bt\}=\emptyset.
	\end{array}
	\right. \\
	&=&\left\{
	\begin{array}{ll}
	\max\limits_{\bu}(B\bt-\bc)^\top\bu	\text{ s.t. }A^\top\bu =-\bd, \bu\geq\mathbf{0} &\text{ if  }\{\bxi|A\bxi \leq \bc-B\bt\}\neq\emptyset\\
	+\infty &\text{ if }\{\bxi|A\bxi \leq \bc-B\bt\}=\emptyset.
	\end{array}
	\right.\nonumber
	\end{eqnarray}
	\emph{
		The formulation \eqref{eq:proximalmapping} means that $\hbt(\by)$ is the proximal mapping of $\by$ with a proximal term $g$ (Definition 1.22 in \cite{Rockafellar:11}). We note that Exercise 13.45 from \cite{Rockafellar:11} characterizes the generalized Jacobian of a proximal mapping, which can be a potential tool to derive $D(\by)$. However, due to the complicated form of the proximal term $g$ in \eqref{eq:prox_term}, it is not easy to directly apply their result to derive the explicit expression of the divergence in our Theorem \ref{thm:div}, and it requires to first introduce many new notions (e.g., second order generalized derivative for nonsmooth functions and graphical derivative) in variational analysis. On the other hand, our proof for the case of $\lambda = 0$ is more elementary and more consistent with the proof when $\lambda > 0$ --- both of them are based on a general local projection lemma (see Lemma \ref{lem:div} below).
	}
\end{remark}

\begin{remark}\label{rem:compute_Iy}
\emph{
The computation of the index set $J_\by$ is straightforward. Given a solution $\hat{\bxi}(\by)$ and  $\hat{\bt}(\by)$ from an optimization solver, we could easily check if $\langle \ba_i, \hat{\bxi}(\by) \rangle +  \langle \bb_i, \hat{\bt}(\by) \rangle$ equals $c_i$, for each $1 \leq i \leq m$. After obtaining $J_\by$, the index set $I_{\by}$ of maximal independent rows can be found by removing all the rows of  $[A_{J_\by}, B_{J_\by}]$ whose removal does not change the rank of the original matrix $[A_{J_\by}, B_{J_\by}]$. In particular, we start with an index set  $K=J_\by$. For each row index $k\in K$, if the rank of $[A_{K\backslash\{k\}}, B_{K\backslash\{k\}}]$ is the same as that of $[A_{K}, B_{K}]$, we remove $k$ from $K$. (Note that the rank can be computed easily by singular value decomposition or by directly applying the \emph{rank} function in Matlab or \emph{rankMatrix} function in R.) We repeat this procedure until no additional index in $K$ can be removed without reducing the rank of the matrix. The obtained index set $K$ is $I_\by$.
}
\end{remark}
\begin{remark}\label{rem:multiple}
\emph{
{When $\lambda=0$, it is possible that there exist multiple $\hbxi(\by)$'s satisfying \eqref{eq:LSE} and they correspond to different $J_{\by}$'s and $I_{\by}$'s; while when $\lambda>0$, $\hbxi(\by)$ is unique.} Even if $\hbxi(\by)$ and $J_{\by}$ are unique, there can still exist multiple maximal independent sets $I_{\by}$. However, according to our proof, for any given  $\hbxi(\by)$, $J_{\by}$ and $I_{\by}$, we show that $D(\by)$ equals the quantity on  the right hand side of \eqref{eq:LSE_div}. Note that $D(\by)$ is well-defined (see its definition in \eqref{eq:Div}), unique and does not depend on the choice of $\hbxi(\by)$,  $J_{\by}$ and $I_{\by}$.} 
\end{remark}

The key tool to proving Theorem \ref{thm:div} is to establish the following lemma, which shows that for a.e.~$\by$, the solution of~\eqref{eq:LSE} is locally an affine projection \emph{with linear and quadratic perturbations}.
\begin{lemma}\label{lem:div}
Suppose that $-\bd=A^\top\bu$ for some $\bu\geq\mathbf{0}$ whenever $\lambda=0$ in \eqref{eq:LSE}.  For any
  $\by \in \R^n$, let $(\hbt(\by), \hbxi(\by)) $ be any solution of~\eqref{eq:LSE} and let
  the index set $J_\by$ be as defined in \eqref{eq:J_y_2}. For a.e.~$\by \in \R^n$,
  \begin{eqnarray}\label{eq:eq_div_mul_cvx_gen}
    \hbt(\bz) = \widetilde{\bt}(\bz), \text{ for any }\bz\text{ in a neighborhood }U\text{ of }\by,
  \end{eqnarray}
  where $\widetilde{\bt}(\bz)$ is defined as the unique $\bt$-component of the optimal solution of the following optimization problem:
  \begin{eqnarray}\label{eq:aff_LSE}
    (\widetilde{\bt}(\bz),\widetilde{\bxi}(\bz))  & \in & \argmin_{\bt, \bxi}\frac{1}{2} \|\bt-\bz\|_2^2 +\bd^\top\bxi+ \frac{\lambda}{2}\|\bxi\|_2^2\\
      & &\;\;\mathrm{s.t.} \; A_{J_\by} \bxi + B_{J_\by} \bt = \bc_{J_\by}. \nonumber
  \end{eqnarray}
\end{lemma}

A rigorous proof of this lemma involves technical arguments from convex analysis, which will be presented in Section \ref{sec:proof_local_lemma} of the supplement.  
The proof of Theorem \ref{thm:div}, based on Lemma \ref{lem:div}, will be provided in Section \ref{sec:supp_theorem_div} of the supplementary file.


\section{DF of (Penalized) Convex Regression}
\label{sec:DFCvxReg}

One important application of Theorem \ref{thm:div} is in characterizing DF for the LSE in \emph{multivariate convex regression} (see e.g.,~\citet{SS11}). In particular, consider the nonparametric regression problem in~\eqref{eq:RegMdl2}  where $f:\R^d \to \R$ ($d > 1$) is a convex function and $\X := \{\bx_1,\ldots, \bx_n\}$ is the set of design points (with $n$ distinct elements) in $\R^d$. The goal is to estimate $\bt^* = (f(\bx_1),\ldots, f(\bx_n))$. Let $\mathcal{K}_{\text{conv}}$ be the set of all vector $\bt=(\theta_1, \ldots, \theta_n) \in \mathbb{R}^n$ for which there exists a convex function $\psi: \mathbb{R}^d \rightarrow \mathbb{R}$ such that $\psi(\bx_i)=\theta_i$ for $i=1,\ldots, n$. It can be shown that $\mathcal{K}_{\text{conv}}$ is a convex cone (see Lemma 2.3 of~\citet{SS11}). The multivariate convex LSE is defined as $\hat \bt(\by) :=\argmin_{\bt \in \mathcal{K}_{\text{conv}}} \frac{1}{2}\|\bt-\by\|_2^2$. In fact,  Lemma 2.2 from \citet{SS11} provides the following explicit characterization of $\mathcal{K}_{\text{conv}}$.
\begin{lemma}[\citet{SS11}]
\label{lem:dual_characterization}
For a vector $\bt \in \R^n$, we have $\bt \in \mathcal{K}_{\text{conv}}$ if and only if there exists a set of $n$ $d$-dimensional vectors $\bxi_1, \ldots, \bxi_n \in \R^d$ such that the following inequalities hold simultaneously:
\begin{eqnarray}\label{eq:mul_cvx_dual}
  \langle \bxi_j,  \bx_k- \bx_j \rangle \leq \theta_k - \theta_j, \;\;\;\mathrm{ for \;all } \;  j \ne k \in \{1, \ldots, n\}.
\end{eqnarray}
\end{lemma}
Lemma~\ref{lem:dual_characterization} is quite intuitive: since $f$ is a multivariate convex function, we have for any pair $\bx_k, \bx_j \in \X$,
\begin{equation}\label{eq:subgrad_convexity}
  f(\bx_k) - f(\bx_j) \geq \langle g(\bx_j), \bx_k - \bx_j \rangle,
\end{equation}
where $g(\bx_j) \in \partial f(\bx_j)$ is a subgradient of the convex function $f$ at $\bx_j$. Letting $\bxi_j = g(\bx_j)$, one can easily see the equivalence between \eqref{eq:subgrad_convexity} and   \eqref{eq:mul_cvx_dual}. Using Lemma \ref{lem:dual_characterization}, the LSE of multivariate convex regression can be formulated as the following optimization problem (see, e.g., \citet{K08}, \citet{SS11}, \citet{HD11} and \citet{LG12}):
{\small \begin{align}\label{eq:multi_cvx}
  (\hat \bt(\by), \hat \bxi(\by))  & =   \argmin_{\substack{\bt \in \R^n \\ \bxi=[\bxi_1^\top, \ldots, \bxi_n^\top]^\top \in \R^{nd}}} \frac{1}{2} \|\bt-\by\|_2^2  \\
                           &   \qquad \text{s.t.}  \quad \langle \bxi_j, \bx_k- \bx_j \rangle \leq \theta_k - \theta_j, \;\; \forall \;  j \ne k \in \{1, \ldots, n\}, \nonumber
\end{align}}
which is a standard linearly constrained quadratic program and can be solved by many off-the-shelf solvers (e.g., SDPT3 \citep{SDPT3}). Next we show that the above optimization problem can be reformulated as a special case of \eqref{eq:LSE} with properly chosen $A$, $B$ and $\bc=\mathbf{0}$, $\bd=\mathbf{0}$~and~$\lambda=0$.
\begin{proposition}\label{prop:multi_cvx_reform}
The optimization problem for multivariate convex  regression in \eqref{eq:multi_cvx} can be formulated as \eqref{eq:LSE_mul_cvx} with $p=nd$ and $\bxi=[\bxi_1^\top, \ldots, \bxi_n^\top]^\top \in \R^{nd}$. In this scenario, $A$ in \eqref{eq:LSE_mul_cvx} is a $[n(n-1)]\times nd$ matrix and each row of $A$ is indexed by a pair $r=(j,k)$ with $j \ne k \in \{1, \ldots, n\}$ and each column is indexed by a pair $c=(j',s)$ with $j'\in \{1, \ldots, n\}$ and $s\in \{1, \ldots, d\}$. Moreover, we partition $A$ into $[n(n-1)]\times n$ blocks with each block of size $1\times d$. 
Let $A_{r,j'}$ be the block of $A$ with row $r=(j,k)$ and column $j' \in \{1,\ldots, n\}$. $A_{r,j'}$ is defined as $A_{r,j'} = \bx_k^\top-\bx_j^\top$ if $j=j'$ and $A_{r,j'} = \mathbf{0}^\top$ if $j\neq j'$.
The corresponding $B$ is a $[n(n-1)]\times n$ matrix and each row of $B$ is indexed by a pair $r=(j,k)$ with $j \ne k \in \{1, \ldots, n\}$ and each column is indexed by $c\in \{1, \ldots, n\}$. Let $B_{r,c}$ be the entry in row $r =(j,k)$ and column $c$ of the matrix $B$ defined as $B_{r,c} = 1$ if $c=j$, $B_{r,c} = -1$ if $c=k$, and $B_{r,c} =
0$ otherwise.
The corresponding $\bc$ will be an all-zero vector in $\mathbb{R}^{n(n-1)}$.
\end{proposition}
The proof of Proposition \ref{prop:multi_cvx_reform} is  straightforward and thus omitted. Given the matrices
$A$ and $B$ defined in Proposition \ref{prop:multi_cvx_reform}, one can define the corresponding polyhedron $\Q$ of $(\bxi, \bt)$  in \eqref{eq:Q} and it is clear that $\mathcal{K}_{\text{conv}}=\Proj_{\bt}(\Q)$, which is a projected convex polyhedron. Given Proposition \ref{prop:multi_cvx_reform}, it is straightforward to apply Theorem~\ref{thm:div} (with $\bd=\mathbf{0}$ and $\lambda=0$) to calculate the DF of the LSE for multivariate convex regression.
\begin{corollary}\label{corr:multi_cvx_reform}
For multivariate convex LSE in \eqref{eq:multi_cvx}, let the set of tight constraints be $J_{\by}:=\{(j,k): \langle \hbxi_j, \bx_k- \bx_j \rangle =\widehat{\theta}_k - \widehat{\theta}_j\}$. Let $I_{\by} \subseteq J_\by$ be the index set of maximal independent rows of the matrix $[A_{J_{\by}}, B_{J_{\by}}]$, where $A$ and $B$ are defined in Proposition \ref{prop:multi_cvx_reform}. Then for a.e. $\by$, we have
$
  D(\by)= n- |I_\by| + \mathrm{rank}(A_{I_\by})$  and  $\mathrm{df} (\hbt(\by))= n-  \E[|I_\by|] + \E\left[\mathrm{rank}(A_{I_\by})\right]$.
\end{corollary}

The multivariate convex LSE described in~\eqref{eq:multi_cvx} tends to overfit the data, especially near the boundary of the convex hull of the design points --- the subgradients take large values near the boundary. Thus, we might want to regularize the convex LSE.  A natural way to achieve this is to impose bounds on the norm of the subgradients; see e.g.,~\cite{SM13},~\cite{Lim14}. In the penalized form this would lead to the following problem:
\begin{eqnarray}\label{eq:LSE_Lip}
  (\hbt(\by), \hbxi(\by)) & = &  \argmin_{\substack{\bt \in \R^n \\ \bxi=[\bxi_1^\top, \ldots, \bxi_n^\top]^\top \in \R^{nd}}} \frac{1}{2} \|\bt-\by\|_2^2 + \frac{\lambda}{2}\sum_{j=1}^n \|\bxi_j\|_2^2\\
            & & \;\; \text{s.t.} \;    \langle \bxi_j,  \bx_k- \bx_j \rangle \leq \theta_k - \theta_j \quad \forall \; j \neq k, \nonumber
\end{eqnarray}
which can be formulated as~\eqref{eq:LSE_Lip_multi_conv} with $p=nd$ and $\bxi=[\bxi_1^\top, \ldots, \bxi_n^\top]^\top \in \R^{nd}$, where $A$, $B$ and $\bc$ are defined in Proposition \ref{prop:multi_cvx_reform}.
The divergence of the penalized convex regression estimator $\hbt(\by)$ in \eqref{eq:LSE_Lip} can be easily  characterized by Theorem \ref{thm:div} (with $\bd=\mathbf{0}$ and $\lambda>0$).
\begin{corollary}
\label{thm:div_mul_cvx_LSE}
  For the penalized multivariate convex LSE described in~\eqref{eq:LSE_Lip}, let the set of tight constraints be $J_{\by}:=\{(j,k): \langle \hbxi_j, \bx_k- \bx_j \rangle =\widehat{\theta}_k - \widehat{\theta}_j\}$. Let $I_{\by} \subseteq J_\by$ be the index set of maximal independent rows of the matrix $[A_{J_{\by}}, B_{J_{\by}}]$, where $A$ and $B$ are defined in Proposition \ref{prop:multi_cvx_reform}. Then for a.e. $\by$, we have
$
D(\by)=n - \mathrm{trace}\left(B_{I_\by}^\top \left(B_{I_\by}B_{I_\by}^\top+ \frac{1}{\lambda} A_{I_\by} A_{I_\by}^\top\right)^{-1}B_{I_\by} \right)$ and $\mathrm{df} (\hbt(\by))= \E [D(\by)]$.
\end{corollary}

\section{DF of (Bounded) Isotonic Regression}\label{sec:DFIso}

Let us consider isotonic regression on a general partially ordered set;  see e.g.,~\citet[Chapter 1]{RWD88}.  Let $\X := \{x_1,\ldots, x_n\}$ be a set (with $n$ distinct elements) in a metric space with a {\it partial order}, i.e., there exists a binary relation $\lesssim$ over $\X$ that is reflexive ($x \lesssim x$ for all $x \in \X$), transitive ($u, v, w \in \X, \; u \lesssim v$ and $v \lesssim w$ imply $u \lesssim w$), and antisymmetric ($u , v  \in \X, \; u \lesssim v$ and $v \lesssim u$ imply $u  = v$). Consider~\eqref{eq:RegMdl2} where now the real-valued function $f$ is assumed to be {\it isotonic} with respect to the partial order $\lesssim$, i.e., any pair $u,v  \in \X$, $ u \lesssim v$ implies $f(u) \le f(v)$. This model can be expressed in the sequence form as~\eqref{eq:SeqMdl} by letting $\theta^*_i=f(x_i)$ for $i=1, \ldots, n$. To construct the LSE in this problem, we add \emph{isotonic} constraints on $\bt$, which are of the form $\theta_i \le \theta_j$ if $x_i \lesssim x_j$, for some $i,j \in \{1,\ldots, n\}$. As a special case, let us consider $\X \subset \mathbb{R}$ for the univariate isotonic regression. Assuming without loss of generality that  $x_1 \leq x_2 \leq \cdots \leq x_n$,  the isotonic constraint set on $\bt$ takes the form of the isotonic cone $\mathcal{M}$ (see~\eqref{eq:IsoReg}) and the LSE  is the projection $\hbt(\by)$ of $\by$ onto $\mathcal{M}$. For the ease of illustration, the isotonic constraints can be represented by an acyclic directed  graph $\widetilde{G}=(V, \widetilde{E})$  where $V=\{1, \ldots, n\}$ (corresponding to $\{\theta_i\}_{i=1}^n$) and the set of the directed edges is denoted by
\begin{equation}\label{eq:tilde_E}
\widetilde{E}=\{(i,j):  x_i \lesssim x_j\}.
\end{equation}
For the univariate isotonic cone $\mathcal{M}$, the edge set $\tilde{E}$ contains $n-1$ edges, where the $i$-th edge runs from node $\theta_i$ to $\theta_{i+1}$ for $i=1, \ldots, n-1$, i.e., $\tilde{E}=\{(i,{i+1}): i=1, \ldots, n-1\}$.

It is well-known that the projection $\hbt(\by)$ of $\by$ onto the isotonic constraint set suffers from the {\it spiking effect}, i.e., over-fitting near the boundary of the convex hull of the predictor(s) (see~\citet{Pal08} and~\citet{WS93}). However such monotonic relationships among variables arise naturally in many applications and this has lead to a recent surge of interest in regularized isotonic regression; see e.g.,~\citet{LRS12}, \citet{LR14}, and \citet{WMO15}. Probably the most natural form of regularization involves constraining the range of $\hbt(\by)$, i.e., $\max_i \hat{\theta}_i - \min_i \hat{\theta}_i$; this leads to {\it bounded isotonic regression}. More specifically, when the range of $f$ is known to be bounded (from above) by some $\gamma \geq 0$, we can impose this boundedness restriction of $f$ by adding the boundedness constraints and the corresponding bounded isotonic LSE can be defined as follows.

\begin{definition}\label{def:bounded_iso}
The bounded isotonic LSE (with boundedness parameter $\gamma$) is defined as the projection estimator $\hat \bt_{\gamma}(\by):=\argmin_{\btheta \in \C} \| \by - \btheta\|_2^2$,
where the constraint set is
{\begin{eqnarray}\label{eq:partial_iso_bounded}
\C :=  \Big\{\btheta \in \R^n: \theta_i \leq \theta_j \; \forall \, (i,j)\in \widetilde{E}, \; \theta_i\leq \theta_j+\gamma, i\in \max(V),j\in \min(V),i\neq j \Big\}.
\end{eqnarray}}
Here, $\max(V)$ and $\min(V)$ are the maximal and minimal sets of $V$ with respect to this partial order:
\[
\max(V)=\{i\in V: \widetilde{n}^+(i)=\emptyset\}  \quad \mbox{and} \quad \min(V)=\{i\in V: \widetilde{n}^-(i)=\emptyset\},
\]
where for any node $i$, $\widetilde{n}^+(i):=\{j \in V: (i,j)\in \widetilde{E} \}$ is the set of elements that are ``greater than $i$" with respect to the partial order (i.e., successors of $i$), and $\widetilde{n}^-(i):=\{j \in V: (j,i)\in \widetilde{E} \}$  is the set of elements that are ``smaller than $i$" (i.e., predecessors of $i$).
\end{definition}

In Definition~\ref{def:bounded_iso},  both $\max(V)$ and $\min(V)$ must be nonempty for any nonempty partially ordered set. This is because  $\widetilde{G}$ is an acyclic directed graph where there always exist nodes with no successor and nodes with no predecessor. We also note that $\max(V)$ and $\min(V)$ might overlap, for example, when there exist nodes that cannot be compared with any other nodes under the given partial order. For each $i\in \max(V)$ and $j\in \min(V)$ with $i\neq j$, we add a constraint $\theta_i\leq \theta_j+\gamma$ to impose the boundedness restriction on the range of $f$.

Similar to the unbounded case, we can represent the constraints in \eqref{eq:partial_iso_bounded} by a graph $G=(V, E)$ where $V=\{1,\ldots, n\}$ and
\[
E :=\widetilde{E} \cup \{(i,j): i \in \max(V), j \in \min(V), i \neq j\}.
\]
As a special case, for univariate bounded isotonic regression, the constraint set $\C$ in \eqref{eq:partial_iso_bounded} becomes
$\{\btheta \in \R^n: \theta_1 \le \cdots \le \theta_n,  \;\theta_n - \theta_1 \le \gamma \}$ and the corresponding edge set is  $E=\{(i,{i+1}), i=1, \ldots, n-1\} \cup \{(n,1)\}$.

To compute the DF of bounded isotonic LSE $\hat \bt_{\gamma}(\by)$,  first notice that the set $\C$ can be easily represented as a convex polyhedron of the form in \eqref{eq:CvxPoly}. We note that as compared to unbounded isotonic regression, the $\C$ in \eqref{eq:partial_iso_bounded} is a convex polyhedron rather than a polyhedral cone due to the additional boundedness constraints. Given the fact that bounded isotonic LSE is a projection estimator onto a convex polyhedron, Theorem \ref{thm:div} (with $\bd=\mathbf{0}$, $\lambda=0$ and $A=0$) can be used to compute its DF. Instead of directly applying Theorem \ref{thm:div} in its original form, we draw some interesting connections to graph theory, which also leads to a faster computation of the divergence. In particular,  let $\omega(G)$  denote the number of connected components of the undirected version of the graph $G=(V, E)$ (removing the directions of edges in $G$), i.e., the number of maximal connected subgraphs of $G$. The divergence of $\hat \bt_{\gamma}(\by)$ can be characterized using the number of connected components of a subgraph of $G$ as shown in the following proposition (see the proof in Section \ref{sec:proof_uni_bound_iso} in the supplement).

\begin{proposition}\label{prop:uni_bound_iso}
The bounded isotonic constraint set $\C$ defined in~\eqref{eq:partial_iso_bounded} is a convex polyhedron in the form of \eqref{eq:CvxPoly} where $m=|E|$ and $B\in\mathbb{R}^{|E|\times n}$ is defined as (the rows of $B$ are indexed by the edge set) $\vspace{-0.05in}$
\begin{eqnarray}\label{eq:A_iso_bounded_partial_order}
B_{e,i}=\left\{\begin{array}{ll}
1&\text{ if }e=(i,j)\in E\text{ for some } j\neq i\\
-1&\text{ if }e=(j,i)\in E\text{ for some } j\neq i\\
0&\text{ otherwise } \vspace{-0.05in}
\end{array}
\right.
\end{eqnarray}
and $\bc = (c_e)_{e=1}^{|E|} \in\mathbb{R}^{|E|}$ is defined as
\begin{eqnarray}\label{eq:b_iso_bounded_partial_order}
c_e=\left\{\begin{array}{ll}
\gamma&\text{ if }e=(i,j)\in E\text{ for } i \in \max(V), j \in \min(V)\\
0&\text{ otherwise}.
\end{array}
\right.
\end{eqnarray}
Let $B_e$ be the $e$-th row of $B$ and $J_\by:=\{e\in E: B_e \hat \bt_\gamma(\by) = c_e\}$. Further,  let $G_{J_\by}$  be the subgraph of $G$ with the edge set $J_\by$. The divergence of $\hat \bt_\gamma (\by)$ is  the number of connected components of $G_{J_\by}$ for a.e. $\by$, i.e.,
  $D(\by)=\omega(G_{J_\by})$, and therefore $\df(\hat \bt_\gamma(\by))= \E[\omega(G_{J_\by})]$.
\end{proposition}

%

\begin{figure}[!t]
\centering%
\subfigure[t][Matrix $B$]{$%
\begin{matrix}
    B=\begin{pmatrix}%
        +1      & -1       & 0         & 0         & 0    \\
        0       & +1         & -1       & 0         & 0    \\
        0       & 0       & +1         & -1      & 0    \\
        0       & 0         & 0       & +1         & -1   \\
        -1  & 0    & 0    & 0    & +1   \\
    \end{pmatrix}\\
   ~
\end{matrix}$ \label{fig:A_matrix}
}\hspace{8mm}
\subfigure[t][Graph $G$]{
  \includegraphics[width=0.38\textwidth]{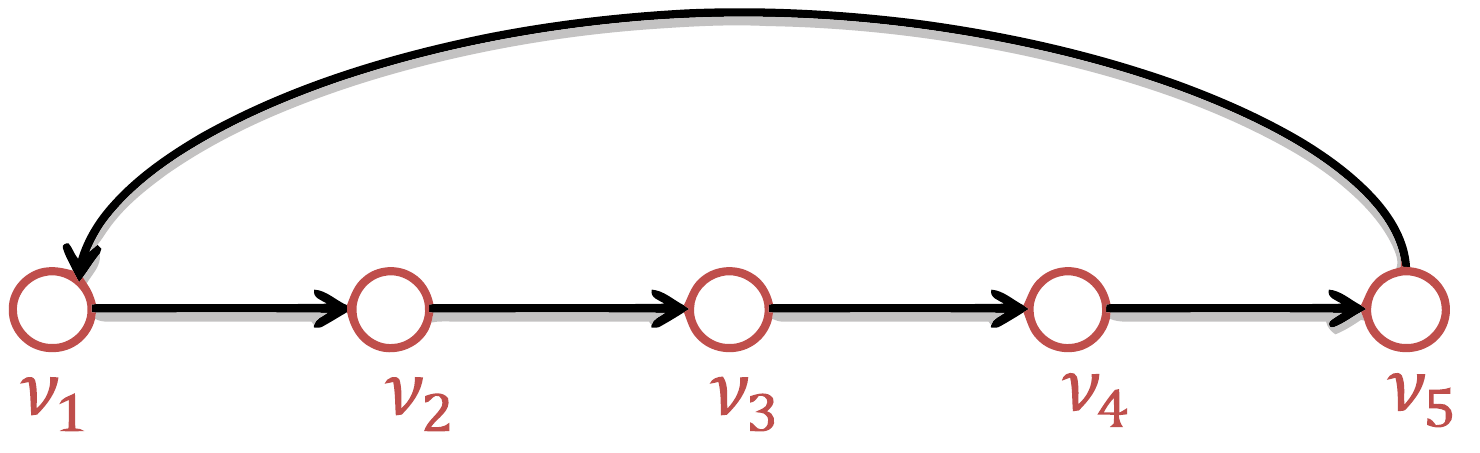}
	    \label{fig:G}
}\hspace{-3mm}
\caption{The matrix $B$ and the induced graph $G$.}
\label{fig:incidence_graph}
\end{figure}

\begin{figure}[!t]
\centering
\subfigure[t][Matrix $B_{J_\by}$]{$
\begin{matrix}
    B_{J_\by}=\begin{pmatrix}%
        +1      & -1       & 0         & 0         & 0    \\
        0       & 0       & +1         & -1      & 0    \\
        -1  & 0    & 0    & 0    & +1   \\
    \end{pmatrix}\\
   ~
\end{matrix}$
	    \label{fig:sub_A_matrix}
}\hspace{8mm}
\subfigure[t][Graph $G_{J_\by}$]{
  \includegraphics[width=0.38\textwidth]{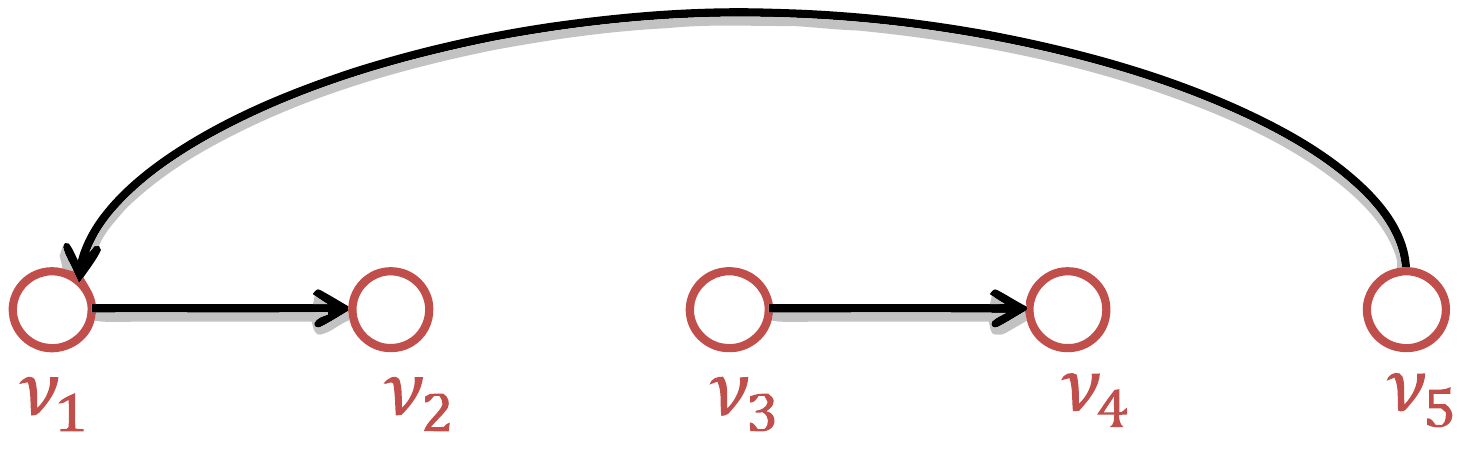}
	    \label{fig:sub_G}
}\hspace{-3mm}
\caption{The matrix $B_{J_\by}$ and the induced graph $G_{J_\by}$.}
\label{fig:sub_incidence_graph}
\end{figure}

The characterization of divergence in Proposition \ref{prop:uni_bound_iso} not only has interesting connections to graph theory but also leads to a computationally fast procedure to compute the divergence. In fact, it is easy to compute  $\omega(G_{J_\by})$ using either breadth-first or depth-first search in linear time in $n$, which is \emph{computationally much cheaper} than directly calculating the rank of $B_{J_\by}$ in Proposition \ref{prop:div_poly}. 
To facilitate the understanding of Proposition \ref{prop:uni_bound_iso},  we provide a toy example. Consider the following bounded isotonic constraint set with $n=5$:
\begin{equation}
\C =  \{\btheta \in \R^n: \theta_1 \le \ldots \le \theta_n, \;\mbox{and} \;\theta_n - \theta_1 \le \gamma \}.
\end{equation}
 The set $\C$ can be represented as $\C=\{\btheta \in \R^n: B \btheta \leq \bc\}$ where $B$ is shown in Figure \ref{fig:A_matrix} and $\bc$ only has one non-zero element at the $n$'th position, i.e., $c_n=\gamma$. The graph $G$ induced from $B$, which has only one connected component (i.e., $\omega(G)=1$), is shown in Figure~\ref{fig:G}.

Now suppose that we have $\hat \theta_{\gamma, 1}=\hat \theta_{\gamma, 2} < \hat \theta_{\gamma, 3}=\hat \theta_{\gamma, 4}<\hat \theta_{\gamma, 5}$ and  $\hat \theta_{\gamma, 5}=\hat \theta_{\gamma, 1} + \gamma$. Then  $J_\by=\{1,3,5\}$ and the corresponding $B_{J_\by}$ and $G_{J_\by}$ are presented in Figure \ref{fig:sub_incidence_graph}. From Figure \ref{fig:sub_incidence_graph}, $G_{J_\by}$ has 2 connected components $\{\theta_1, \theta_2, \theta_5\}$ and $\{\theta_3, \theta_4\}$ and thus $D(\by)=\omega(G_{J_\by})=2$.  It is of interest to compare this with the univariate unbounded isotonic regression example where the divergence of $\hat \bt_{\gamma}(\by)$ would be 3 (i.e., the number of distinct values of $\hat \theta_i$'s;  see Proposition 1 from \cite{MW00}) instead of 2.


Using exactly the same proof technique as that of Proposition \ref{prop:uni_bound_iso},  we can easily derive the following result for the DF of \emph{unbounded isotonic regression} on a partially ordered set. In particular,
recall the unbounded isotonic cone $\mathcal{M}=\{\bt \in \mathbb{R}^n : \theta_i \leq \theta_j, \forall (i,j) \in \widetilde{E}\}$ where $\widetilde{E}$ is defined in \eqref{eq:tilde_E} and the corresponding LSE $\hat{\bt}(\by)=\argmin_{\bt \in \mathcal{M}} \|\bt-\by\|_2^2$. The cone  $\mathcal{M}$ can be represented as $\mathcal{M}=\{\bt \in \mathbb{R}^n: B\bt \leq \mathbf{0}\}$, where $B \in \mathbb{R}^{|\widetilde{E}| \times n}$ is defined similarly as in \eqref{eq:A_iso_bounded_partial_order} (replacing $E$ in \eqref{eq:A_iso_bounded_partial_order} by $\widetilde{E}$).
Let $B_e$ be the $e$-th row of $B$, $J_\by:=\{e\in \widetilde{E}: B_e \hat \bt (\by) = b_e\}$ and  $\widetilde{G}_{J_\by}$  be the subgraph of $\widetilde{G}$ with the edge set $J_\by$. The divergence of $\hat \bt (\by)$ for unbounded isotonic regression is $D(\by)=\omega(\widetilde{G}_{J_\by})$, and therefore $\df(\hat \bt(\by))= \E[\omega(\widetilde{G}_{J_\by})]$.

In addition to characterizing the DF for general bounded isotonic regression, we also show a useful property of the divergence $D_\gamma(\by)$ in Theorem \ref{thm:iso_bounded_monotone} (where we make the dependence on the model complexity parameter $\gamma$ explicit). In particular, we prove that the divergence $D_\gamma(\by)$  (and thus the DF) is nondecreasing in $\gamma$. To show this we first present an important connection between the solution of bounded isotonic regression and that of unbounded isotonic regression (which can be viewed as a special case of bounded isotonic regression with $\gamma = + \infty$).  This result is of independent interest by itself.

We start with some notation. It is well known that the LSE for \emph{unbounded} isotonic regression $\hat\bt$ has a group-constant structure (here $\by$ is suppressed for notational simplicity). That is, there exists a partition $U_1,U_2,\dots,U_r$ of $V=\{1,\ldots, n\}$ (i.e., $U_s$'s are disjoint and $V=\bigcup_{s=1}^r U_s$) such that $\hat{\theta}_i=\bar\theta_s$ for some value $\bar\theta_s$ for each $i\in U_s$, for $1\leq s\leq r$. Moreover, without loss of generality, we assume that $\bar\theta_1<\bar\theta_2<\dots<\bar\theta_r$. Let $\hat \bt_{\gamma}$ be the LSE for bounded isotonic regression with the boundedness parameter $\gamma$. The next proposition shows that $\hat \bt_{\gamma}$  can be obtained by appropriately thresholding $\hat\bt$.

\begin{proposition}\label{prop:threshold_gen}
Let $|U_s|=k_s$ for $s=1,\dots,r$ and $H(L,\gamma)$ be a function on $\mathbb{R}^2$ defined as
\begin{eqnarray}
\label{eq:HL_gen}
H(L,\gamma) :=\sum_{s=1}^rk_s\left(L-\bar\theta_{ s}\right)_+ +
\sum_{s=1}^rk_s\left(L+\gamma-\bar\theta_{ s}\right)_-,
\end{eqnarray}
where $(x)_+=\max\{x,0\}$ and $(x)_-=\min\{x,0\}$. For any given $\gamma$ with $\bar\theta_{ r}-\bar\theta_{ 1}\geq\gamma\geq0$,  $H(L,\gamma)$ is a continuous and strictly increasing function of $L$. Moreover,  $\lim_{L\rightarrow-\infty}H(L,\gamma)=-\infty$ and $\lim_{L\rightarrow+\infty}H(L,\gamma)=+\infty$ so that there exists a unique $L_{\gamma}$ satisfying $H(L_{\gamma},\gamma)=0$. Then, we have
\begin{equation}\label{eq:thresh}
\hat{\theta}_{\gamma, i}=\max(L_{\gamma},\min(L_{\gamma}+\gamma,\bar\theta_{ s})),\text{ for all }i\in U_s.
\end{equation}
Moreover, $L_{\gamma}$ is nonincreasing in $\gamma$.
\end{proposition}

Proposition \ref{prop:threshold_gen} also provides an efficient way to compute the LSE for bounded isotonic regression. In particular, one can first compute $\hat\bt$ by solving the corresponding unbounded isotonic regression, which can be efficiently computed by using existing off-the-shelf solvers (e.g., SDPT3 \citep{SDPT3}). Given $\hat\bt$, one obtains the values of $\bar\theta_{ s}$  and $k_s$ for $s=1,\dots,r$, which are necessary for constructing the function in \eqref{eq:HL_gen}. If $\gamma>\bar\theta_{ r}-\bar\theta_{ 1}$, the boundedness constraint will be non-effective and $\hat \bt_\gamma = \hat \bt$. On the other hand, if $\bar\theta_{ r}-\bar\theta_{ 1}\geq\gamma\geq0$, since $H(L,\gamma)$ is a continuous and strictly increasing function of $L$, one can use \emph{bisection search} to compute $L_{\gamma}$ such that $H(L_{\gamma},\gamma)=0$. Then by \eqref{eq:thresh}, we threshold $\hat \bt$ to obtain $\hat \bt_\gamma $: for each $U_s$, if $\bar{\theta}_s < L_\gamma$, $\hat{\theta}_{\gamma, i} = L_\gamma$ for all $i \in U_s$; if $\bar{\theta}_s > L_\gamma+\gamma$, $\hat{\theta}_{\gamma, i} = L_\gamma + \gamma $ for all $i \in U_s$; otherwise $\hat{\theta}_{\gamma, i}$ is set to $\bar{\theta}_s$ for all  $i \in U_s$.

The key to the proof of the above result is to find appropriate values of dual variables such that the primal solutions in \eqref{eq:thresh} and dual solutions together satisfy the KKT condition of $\min_{\btheta \in \C} \| \by - \btheta\|_2^2$ with $\C$ in \eqref{eq:partial_iso_bounded}. We achieve this by designing a \emph{transportation problem}, which is a classical problem in operations research (see, e.g., Chapter 14 in \citet{Dantzig:59}). The dual solutions are constructed based on the solution of such a transportation problem. Please refer to the proof in Section \ref{sec:proof_DFIso} in the supplementary material for details.

Combining Proposition \ref{prop:threshold_gen} and Proposition \ref{prop:uni_bound_iso},  we obtain the following theorem which shows the monotonicity of DF in terms of the boundedness parameter $\gamma$ in bounded isotonic regression (see Section \ref{sec:supp_monotone} in the supplementary material for the proof).
\begin{theorem}\label{thm:iso_bounded_monotone}
For any given $\by \in \R^n$ the divergence of $\hat \bt_{\gamma} (\by)$ is nondecreasing in $\gamma$. This implies that $\df(\hat\bt_\gamma(\by))$ is nondecreasing in $\gamma$.
\end{theorem}


\section{Additive TV Regression and Other Applications}
\label{sec:other_app}
In this section we apply our main result to derive the DF for additive TV regression (see Example 3 in the Introduction) and $\ell_\infty$-regularized group Lasso. Moreover, our main result (Theorem \ref{thm:div}) also yields, as special cases, known results on DF of many popular estimators, e.g., \emph{Lasso and generalized Lasso}, \emph{linear regression}, and \emph{ridge regression}. Due to space constraints, we illustrate these applications in Section \ref{sec:recover_other} of the supplementary file; the proofs of the results in this section are also provided in Section~\ref{sec:proof_other_app}.

\subsection{Additive Generalized TV Regression}\label{sec:AddMdl}
For each response $y_i$ and input $\mathbf{x}_i=(x_{i1}, \ldots, x_{id})$, where $1 \leq i \leq n$,  the additive model assumes that $\E(y_i | \mathbf{x}_i) = \sum_{j=1}^d f_j(x_{ij})$. Let $\theta_{ji}^*=f_j(x_{ij})$ and $\bt_j^*=(\theta_{j1}, \ldots, \theta_{jn})$, where it is typically assumed that each $\bt_j$ has zero mean (i.e., $\mathbf{1}^\top  \bt_j=0$). \cite{Petersen:16} proposed the following additive TV regularizer. Let $D \in \R^{(n-1) \times n}$ be the discrete  first derivative matrix (i.e., the $i$-th row of $D$ only contains two non-zero elements: $D_{i,i}=1$ and $D_{i,i+1}=-1$) and $P_j \in \R^{n \times n}$  be the permutation matrix that orders the $j$-th feature from least to greatest. The estimation of $\{\bt_j^*\}_{j=1}^d$ in an additive TV regularized regression takes the form$\vspace{-0.08in}$:
\begin{eqnarray*}
\{\hat{\theta}_0, \{\hat{\bt}_j\}_{j=1}^d \}= & \argmin_{\{\bt_j\}_{j=1}^d} & \frac{1}{2} \Big\|\by -\sum_{j=1}^d \bt_j - \theta_0 \mathbf{1} \Big\|_2^2  + \tau \sum_{j=1}^d \|DP_j \bt_j\|_1 \\
                         & & \text{s.t.} \qquad \mathbf{1}^\top  \bt_j=0, \quad 1 \leq j \leq d. \vspace{-0.08in}
\end{eqnarray*}
The penalty $\|DP_j \bt_j\|_1$ encourages $\bt_j$ to be piecewise constant with a small number of jumps, depending on the regularization $\tau$. 
In fact, instead of using the discrete first derivative matrix $D$, we could impose a higher order smoothness for each component function $f_j$. More precisely, one can use a higher order discrete difference matrix $D_j$ for each  $f_j$; in the sequel we will consider this more general setup. For example, the second order differencing matrix produces piecewise affine fits, with a few number of kink points. The specific form of  higher order discrete difference matrix  is given in Eq.~(41) of~\cite{T14}. Let us denote $D_j P_j$ by $Q_j \in \R^{n_j \times n}$ for notational simplicity, and we consider the following \emph{additive generalized TV regression}: $\vspace{-0.08in}$
\begin{eqnarray}\label{eq:fused_add}
\{\hat{\theta}_0, \{\hat{\bt}_j\}_{j=1}^d \} = & \argmin_{\{\bt_j\}_{j=1}^d} & \frac{1}{2} \Big\|\by -\sum_{j=1}^d \bt_j - \theta_0 \mathbf{1} \Big\|_2^2  + \tau \sum_{j=1}^d \|Q_j \bt_j\|_1 \\
                         & & \text{s.t.} \qquad \mathbf{1}^\top  \bt_j=0, \quad 1 \leq j \leq d. \vspace{-0.08in}\nonumber
\end{eqnarray}

Let the $\hbt(\by):=\sum_{j=1}^d \hbt_j(\by) + \hat\theta_0(\by) \mathbf{1}$ be the estimated function values at the design points. To characterize its divergence, we rewrite the optimization problem in \eqref{eq:fused_add} as
{\small \begin{eqnarray}\label{eq:lasso_gen_new1}
  (\hbt(\by), \{\hbt_j(\by)\}_{j=1}^d,\hat\theta_0(\by),\{\hat\bg_j(\by)\}_{j=1}^d) & \in &\argmin_{\bt, \bt_j,\theta_0,\bg_j } \frac{1}{2}\|\bt-\by\|_2^2+\sum_{j=1}^d\tau\mathbf{1}^\top\bg_j \\
                          & & \;\;\;  \text{s.t.}  \;\; \bt-\sum_{j=1}^d \bt_j - \theta_0 \mathbf{1} \leq \mathbf{0}, \;\; -\bt+\sum_{j=1}^d \bt_j + \theta_0 \mathbf{1}  \leq \mathbf{0} \nonumber\\
                          & & \;\;\;  \;\;\;\;   \;\; \;Q_j\bt_j-\bg_j  \leq \mathbf{0}, \;\; -Q_j\bt_j-\bg_j \leq \mathbf{0}\nonumber\\
                          & & \;\;\;  \;\;\;\;   \;\; \;\mathbf{1}^\top  \bt_j\leq0, \;\; -\mathbf{1}^\top  \bt_j\leq0,\quad 1 \leq j \leq d. \nonumber
\end{eqnarray}}
With some algebraic manipulations, we show that the optimization in \eqref{eq:lasso_gen_new1} is a special case of \eqref{eq:LSE} with a linear perturbation term $\mathbf{d}^\top\bxi$ and $\lambda=0$ (in particular, in the form of \eqref{eq:linear_per}); see the proof in the supplementary file for the details. We then apply Theorem~\ref{thm:div} to obtain the following result on the DF for $\hbt(\by)$. In our proof, we also verify that the condition in Theorem~\ref{thm:div} (i.e., $-\mathbf{d}= A^\top \mathbf{u}$ for some $\mathbf{u} \geq \mathbf{0}$) indeed holds.

\begin{proposition}\label{prop:fused_add}
For the estimator $\hbt(\by)=\sum_{j=1}^d \hbt_j(\by) + \hat\theta_0(\by) \mathbf{1}$ in  \eqref{eq:fused_add}, the divergence of $\hbt(\by)$ is,
\[
D(\by)=\mathrm{dim}(\mathrm{span}\{\mathbf{1}_{n\times1},\mathrm{ker}(K_1),\dots,\mathrm{ker}(K_d)\}),
\]
where, for $j=1,\dots,d$,  $K_j=\begin{pmatrix}Q_0^j\\\mathbf{1}_{1\times n}\end{pmatrix}$,
$Q_0^j$ is the sub-matrix of $Q_j$ consisting of rows $\mathbf{q}_{ji}$ ($1\leq i\leq n_j$) of $Q_j$ such that $\mathbf{q}_{ji}^\top\hbt_j(\by)=0$ and
$\mathrm{ker}(K_j) :=\{\bx\in\mathbb{R}^n : Q_0^j\bx=\mathbf{0}\text{ and }\mathbf{1}_{1\times n}\bx=0\}$ is the kernel of $K_j=\begin{pmatrix}Q_0^j\\\mathbf{1}_{1\times n}\end{pmatrix}$. Further, $\mathrm{df} (\hbt(\by))=\E(D(\by))$.
\end{proposition}

\begin{remark}\label{rem:compute_NullSpace}
\emph{
For each $j$, the matrix $K_j$ can be easily constructed by checking if $\mathbf{q}_{ji}^\top\hbt_j(\by)=0$ for $1\leq i\leq n_j$. After obtaining $K_j$, the basis for the null space $\mathrm{ker}(K_j)$ can be easily computed by transforming $K_j$ into the reduced row echelon form using Gaussian elimination (note that one can use the \emph{null} function in Matlab or the \emph{Null} function in R to compute the basis of $\mathrm{ker}(K_j)$). Then, we construct a matrix using the basis of $\mathrm{ker}(K_j)$ for each $j$ and $\mathbf{1}_{n\times1}$ as its column so that $D(\by)$ can be computed as the rank of this matrix.
}
\end{remark}

\subsection{$\ell_\infty$-regularized Group Lasso}\label{sec:grouplasso}
 Let $\mathbb{G}=\{\mathcal{G}_1,\mathcal{G}_2,\dots,\mathcal{G}_l\}$ be a partition of $\{1,2,\dots,d\}$. Each element $\mathcal{G}\in\mathbb{G}$ represents a group of variables. 
The $\ell_\infty$-regularized group Lasso estimator can be formulated as the following optimization problem \citep{Zhao:09Group,Negaban:2011}:
\begin{equation}\label{eq:grouplasso_gen}
\hbbt(\by)\in\argmin_{\bbt\in\mathbb{R}^d}\frac{1}{2}\|\by-X\bbt\|_2^2+\tau\sum_{\mathcal{G}\in \mathbb{G}}\|\bbt_{\mathcal{G}}\|_\infty,
\end{equation}
where $\bbt_{\mathcal{G}}$ is the sub-vector of $\bbt$ consisting of the coordinates indexed by the elements in $\mathcal{G}$. We can easily see that \eqref{eq:grouplasso_gen} is a special case of the optimization problem~\eqref{eq:LSE}. In fact, by introducing the variable $\bg\in\mathbb{R}^{l}$ and letting $\bt=X\bbt$, \eqref{eq:grouplasso_gen} can be equivalently reformulated as
{\small \begin{eqnarray}\label{eq:grouplasso_gen_new}
  (\hbt(\by), \hat\bbt(\by),\hat\bg(\by)) & \in &\argmin_{\bt, \bbt,\bg } \frac{1}{2}\|\bt-\by\|_2^2+\tau\mathbf{1}^\top\bg \\
                          & & \;\;\;  \text{s.t.}  \;\; X\bbt - \bt \leq \mathbf{0}, \;\; -X\bbt + \bt \leq \mathbf{0} \nonumber\\
                          & & \;\;\;  \;\;\;\;   \;\; \;\bbt_{\mathcal{G}_j}-\gamma_j\mathbf{1}_{|\mathcal{G}_j|}  \leq \mathbf{0}, \;\; -\bbt_{\mathcal{G}_j}-\gamma_j\mathbf{1}_{|\mathcal{G}_j|} \leq \mathbf{0}.\nonumber
\end{eqnarray}}
By setting $\bxi=(\bbt^\top,\bg^\top)^\top$ and defining $E$ as the $d\times l$ matrix with $E_{ij}=1$ if $i\in\mathcal{G}_j$ and $E_{ij}=0$ otherwise,~\eqref{eq:grouplasso_gen_new} is a special case of \eqref{eq:LSE} with
{\small \begin{equation}\label{eq:groupLasso_A_B}
  \bd=(\mathbf{0}_{1 \times d}, \tau \mathbf{1}_{1 \times l})^\top, \; \lambda=0, \; A=\begin{pmatrix}
  X & \mathbf{0}_{n \times l} \\
  -X & \mathbf{0}_{n \times l} \\
  I_d  & - E  \\
  -I_d  & - E
\end{pmatrix}, \; B=\begin{pmatrix}
  -I_{n}  \\
   I_{n}  \\
   \mathbf{0}_{d \times n}   \\
   \mathbf{0}_{d \times n}
\end{pmatrix}, \; \bc=\mathbf{0}.
\end{equation}}
In the next corollary, we characterize the DF of the $\ell_\infty$-regularized group Lasso estimator using Theorem~\ref{thm:div}. 

\begin{corollary}\label{thm:grouplasso_gen}
In the $\ell_\infty$-regularized group Lasso problem described in \eqref{eq:grouplasso_gen} and \eqref{eq:grouplasso_gen_new}, for a.e.~$\by\in\mathbb{R}^n$, $\mathrm{df}(\hat\bt(\by))=\mathrm{df} (X\hat\bbt(\by))=\E[\mathrm{rank}(X_{J_0^c})],$
where
\[
J_0=\Big\{i \in \{1,\ldots, d\}:i\in\mathcal{G}_j, \hat\beta_i(\by)=\|\hat\bbt_{\mathcal{G}_j}(\by)\|_\infty \; \text{ for some } j\in\{1,2,\dots,l\}\Big\},
\]
and $J_0^c$ is the complement set of $J_0$ and $X_{J_0^c}$ consists of the columns of $X$ indexed by $J_0^c$.
\end{corollary}

\section{Application: SURE and the Choice of Tuning Parameters}\label{sec:SURE}

Consider the formulation of the problem posited in~\eqref{eq:SeqMdl}. For notational simplicity, we will use $\lambda$ to denote the tuning parameter in the regularized/constrained LSE $\hbt_\lambda(\by)$ (we highlight the dependence of $\hbt(\by)$ on $\lambda$ in this section). For example, in bounded isotonic regression the tuning parameter is the choice of the range of $\bt$ (i.e., the parameter $\gamma$ in \eqref{eq:partial_iso_bounded}); in penalized convex regression (see~\eqref{eq:LSE_Lip}) the estimator depends on the tuning parameter $\lambda$ on the norm of the subgradients.



In this section we use SURE to choose the tuning parameter $\lambda$. Let
\begin{equation}\label{eq:L_n}
	L_n(\lambda) = \| \hbt_\lambda(\by) - \bt^*\|^2_2
\end{equation}
denote the loss in estimating $\bt^*$ by $\hbt_\lambda(\by)$.  We would ideally like to choose $\lambda$ by minimizing $L_n(\cdot)$. Let 
$\lambda^* := \arg \min_{\lambda \ge 0} L_n(\lambda).$
We note that $\lambda^*$ is a random quantity as $L_n(\lambda)$ is random. Of course, we cannot compute $\lambda^*$ as we do not know $\bt^*$. However we can minimize an (unbiased) estimator of $L_n$, assuming $\sigma$ is known, as described below. Let
\begin{equation}\label{eq:U_n}
	U_n(\lambda) := \|\by -  \hbt_\lambda(\by)\|^2_2 + 2 \sigma^2 D(\hbt_\lambda(\by)) - n \sigma^2,
\end{equation}
where $D(\hbt_\lambda(\by))$ denotes the divergence of $ \hbt_\lambda(\by)$. It is well known that for all $\lambda \ge 0$, $\E[U_n(\lambda)] = \E[L_n(\lambda)]$; see \citet{S81} (also see Proposition 2 of \citet{MW00}). The quantity $U_n$ in \eqref{eq:U_n} is usually called the SURE. Let
\begin{equation}\label{eq:SURE_tuned}
\hat \lambda := \arg \min_{\lambda \ge 0} U_n(\lambda)
\end{equation}
be the minimizer of $U_n(\lambda)$, which can be computed from the data (if $\sigma^2$ is assumed known). Note that here we would need to compute the divergence of $ \hbt_\lambda(\by)$, which we can calculate using the results in the previous sections.

We empirically study the behavior of the ratio $L_n(\hat \lambda)/L_n(\lambda^*)$ for bounded isotonic regression and penalized convex regression. We also compare the performance of different tuning parameter selection methods --- SURE and cross-validation --- including the no-tuning parameter approach (e.g., the standard unbounded isotonic regression and un-penalized convex regression) for these two problems.

{In Sections~\ref{sec:exp_iso} and~\ref{sec:exp_cvx} we provide simulation studies when the true value of the noise variance $\sigma^2$ is assumed known for SURE. When $\sigma^2$ is known, the SURE method significantly outperforms its competitors. However, we note that the CV method does not require any knowledge of $\sigma^2$. In Section \ref{sec:exp_sigma}, we estimate  $\sigma^2$ using an approach proposed in~\cite{MW00}. In this case, the performance of SURE and CV are comparable but CV is computationally more expensive than SURE.
}


\subsection{Bounded Isotonic Regression}
\label{sec:exp_iso}


\begin{figure}[!t]
        \centering
        \subfigure[b][$d=2$]{
                \includegraphics[width=0.40\textwidth]{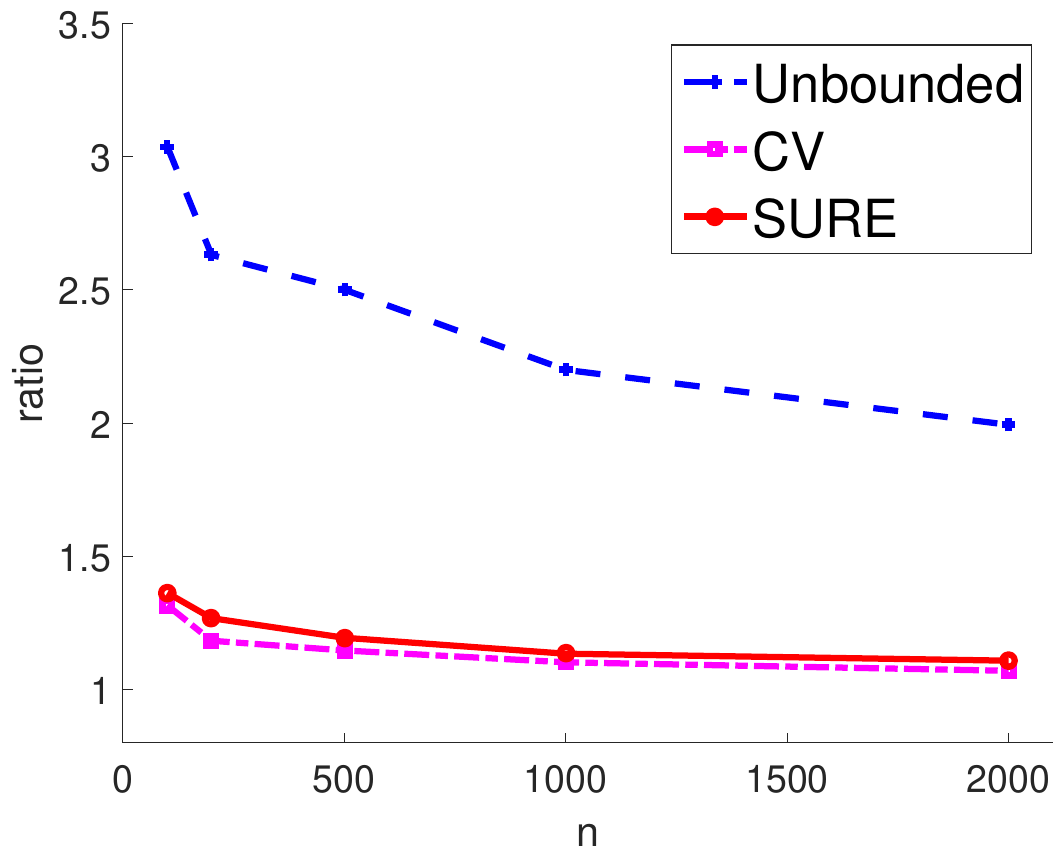}
                \label{fig:iso_ratio_d_2}
                }
        \subfigure[b][$d=5$]{
                \includegraphics[width=0.40\textwidth]{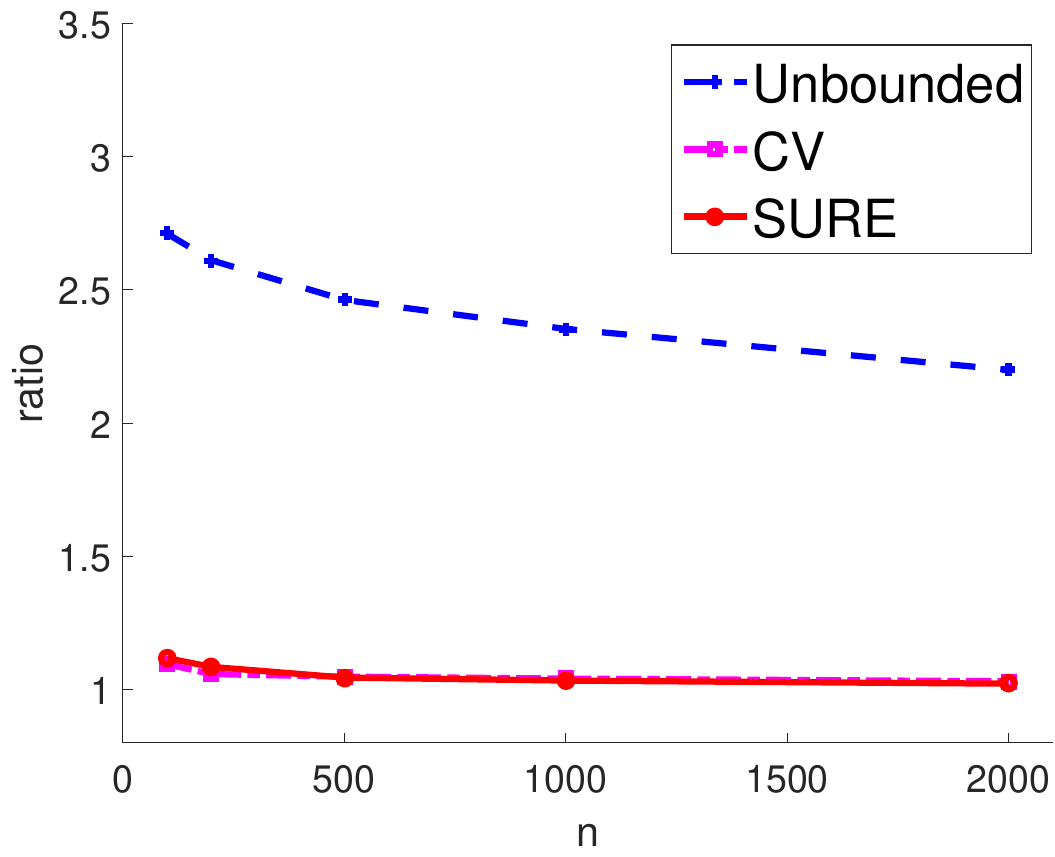}
                \label{fig:iso_ratio_d_5}
                }\\%
        \subfigure[b][$d=7$]{
                \includegraphics[width=0.40\textwidth]{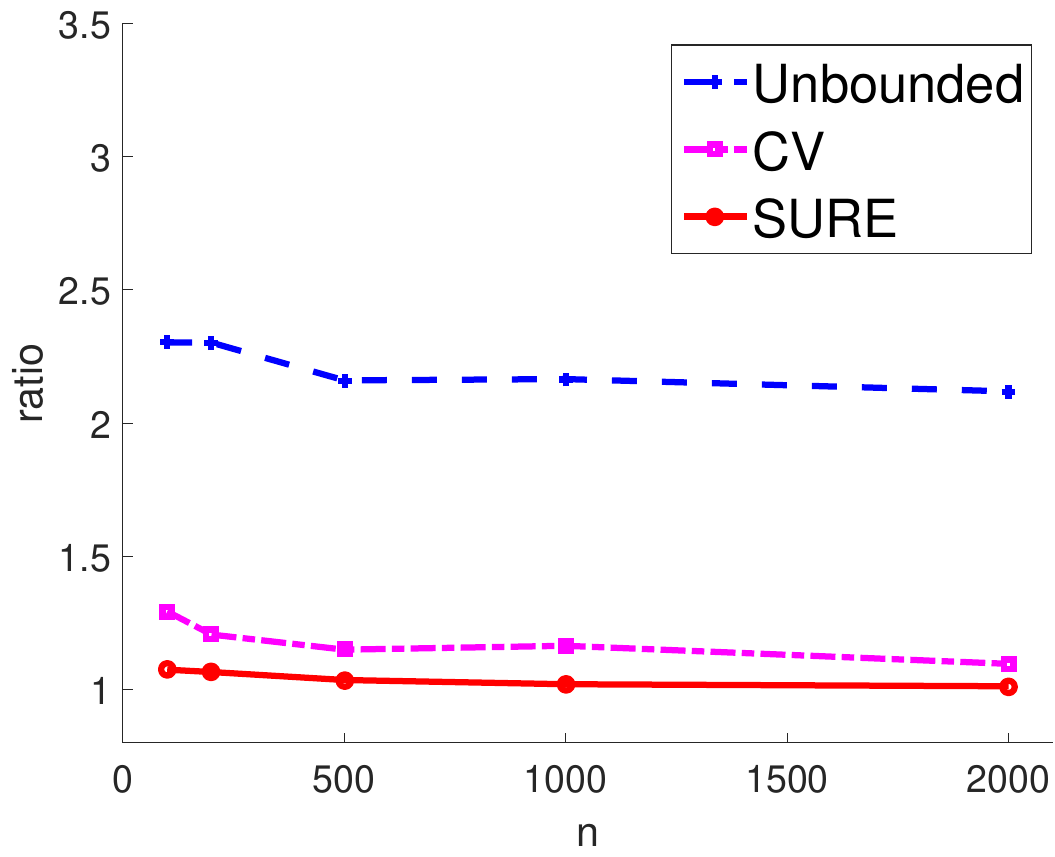}
                \label{fig:iso_ratio_d_7}
                }\
         \subfigure[b][$d=10$]{
                \includegraphics[width=0.40\textwidth]{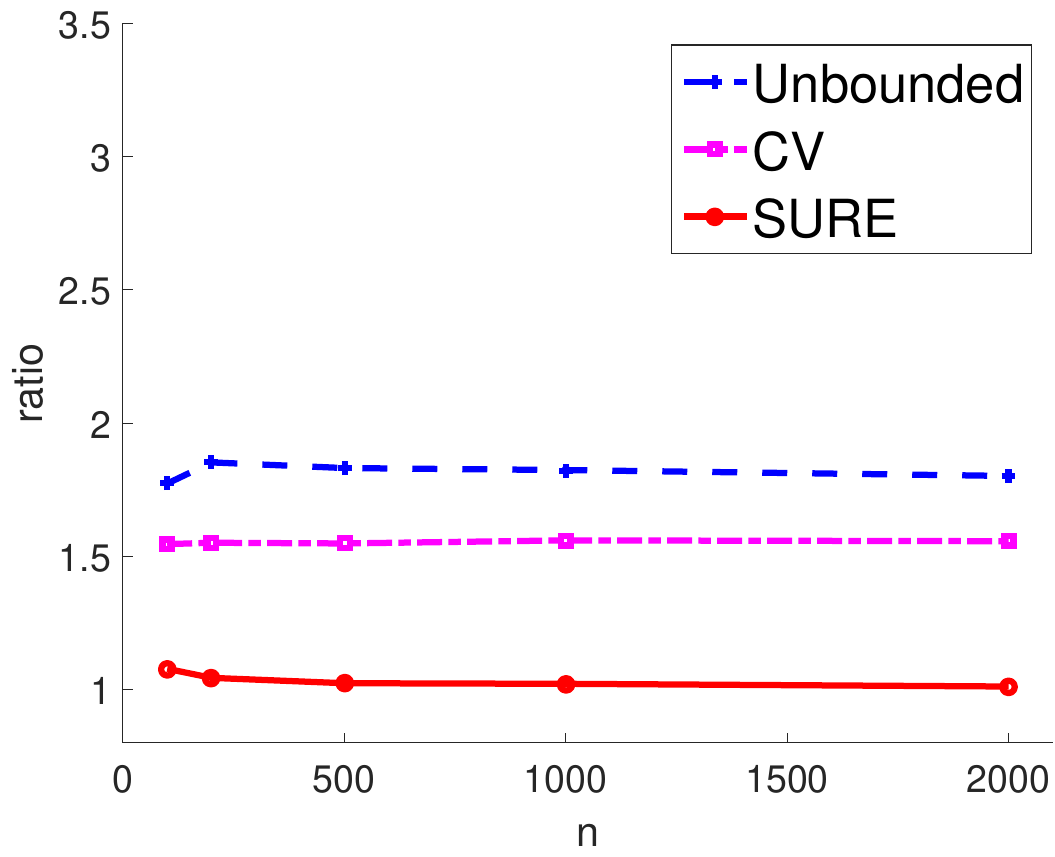}
                \label{fig:iso_ratio_d_10}
                }
        \caption{Comparison between the unbounded ratio, the CV ratio and the SURE ratio for isotonic regression.}
        \label{fig:iso_ratio} \vspace{-2mm}
\end{figure}
We generate $n$ i.i.d.~design points $\bx_{i} \sim \text{Unif}[0,1]^d$, for $i=1,\ldots,n$. We set the regression function $f: \mathbb{R}^d \rightarrow \mathbb{R}$ to be $f(\bx)=\|\bx\|_2^2$. Recall that $\bt^* = (f(\bx_1),\ldots, f(\bx_n))$, which is a bounded vector (since $\|\bx\|_2^2 \leq d$) and satisfies $\theta_i^* \leq \theta_j^*$ whenever $\bx_i \leq \bx_j$. We generate the response $y_i$, for $i=1,\ldots,n$, according to model~\eqref{eq:RegMdl2} with $\sigma^2=1$.

Since the true regression function $f$ is a bounded isotonic function, we estimate $\bt^*$ by minimizing $\|\bt-\by\|_2^2$ subject to the following constraints. For each pair $(i,j)$, we put an isotonic constraint $\theta_i \leq \theta_j$ whenever $\bx_i \leq \bx_j$. We further add one additional \emph{boundedness constraint} $\max \theta_i -\min \theta_i \leq \lambda$, where $\lambda$ is the tuning parameter (i.e., the parameter $\gamma$ in \eqref{eq:partial_iso_bounded}). For each given $\lambda$, we obtain the LSE $\hat \bt_\lambda(\by)$.

We demonstrate the performance of the selected parameter $\widehat{\lambda}$ using SURE. In particular, we compute the ratio $L_n(\widehat{\lambda})/L_n(\lambda^*)$, where $\widehat{\lambda}$ is selected by \eqref{eq:SURE_tuned} (we call this the {\it SURE ratio}). 
We compare the SURE ratio to the so-called \emph{CV ratio}, where the boundedness parameter is selected by 5-fold cross-validation. We note that when implementing the CV method, for a given training set $\mathcal{T}_{\text{tr}}$, the estimated function value at a point $\bx$ is set to $\hat{f}(\bx) := \min_{\bx_i \in \mathcal{T}_{\text{tr}}: \bx_i \geq \bx} \hat{\theta}_{\lambda, i}$, where $\hat{\theta}_{\lambda, i}$ the estimated function value at the training data point $\bx_i$ obtained from the bounded isotonic LSE. Such a way of extending the estimated function values (on the training set) to new data points ensures that the extended function is monotone and bounded; this extension has also been used by other authors (see e.g., \citet{Sabyasachi15}). We also compare the performance of the bounded isotonic LSE with the unbounded LSE where we do not include the boundedness constraint $\max \theta_i -\min \theta_i \leq \lambda$ (or equivalently, set $\lambda=+\infty$ and compute ${L_n(\infty)}/{L_n(\lambda^*)}$). 

We set $d=2, 5, 7, 10$ and for each fixed $d$, we vary the sample size $n=100, 200, 500, 1000, 2000$ and compute the SURE, CV and unbounded ratios over 100 independent replications and plot the results in Figure~\ref{fig:iso_ratio}.  
From Figure \ref{fig:iso_ratio} one can see that the SURE ratios are, in general, much smaller than the unbounded ratios, illustrating the usefulness of including the boundedness constraint in isotonic regression. When the dimension is very small (e.g., $d=2$)  the CV ratio slightly outperforms the SURE ratio; while for larger $d$ (e.g., $d=7$ or $d=10$)  the SURE based method significantly outperforms the CV approach. Moreover, for larger sample sizes $n$, the SURE ratios are close to 1 indicating that the bounded LSE tuned via SURE performs as good as the bounded LSE with oracle tuning.



\subsection{Penalized Multivariate Convex Regression}
\label{sec:exp_cvx}

\begin{figure}[!t]
        \centering
        \subfigure[b][$n=100$, $d=4$]{
                \includegraphics[width=0.40\textwidth]{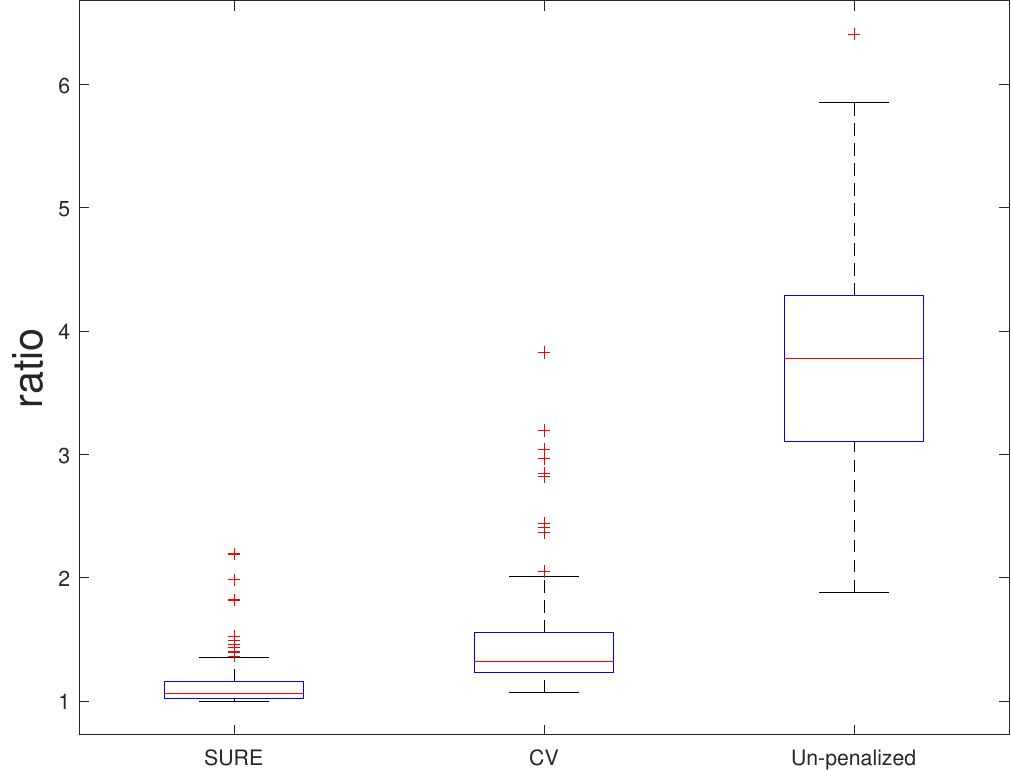}
                \label{fig:cvx_n_100_d_2}
                }
        \subfigure[b][$n=500$, $d=4$]{
                \includegraphics[width=0.40\textwidth]{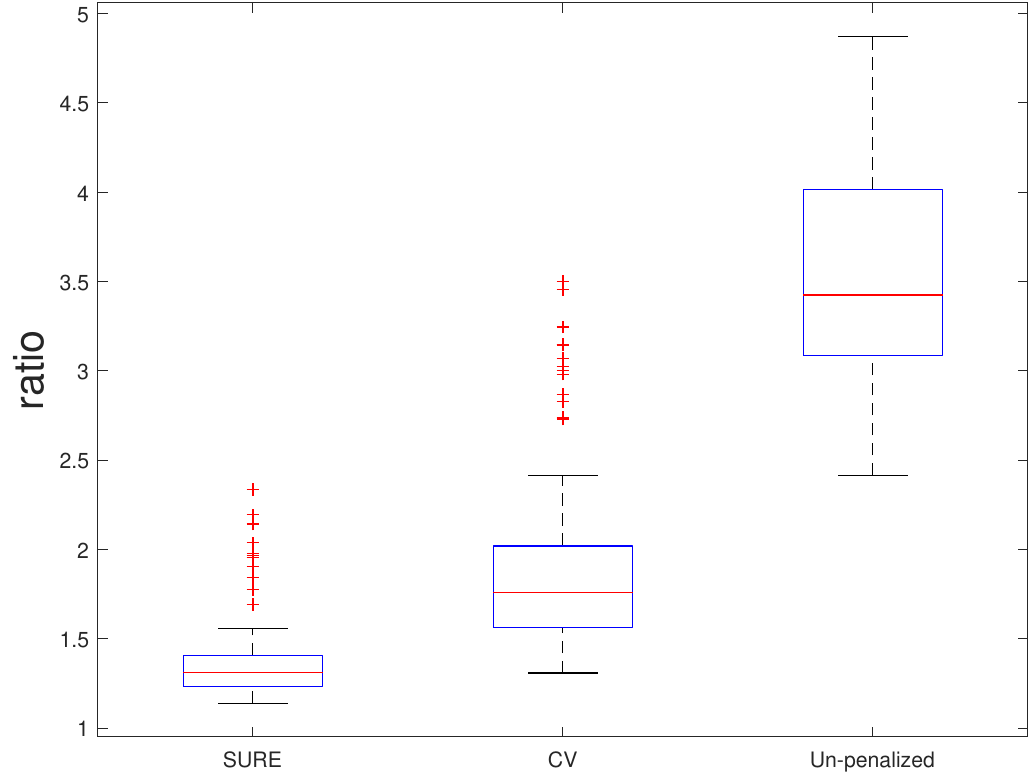}
                \label{fig:cvx_n_100_d_4}
                }\\
         \subfigure[b][$n=100$, $d=10$]{
                \includegraphics[width=0.40\textwidth]{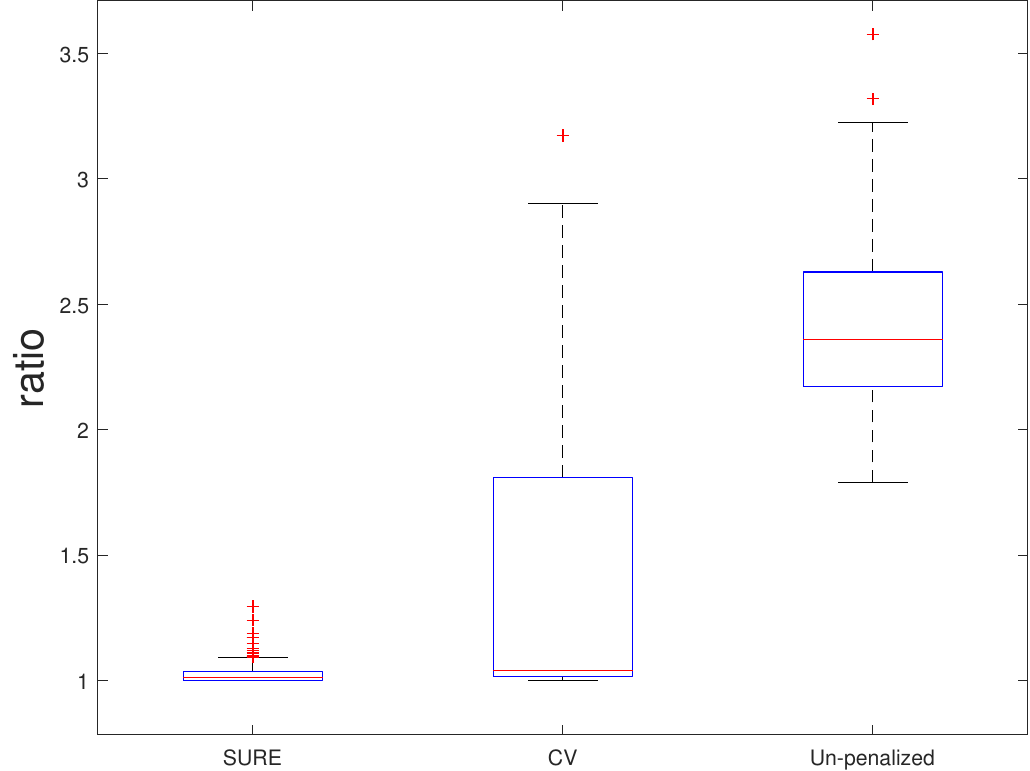}
                \label{fig:cvx_n_500_d_4}
                }\
         \subfigure[b][$n=500$, $d=10$]{
                \includegraphics[width=0.40\textwidth]{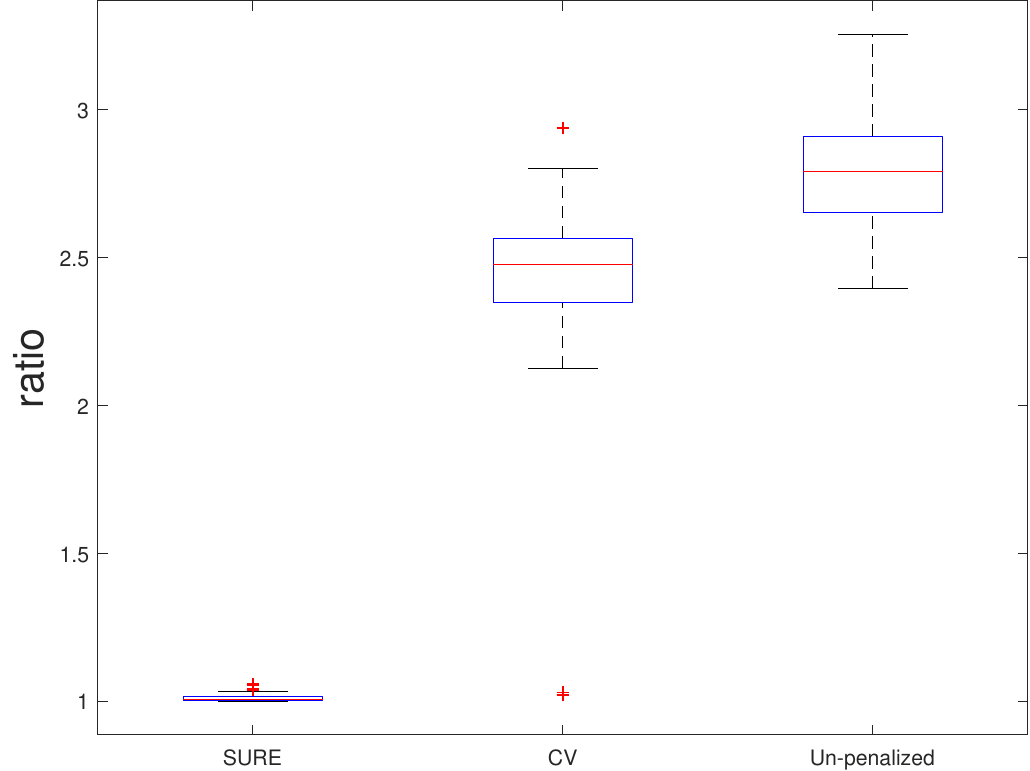}
                \label{fig:cvx_n_500_d_10}
                }
        \caption{Boxplots of the SURE ratio, the CV ratio and un-penalized ratio (from left to right) for multivariate convex regression.}
        \label{fig:cvx_ratio}\vspace{-2mm}
\end{figure}

We generate $n$ i.i.d.~design points $\bx_{i} \sim \text{Unif}[-1,1]^d$, for $i=1,\ldots,n$. We set the convex regression function $f: \mathbb{R}^d \rightarrow \mathbb{R}$ to be $f(\bx)=\|\bx\|_2^2$, which is symmetric around $\mathbf{0}$. We generate the response $y_i$, for $i=1,\ldots,n$, according to model~\eqref{eq:RegMdl2} with $\sigma=0.5$. Let $\bt^* = (f(\bx_1),\ldots, f(\bx_n))$. We estimate $\bt^*$ by solving the penalized multivariate convex regression problem described in~\eqref{eq:LSE_Lip} using the SDPT3 package \citep{SDPT3}. 
We note that since the optimization problem for penalized multivariate convex regression (in \eqref{eq:LSE_Lip}) has a lot of constraints and many variables (i.e., $n(n-1)$ constraints and $nd$ variables), we only consider smaller sample sizes ($n$) in our simulation experiments. Nevertheless, a smaller $n$ is still sufficient to demonstrate the superior performance of the estimator tuned by minimizing SURE.  In particular, we consider $d=4$ and $10$, $n=100$ and $500$, and compute the SURE ratio $L_n(\widehat{\lambda})/L_n(\lambda^*)$, where $\widehat{\lambda}$ is defined as in  \eqref{eq:SURE_tuned}. We compare the SURE ratio to the CV ratio, where $\lambda$ is selected by 5-fold cross-validation. We note that when implementing the CV method, for a given training set $\mathcal{T}_{\text{tr}}$, the estimated function value at any $\bx$ is set to
\begin{equation}\label{eq:linear_interp}
\hat{f}(\bx) = \max_{\bx_i \in \mathcal{T}_{\text{tr}}} \left( \hat{\theta}_{\lambda, i} + (\bx- \bx_i)^\top \hat{\bxi}_{\lambda, i} \right),
\end{equation}
where $\hat{\theta}_{\lambda, i}$ and $\hat{\bxi}_{\lambda, i}$ are solutions of the penalized multivariate convex regression problem in \eqref{eq:LSE_Lip}. The constructed $\hat{f}: \mathbb{R}^d \rightarrow \mathbb{R}$ is clearly a (piecewise affine) convex function; see Section 6.5.5 in \cite{Boyd:2004}. We also include the ``un-penalized ratio" $L_n(0)/L_n(\lambda^*)$ as a competitor, i.e., the ratio between the loss obtained from the un-penalized multivariate convex regression estimator as defined in \eqref{eq:multi_cvx} and the oracle loss.  

We present the results in the form of boxplots in Figure \ref{fig:cvx_ratio}, obtained from 100 independent replicates of $\by$ (fixing the design variables). We observe that penalized multivariate convex regression, with the regularization parameter tuned by SURE, has better performance. As we had inferred from Figure~\ref{fig:iso_ratio}, Figure~\ref{fig:cvx_ratio} also shows that the SURE ratios are much smaller than both the CV ratios and un-penalized ratios and their difference is more pronounced as the dimension $d$ increases. Further, the SURE ratio concentrates near 1 suggesting that SURE is doing a very good job in selecting the tuning parameter. 

\subsection{SURE Without the Knowledge of $\sigma^2$}
\label{sec:exp_sigma}

\begin{table}[!t]
\centering
\caption{Comparison of the different tuning parameter selection methods for isotonic regression: the unbounded ratio, the CV ratio, the SURE ratio with known $\sigma^2$, and the SURE ratio with estimated $\hat{\sigma^2}$. The standard errors are provided in parenthesis.}
\label{tab:CV_iso}
{\small
\vspace{0.5cm}
  \begin{tabular}{|c c p{2.5cm} p{2.5cm} p{3cm} p{3cm}|} \hline
    $n$ & $d$            & Unbounded & CV & SURE  known $\sigma^2$  & SURE est $\hat{\sigma^2}$ \\ \hline \hline
    \multirow{3}{*}{100} & 2 & 3.09 (0.86) & 1.28 (0.23) & 1.27 (0.22) & 1.28 (0.23)  \\
                         & 5 & 2.66 (0.37) & 1.12 (0.11) & 1.11 (0.14) & 1.47 (0.15)  \\
                         & 10 & 1.76 (0.25) & 1.55 (0.17) & 1.09 (0.11) & 1.62 (0.17) \\\hline
    \multirow{3}{*}{1000} & 2 & 2.42 (0.50) & 1.07 (0.10) & 1.10 (0.12) & 1.22 (0.15) \\
                          & 5 & 2.35 (0.18) & 1.04  (0.03) & 1.03 (0.05) & 1.04 (0.06) \\
                          & 10 & 1.80 (0.07) & 1.55 (0.05)  & 1.02 (0.02) & 1.48 (0.04) \\ \hline
  \end{tabular}
}
\end{table}

\begin{table}[!t]
\centering
\caption{Comparison of the different tuning parameter selection methods for convex regression: the un-penalized ratio, the CV ratio, the SURE ratio with known $\sigma^2$, and the SURE ratio with estimated $\hat{\sigma^2}$. The standard errors are provided in parenthesis.}
\label{tab:cvx}
{\small
\vspace{0.5cm}
  \begin{tabular}{|c c p{2.5cm} p{2.5cm} p{3cm} p{3cm}|} \hline
    $n$ & $d$ & Un-penalized & CV &  SURE  known $\sigma^2$  & SURE est $\hat{\sigma^2}$ \\ \hline \hline
    \multirow{3}{*}{100} & 2 & 2.74 (1.12) & 1.68 (0.52) & 1.35 (0.32) & 1.46 (0.39)  \\
                         & 3 & 3.22 (0.86) & 1.42 (0.30) & 1.12 (0.22) &  1.15 (0.23)  \\
                         & 5 & 3.62 (0.53) & 1.14 (0.25) & 1.04 (0.15) & 1.30  (0.18)  \\\hline
    \multirow{3}{*}{500} & 2 & 2.77 (0.98) & 1.20 (0.32) & 1.07 (0.11) & 1.22 (0.12) \\
                          & 3 & 3.47 (0.74) & 1.51 (0.29) & 1.38 (0.08) & 1.49 (0.08) \\
                          & 5 & 3.91 (0.50) & 1.40 (0.18) & 1.05 (0.05) & 1.05  (0.06) \\ \hline
  \end{tabular}
  }
\end{table}

{In this section, we assume that the noise variance $\sigma^2$ in unknown. To estimate $\sigma^2$ we adopt a method proposed in~\cite{MW00} and then apply SURE with the estimated $\sigma^2$. In particular, we first obtain an initial estimator $\hat{\bt}$ using unbounded isotonic regression (or un-penalized convex regression) and then estimate $\sigma^2$ by $\hat{\sigma^2}=\frac{\|\hat{\bt}-\by\|_2^2}{n-2D(\by)}$, where $D(\by)$ is the divergence of the initial estimator $\hat{\bt}$. The rationale for this choice comes from~\citet[Corollary 1]{MW00} where the authors study (unbiased) estimators for $\sigma^2$ in the setup of~\eqref{eq:SeqMdl}. The averaged ratios $L_n(\widehat{\lambda})/L_n(\lambda^*)$ over 100 independent runs for different tuning parameter selection methods are provided in Table~\ref{tab:CV_iso} (for isotonic regression) and Table~\ref{tab:cvx} (for convex regression).  
For convex regression, the SURE with unknown ${\sigma^2}$ outperforms CV in most cases, whereas  for isotonic regression CV performs better in some cases. Moreover, we point out the SURE is computationally more efficient than CV. In particular, 5-fold CV needs to solve five optimization problems for each value of the tuning parameter; thus the SURE method is about five times faster. Moreover, the standard errors of SURE are comparable to those errors of the CV method, and are smaller than the errors for the unbounded and un-penalized cases.
}

\newtheorem*{rep@theorem}{\rep@title}
\newcommand{\newreptheorem}[2]{%
	\newenvironment{rep#1}[1]{%
		\def\rep@title{#2 \ref{##1}}%
		\begin{rep@theorem}}%
		{\end{rep@theorem}}}
\makeatother
\newreptheorem{theorem}{Theorem}
\newreptheorem{lemma}{Lemma}
\newreptheorem{corollary}{Corollary}
\newreptheorem{proposition}{Proposition}

\newcounter{scnt}[section]
\renewcommand{\thercnt}{(\roman{scnt})}
\setcounter{equation}{50}
\numberwithin{figure}{section}
\renewcommand\thesection{\Alph{section}}

\def\spacingset#1{\renewcommand{\baselinestretch}%
	{#1}} \spacingset{1.4}

\newpage
\if1\blind
{
	\title{\bf Supplement to On Degrees of Freedom of Projection Estimators with Applications to Multivariate Nonparametric Regression}
	\author{Xi Chen \hspace{.2cm}\\
		Stern School of Business, New York University \\
		and \\
		Qihang Lin \\ 
		Tippie College of Business, University of Iowa \\
		and \\
		Bodhisattva Sen \footnote{Supported by NSF Grant DMS-1150435} \\
		Department of Statistics, Columbia University
	}
	\date{}
	\maketitle
} \fi

\if0\blind
{
	\bigskip
	\bigskip
	\bigskip
	\begin{center}
		{\LARGE\bf  Supplement to On Degrees of Freedom of Projection Estimators with Applications to Multivariate Nonparametric Regression}
	\end{center}
	\medskip
} \fi

The supplementary material is organized as follows:
\begin{enumerate}
	\item In Section \ref{sec:convex}, we provide the necessary background on convex analysis, which will be heavily used in our proofs.
	
	
	
	\item In Section \ref{sec:supp_main}, we provide some results used in the proof of our main theorem --- Theorem \ref{thm:div}. In particular, we provide proofs of Lemma \ref{lem:bounded} (in Section \ref{sec:proof_bounded}), Lemma \ref{lem:div} (in Section \ref{sec:proof_local_lemma}), and Theorem \ref{thm:div} (in Section \ref{sec:supp_theorem_div}). A simple sanity check for Theorem~\ref{thm:div} is given in Section~\ref{sec:SanityCheck}.
	
	Moreover,  we provide a concrete example to highlight the difference between our result Theorem \ref{thm:div} for  the $\lambda>0$ case and the previous results on the divergence of projection estimators (see Section \ref{sec:supp_nontrivial}).
	
	\item In Section \ref{sec:supp_DFIso}, we provide proofs of the results for (bounded) isotonic regression, including the proofs of Proposition \ref{prop:uni_bound_iso} (in Section  \ref{sec:proof_uni_bound_iso}), Proposition \ref{prop:threshold_gen} (in Section  \ref{sec:proof_DFIso}), and Theorem \ref{thm:iso_bounded_monotone} (in Section \ref{sec:supp_monotone}). 
	
	\item In Section \ref{sec:proof_other_app}, we provide the proofs of Proposition \ref{prop:fused_add} (DF for additive models; see Section \ref{sec:proof_additive}) and Corollary \ref{thm:grouplasso_gen} (DF for generalized group Lasso; see Section \ref{sec:proof_group}). In Section \ref{sec:recover_other}, we apply our general theorem to recover several well-known results on the DF including  Lasso, generalized Lasso, linear regression, and ridge regression.
	
	
\end{enumerate}

\section{Background Knowledge on Convex Analysis}
\label{sec:convex}

We start with some definitions and notations. We denote by $\langle \cdot, \cdot \rangle$ the usual inner product in Euclidean spaces. Recall that a set $\C\subseteq \mathbb{R}^n$ is a \emph{convex polyhedron} if it can be represented as in~\eqref{eq:CvxPoly} for some known matrix $B:=[\bb_1, \ldots, \bb_m]^\top \in \mathbb{R}^{m \times n}$ and a vector $\mathbf{c} := [c_1,\ldots, c_m]^\top \in \R^{m \times 1}$. When $\mathbf{c}=\mathbf{0}$, it becomes a \emph{polyhedral cone} (denoted by $\mathcal{K}$), which is the intersection of finitely many halfspaces that contain the origin and can be represented as,
\begin{eqnarray}  \label{eq:CvxCone}
\mathcal{K} = \{\btheta \in \R^n: B \btheta \leq \mathbf{0}\}.
\end{eqnarray}


A finite collection of vectors $\bt_1,\bt_2,\dots,\bt_k\in\mathbb{R}^n$ is \emph{affinely independent} if the only unique solution to the equality system $\sum_{i=1}^k\alpha_i\bt_i=0$ and $\sum_{i=1}^k\alpha_i=0$ is $\alpha_i=0$, for $i=1,2,\dots,k$. The \emph{dimension}  of $\C$ (denoted by $\text{dim}(\C)$) is the maximum number of affinely independent points in $\C$ minus one. We say that $\C$ has \emph{full dimension} if $\text{dim}(\C)=n$. The \emph{affine hull} of $\C$, denoted by $\text{aff}(\C)$, is the \emph{affine space} consisting of all affine combinations of elements of $\C$, i.e.,
$
\text{aff}(\C):=\left\{\sum_{i=1}^k\alpha_i\bt_i : k>0,\bt_i\in \C,\alpha_i \in \R,\sum_{i=1}^k\alpha_i=1\right\}.
$
Note that $\C$ has full dimension if and only if $\text{aff}(\C)=\mathbb{R}^n$.

For a given convex polyhedron $\C$ in the form of  \eqref{eq:CvxPoly},
a nonempty subset $F \subseteq \C$ is called a \emph{face} of $\C$ if there exists $J\subseteq \{1,2,\dots,m\}$ so that $\vspace{-0.07in}$
\begin{equation}\label{eq:F}
F=\{\bt \in \C: \langle \bb_i,\btheta \rangle = c_i,  \; \forall \, i \in J \}. \vspace{-0.07in}
\end{equation}
A point $\bt\in\C$ can belong to more than one face. The smallest face of $\C$ containing $\bt$, in the sense of set inclusion, is called the \emph{minimal face containing} $\bt$. The following lemma characterizes the affine hull of a face of a polyhedron.

\begin{lemma}\label{lem:aff_hull}
	For any face $F$ of $\C$ in \eqref{eq:CvxPoly}, let $J_F=\{ i \in \{1, \ldots, m\}: \langle \bb_i,\btheta \rangle = c_i, \; \forall \; \bt \in F\}$. Then the affine hull of $F$ can be represented as
	$
	\mathrm{aff}(F)=\left\{\bt\in\mathbb{R}^n: \langle \bb_i,\btheta \rangle = c_i, \; \forall \, i \in J_F \right\}.
	$
\end{lemma}

\begin{proof}[Proof of Lemma~\ref{lem:aff_hull}]
	Suppose that $\bt \in\text{aff}(F)$, i.e., $\btheta=\sum_{j=1}^k\alpha_j\btheta_j$ where $k>0$, $\btheta_j\in F$, $\alpha_j\in\mathbb{R}$ and $\sum_{j=1}^k\alpha_j=1$. For any $i \in J_F $, $\langle \bb_i,\btheta \rangle=\sum_{j=1}^k\alpha_j\langle \bb_i,\btheta_j \rangle=\sum_{j=1}^k\alpha_jc_i=c_i$. Therefore, the inclusion $\subseteq$ follows.
	
	Suppose $\btheta$ satisfies $\langle \bb_i,\btheta \rangle = c_i$ for all $i \in J_F$. We claim that there exists $\btheta'\in F$ such that $\langle \bb_i,\btheta' \rangle < c_i$ for all $i \in J_F^c$. In fact, by the definition of maximal index set $J_F$, there exists $\btheta_i\in F$ for each $i \in J_F^c$ such that $\langle \bb_i,\btheta_i \rangle < c_i$. Then, $\btheta'$ can be chosen as $(\sum_{i\in J_F^c}\btheta_i)/|J_F^c|\in F$. If $\btheta=\btheta'$, $\btheta$ belongs to $F\subseteq \text{aff}(F)$.  If $\btheta\neq\btheta'$, there exists a sufficiently small $\epsilon>0$ such that $\btheta_{\epsilon}:=\epsilon\btheta+(1-\epsilon)\btheta'$ satisfies $\langle \bb_i,\btheta_{\epsilon} \rangle = c_i$ for all $i \in J_F$ and  $\langle \bb_i,\btheta_{\epsilon} \rangle \leq c_i$ for all $i \in J_F^c$. Hence, $\btheta_{\epsilon}\in F$ which implies that $\btheta=\btheta_{\epsilon}/\epsilon+(\epsilon-1)\btheta'/\epsilon\in \text{aff}(F) $.  Therefore, the inclusion $\supseteq$ follows.
\end{proof}

\begin{figure}[!t]
	\centering
	\includegraphics[width=0.4\textwidth]{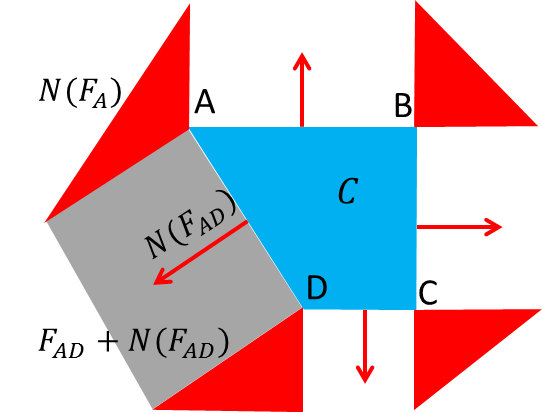}
	\caption{Illustration of the normal cones of a polyhedron: The four vertices of the polyhedron $\mathcal{C}$ are denoted by $A$, $B$, $C$ and $D$, respectively. We denote each face of $\mathcal{C}$ by its vertices, e.g., $F_{AD}$ denotes the line segment connecting $A$ and $D$ (one-dimensional face) while $F_A$ denotes the vertex $A$ (zero-dimensional face). The normal cone of all one-dimensional faces have been depicted by the red arrows while the normal cone of all zero-dimensional faces are depicted by the red conic regions. The grey area corresponds to $F_{AD}+N(F_{AD})$.}
	\label{fig:NormalCone}
\end{figure}

The \emph{normal cone} associated with a face $F$ is defined as
\begin{equation}\label{eq:normal_cone}
N(F):=\left\{\bh\in\mathbb{R}^{n}: F\subseteq \argmax_{\btheta\in \C} \bh^\top\btheta\right\}.
\end{equation}
From a geometric perspective, the normal cone of $F$ is the set of directions in $\mathbb{R}^n$ that are perpendicular to $F$ and point outward from $\C$ (see an illustration in Figure \ref{fig:NormalCone}). In this paper, we will often deal with the polyhedron $F+N(F)=\{\bt + \bh: \bt \in F, \bh \in N(F)\}$, which consists of all points in $\mathbb{R}^n$ that can be reached by moving a point in $F$ along a direction in $N(F)$. As a consequence, the projection of a point in $F+N(F)$ onto $\mathcal{C}$ will lie on the face $F$ of $\mathcal{C}$, which is stated as the following lemma.
\begin{lemma}\label{lem:project_NF}
	Let $F$ be a face of $\C$. For any $\bz \in F+N(F)$, $P_\C(\bz) \in F$, where the operator $P_\C(\cdot)$ is defined in \eqref{eq:Proj}.
\end{lemma}

\begin{proof}[Proof of Lemma~\ref{lem:project_NF}]
	Since $\bz \in F+N(F)$, there exist $\bz'\in F$ and $\bh\in N(F)$ such that $\bz=\bz'+\bh$. Since $\hat\bz := P_C(\bz)$ is the optimal solution of $\min_{\btheta\in\C}\|\btheta-\bz\|_2^2$, by the optimality condition (see e.g.,~\citet[Proposition 4.7.1]{bertsekas03}), we have
	\[
	\langle \hat\bz-\bz ,\bt-\hat\bz\rangle = \langle \hat\bz-\bz'-\bh ,\bt -\hat\bz \rangle \geq0
	\]
	for any $\bt\in\C$. Choosing $\bt=\bz'$ in the inequality above, we have
	\[
	\langle \bh ,\hat\bz-\bz'\rangle\geq \|\hat\bz-\bz'\|_2^2.
	\]
	As $\bh\in N(F)$, $\bz'\in F\subseteq \argmax_{\btheta\in \C} \bh^\top\bt$, which implies $\langle\bh ,\hat\bz-\bz'\rangle\leq 0$, again appealing to the optimality condition. This, together with the above display implies $\hat\bz=\bz'\in F$.
\end{proof}



In additional to the normal cone, some other useful concepts from convex analysis are defined in the following. Given a convex polyhedron $\mathcal{C}$,  the \emph{interior} of $\C$, denoted by $\text{int}(\C)$, is defined as
$$
\text{int}(\C):= \left\{ \bt \in \C : \exists \,\epsilon > 0 \mbox{ such that } B_\epsilon(\bt) \subseteq \C \right\},
$$
where $B_\epsilon(\bt)=\{\mathbf{x} \in\mathbb{R}^n:\|\mathbf{x}-\bt\|_2\leq\epsilon\}$ is the Euclidean ball of radius $\epsilon$ centered at $\bt$. The \emph{boundary} $\text{bd}(\C)$ of $\C$ is defined as
$$
\text{bd}(\C):= \left\{ \bt\in\mathbb{R}^n : \forall \; \epsilon > 0, \; \C \cap B_\epsilon(\bt)\neq\emptyset \text{ and}\left(\mathbb{R}^n \backslash\C\right) \cap B_\epsilon(\bt)\neq\emptyset\right\}.
$$
The \emph{relative interior} $\text{relint}(\C)$ of $\C$ is defined as its interior within $\text{aff}(\C)$, i.e.,
$$
\text{relint}(\C):= \left\{ \bt \in \C : \exists \,\epsilon > 0 \mbox{ such that } B_\epsilon(\bt) \cap \operatorname{aff}(\C) \subseteq \C \right\}.
$$
Similarly, the \emph{relative boundary} $\text{relbd}(\C)$ of $\C$ is defined as its boundary within $\text{aff}(\C)$, i.e.,
$$
\text{relbd}(\C):= \left\{ \bt \in \operatorname{aff}(\C)  : \forall\epsilon > 0, \C \cap B_\epsilon(\bt)\neq\emptyset \text{ and}\left(\text{aff}(\C)\backslash\C\right) \cap B_\epsilon(\bt)\neq\emptyset\right\}.
$$

Consider a polyhedron of a higher dimension defined in \eqref{eq:Q}. Similar to \eqref{eq:F}, the face of $\Q$ is
a nonempty subset $F \subseteq \Q$ if there exists $J\subseteq \{1,2,\dots,m\}$ so that
\begin{equation}\label{eq:F_Q}
F=\{(\bxi,\bt) \in \Q: \langle \ba_i,\bxi \rangle +\langle \bb_i,\btheta \rangle = c_i,  \; \forall \, i \in J \}.
\end{equation}
The \emph{projected polyhedron} of $\Q$ onto the subspace of $\bt$ is defined in~\eqref{eq:proj_C} which is also a polyhedron.
We also note that although $\Proj_{\bt}(\Q)$ is a polyhedron, it is usually not easy to express it explicitly as a set of inequalities as in \eqref{eq:CvxPoly}.
%
In addition to the projected polyhedron, we also introduce the restricted polyhedron as follows. The \emph{restriction} of $\Q$ on the space of $\bt$ at point $\bxi$ is defined as
\begin{equation}\label{eq:restriction}
R_{\bxi}(\Q):=\{\bt\in\mathbb{R}^n: (\bxi ,\bt ) \in \Q \},
\end{equation}
which is also a polyhedron. When $\bxi=\mathbf{0}$, we will omit $\bxi$ in the subscript and denote the restriction of $\Q$ at the point $\mathbf{0}$ by $R(\Q)$. The restriction of a polyhedron is not necessarily the same as the projection of it, even when $\bxi=\mathbf{0}$; see Figure~\ref{fig:PR}  for a visual illustration of the difference between $\Proj_{\bt}(\Q)$ and $R_{\bxi}(\Q)$.

\begin{figure}[!t]
	\centering
	\includegraphics[width=0.3\textwidth]{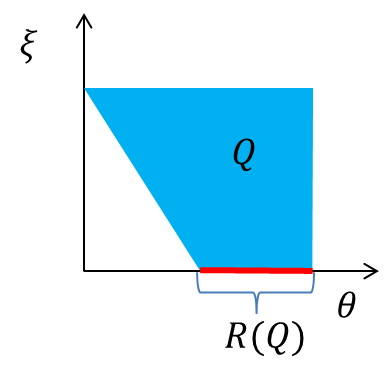}
	\includegraphics[width=0.3\textwidth]{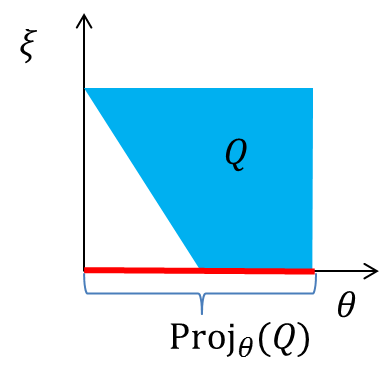}
	\caption{An illustration of the difference between projection and restriction, where both $\bxi$ and $\bt$ are one dimensional. The restriction of $\Q$ on $\bt$  when $\bxi=\mathbf{0}$ is depicted by the red line segment in the figure on the left while the projection on $\bt$ is marked by the red line segment in the figure on the right. This example is taken from~\citet{balas05}.}
	\label{fig:PR}\vspace{-3mm}
\end{figure}

\section{Proof of Results and Additional Material for Section \ref{sec:main}}
\label{sec:supp_main}

\subsection{Proof of Lemma \ref{lem:bounded}}
\label{sec:proof_bounded}
Let us recall the objective function,
\begin{eqnarray}\label{eq:supp_LSE}
(\hbt(\by), \hbxi(\by)) & \in &  \argmin_{\bt, \bxi}\frac{1}{2} \|\bt-\by\|_2^2 + \bd^\top\bxi+ \frac{\lambda}{2}\|\bxi\|_2^2\\
& & \;\;\; \text{s.t.} \;   A\bxi + B\bt \leq \bc. \nonumber
\end{eqnarray}

\begin{replemma}{lem:bounded}
	When $\lambda=0$, the optimization problem in~\eqref{eq:supp_LSE} has a bounded optimal value if and only if $-\bd=A^\top\bu$ for some $\bu\geq\mathbf{0}$.
\end{replemma}
\begin{proof}[Proof of Lemma \ref{lem:bounded}]
	Suppose  $-\bd=A^\top\bu$ for some $\bu\geq\mathbf{0}$. For any $(\bt, \bxi)$ satisfying $ A \bxi + B \bt \leq \bc$, the objective value of \eqref{eq:supp_LSE} is bounded from below as
	$$
	\frac{1}{2}\|\bt-\by\|_2^2+\bd^\top\bxi = \frac{1}{2}\|\bt-\by\|_2^2-\bu^\top A\bxi\geq\frac{1}{2}\|\bt-\by\|_2^2-\bu^\top(\bc-B \bt).
	$$
	As a strongly convex quadratic function of $\bt$, $\frac{1}{2}\|\bt-\by\|_2^2-\bu^\top(\bc-B \bt)$ is always bounded from below for any $\bt$. So is $\frac{1}{2}\|\bt-\by\|_2^2+\bd^\top\bxi$.
	
	Suppose $-\bd\neq A^\top\bu$ for any $\bu\geq\mathbf{0}$. According to Farkas's lemma (see e.g.,~\citet[Corollary 22.3.1]{Rockafellar70book}), there exists $\bh\in \mathbb{R}^p$ such that $A\bh\geq\mathbf{0}$ and $-\bd^\top\bh<0$. Given any feasible solution $(\bxi,\bt)$ for \eqref{eq:supp_LSE}, $(\bxi-t\bh,\bt)$ will also be a feasible solution for any $t\geq0$, whose objective value is
	$$
	\frac{1}{2}\|\bt-\by\|_2^2+\bd^\top(\bxi-t\bh)=\frac{1}{2}\|\bt-\by\|_2^2+\bd^\top\bxi-t\bd^\top\bh,
	$$
	which approaches $-\infty$ as $t$ increases to infinity. Therefore, \eqref{eq:supp_LSE} will not have a bounded optimal value.
\end{proof}

\subsection{Proof of Lemma \ref{lem:div}}
\label{sec:proof_local_lemma}

In this section, we provide the proof of our key technical lemma --- Lemma \ref{lem:div}.
\begin{replemma}{lem:div}
	Suppose $-\bd=A^\top\bu$ for some $\bu\geq\mathbf{0}$ whenever $\lambda=0$ in \eqref{eq:LSE}.  For any
	$\by \in \R^n$, let $(\hbt(\by), \hbxi(\by)) $ be any solution for \eqref{eq:LSE} and let
	the index set $J_\by$ be as defined in \eqref{eq:J_y_2}. For a.e.~$\by \in \R^n$,
	\begin{eqnarray}\label{eq:eq_div_mul_cvx_gen}
	\hbt(\bz) = \widetilde{\bt}(\bz), \text{ for any }\bz\text{ in a neighborhood }U\text{ of }\by,
	\end{eqnarray}
	where $\widetilde{\bt}(\bz)$ is defined as the unique $\bt$-component of the optimal solution of the following optimization problem:
	\begin{eqnarray}\label{eq:aff_LSE}
	(\widetilde{\bt}(\bz),\widetilde{\bxi}(\bz))  & \in & \argmin_{\bt, \bxi}\frac{1}{2} \|\bt-\bz\|_2^2 +\bd^\top\bxi+ \frac{\lambda}{2}\|\bxi\|_2^2\\
	& &\;\;\mathrm{s.t.} \; A_{J_\by} \bxi + B_{J_\by} \bt = \bc_{J_\by}. \nonumber
	\end{eqnarray}
\end{replemma}

We first introduce the following lemma.

\begin{lemma}
	\label{lem:minimal_face_relint}
	Suppose that $\Q$ is a convex polyhedron in $\R^{p +n}$ defined as \eqref{eq:Q} and $(\hbxi,\hbt)\in \Q$. Let
	$
	J:=\{1 \leq i \leq m: \langle \ba_i, \hbxi \rangle +  \langle \bb_i, \hbt \rangle =c_i \}.
	$
	Then, $(\hbxi,\hbt)\in\text{relint}(F)$, where
	$
	F=\{(\bxi,\bt) \in\mathbb{R}^{p+n} : A_J\bxi+B_J\bt=\bc_J,~A_J\bxi+B_J\bt\leq\bc_J\}.
	$
\end{lemma}

\begin{proof}[Proof of Lemma \ref{lem:minimal_face_relint}]
	Let $J^c$ be the complement set of $J$, namely, $J^c:=\{1,2,\dots, m\}\backslash J$. By the defining of $J$, we have
	$A_{J^c}\hbxi+B_{J^c}\hbt<\bc_{J^c}$ so that there exists a small enough $\epsilon>0$ such that  $A_{J^c}\bxi+B_{J^c}\bt<\bc_{J^c}$ for any $(\bxi,\bt)\in B_{\epsilon}(\hbxi,\hbt)$. According to Lemma~\ref{lem:aff_hull},
	$$
	\text{aff}(F)=\{(\bxi,\bt)\in\mathbb{R}^{p+n} : A_{J}\bxi+B_{J}\bt=\bc_{J}\}
	$$
	so that $B_\epsilon(\hbxi,\hbt) \cap \operatorname{aff}(F) \subseteq F$. Hence, by definition, $(\hbxi,\hbt)\in \text{relint}(F)$.
\end{proof}

We are now ready to prove Lemma~\ref{lem:div}.
\begin{proof}[Proof of Lemma~\ref{lem:div}]
	Since  $-\bd=A^\top\bu$ for some $\bu\geq\mathbf{0}$ whenever $\lambda=0$ in \eqref{eq:LSE}, the optimization problem in~\eqref{eq:LSE} has a bounded optimal value for any $\by$ according to Lemma~\ref{lem:bounded} and hence $(\hbt(\by), \hbxi(\by)) $ is well-defined.
	
	Before we prove this lemma, we first provide the KKT conditions of the minimization problem \eqref{eq:LSE}. Let $\hat\bu\in \mathbb{R}^m$ be the Lagrange multiplier for the $m$ constraints in \eqref{eq:LSE} and $J_\by$ be as defined in \eqref{eq:J_y_2}. Note that $(\hbt(\by)$, $\hbxi(\by))$ and $\hat\bu$ must satisfy
	\begin{eqnarray}\label{eq:lang_system}
	\hbt(\by)-\by+B_{J_{\by}}^\top\hat\bu_{J_{\by}}=0,&&\lambda\hbxi(\by)+\bd+A_{J_{\by}}^\top\hat\bu_{J_{\by}}=0,\\
	A_{J_{\by}}\hbxi(\by)+B_{J_{\by}}\hbt(\by)=\bc_{J_{\by}},&&A_{J_{\by}^c}\hbxi(\by)+B_{J_{\by}^c}\hbt(\by)\leq \bc_{J_{\by}^c}, \nonumber \\
	\widehat{\bu}_{J_{\by}}\geq\mathbf{0},&& \widehat{\bu}_{J_{\by}^c}=\mathbf{0},\nonumber
	\end{eqnarray}
	where $\widehat{\bu}_{J_{\by}}$ and $\widehat{\bu}_{J_{\by}^c}$ are sub-vectors of $\hat\bu$ indexed by $J_{\by}$ and $J_{\by}^c$, respectively. We prove this lemma in two cases: $\lambda=0$ and $\lambda>0$.
	
	\noindent \textbf{Case 1: $\lambda=0$.} Given any face $F$ of $\Q$, $\Proj_{\bt}(F)+R_{-\bd}(N(F))$ is itself a polyhedron in $\mathbb{R}^n$ so that its boundary $\text{bd}(\Proj_{\bt}(F)+R_{-\bd}(N(F)))$ is a measure zero set in $\mathbb{R}^n$. Since $\Q$ has finitely many faces, the set
	\begin{equation}\label{eq:y_bd_lem2}
	\underset{F\text{ is a face of }\Q}{\bigcup}\text{bd}\bigg(\Proj_{\bt}(F)+R_{-\bd}(N(F))\bigg)
	\end{equation}
	has measure zero in $\mathbb{R}^n$. Therefore, to prove this lemma, it suffices to show that, for any $\by$ not in \eqref{eq:y_bd_lem2}, there is an associated neighborhood $U$ of $\by$ such that $\hbt(\bz) = \widetilde{\bt}(\bz)$ for every $\bz \in U$.
	
	Suppose that $\by$ is not in \eqref{eq:y_bd_lem2}. Let $(\hbxi(\by),\hbt(\by))$ be any solution of~\eqref{eq:LSE}. We consider the face of $\Q$ defined as
	\begin{align}
	\label{eq:Fby}
	F_{\by}=\{(\bxi,\bt) \in\mathbb{R}^{p+n} : A_{J_{\by}}\bxi+B_{J_{\by}}\bt=\bc_{J_{\by}},~A_{J_{\by}^c}\bxi+B_{J_{\by}^c}\bt\leq\bc_{J_{\by}^c}\},
	\end{align}
	where $J_{\by}^c$ is the complement set of $J_{\by}$. According to Lemma~\ref{lem:minimal_face_relint}, we have $(\hbxi(\by),\hbt(\by))\in \text{relint}(F_{\by})$.

	Next we want to show that $\by\in \Proj_{\bt}(F_{\by})+R_{-\bd}(N(F_{\by}))$. Consider the following linear optimization problem
	$$
	\max_{(\bxi, \btheta)\in \Q} \langle -\bd, \bxi \rangle+\langle \by-\hbt(\by), \bt \rangle.
	$$
	Its KKT conditions suggest that $(\bxi, \btheta)$ is its optimal solution if and only if there exists a Lagrange multiplier $\bu\in \mathbb{R}^m$ such that
	\begin{eqnarray}
	\label{eq:KKTLP}
	\bt(\by)-\by+B^\top\bu=0,&&\bd+A^\top\bu=0,\\
	A\bxi+B\bt\leq\bc,&&\bu\geq0\nonumber \\
	(\langle\ba_i,\bxi\rangle+\langle\bb_i,\bt\rangle-c_i)u_i=0,&&\forall \; i=1,2,\dots m.\nonumber
	\end{eqnarray}
	However, according to the KKT conditions \eqref{eq:lang_system} of \eqref{eq:LSE} with $\lambda=0$ and the definition of $J_{\by}$ and $F_{\by}$, if we choose $\bu=\hat\bu$, all the conditions in \eqref{eq:KKTLP} hold for any $(\bxi,\bt)\in F_{\by}$, which imply $F_{\by}\subseteq\argmax_{(\bxi, \btheta)\in \Q} \; \langle -\bd, \bxi \rangle+\langle \by-\hbt(\by), \bt \rangle.$
	From the definition of a normal cone, we have $(-\bd,\by-\hbt(\by))\in N(F_{\by})$, and thus, $\by-\hbt(\by)\in R_{-\bd}(N(F_{\by}))$. Hence, we have $\by=(\by-\hbt(\by))+\hbt(\by)\in \Proj_{\bt}(F_{\by})+R_{-\bd}(N(F_{\by}))$.
	
	Because $\by$ is not in \eqref{eq:y_bd_lem2}, $\Proj_{\bt}(F_{\by})+R_{-\bd}(N(F_{\by}))$ must have a full dimension and contain $\by$ in its interior. Therefore, there exists a neighborhood $U$ of $\by$ contained in $\text{int}(\Proj_{\bt}(F_{\by})+R_{-\bd}(N(F_{\by})))$ such that, for any $\bz \in U$, there exist $(\bar{\bxi}(\bz) ,\bar{\bt}(\bz) ) \in F_{\by}$ with $\bz-\bar{\bt}(\bz)\in R_{-\bd}(N(F_{\by}))$. This follows from the fact that, if $\bz \in U\subset\text{int}(\Proj_{\bt}(F_{\by})+R_{-\bd}(N(F_{\by})))$, $\bz$ can be expressed as $\bz = \bar{\bt}(\bz) + (\bz - \bar{\bt}(\bz))$ where $ \bar{\bt}(\bz) \in \Proj_{\bt}(F_{\by})$ and $\bz - \bar{\bt}(\bz) \in R_{-\bd}(N(F_{\by}))$. Now from the definition of $\Proj_{\bt}(F_{\by})$, there exists $\bar{\bxi}(\bz)$ such that $(\bar{\bxi}(\bz) ,\bar{\bt}(\bz) ) \in F_{\by}$.   If there exist multiple qualified $\bar{\bxi}(\bz)$, we choose the one that minimizes $\|\bar{\bxi}(\bz)-\hbxi(\by)\|_2^2$.
	
	Since $\bz - \bar{\bt}(\bz) \in R_{-\bd}(N(F_{\by}))$, by the definition of $R_{-\bd}(N(F_{\by}))$, we have
	$(-\bd,\bz - \bar{\bt}(\bz)) \in N(F_{\by})$, which further implies
	\[
	F_{\by}\subseteq\argmax_{(\bxi, \btheta)\in \Q} \langle -\bd, \bxi \rangle+\langle \bz - \bar{\bt}(\bz), \bt \rangle,
	\]
	by the definition of $N(F_{\by})$. Since $(\bar{\bxi}(\bz) ,\bar{\bt}(\bz) ) \in F_{\by}$, we have
	\[
	(\bar{\bxi}(\bz) ,\bar{\bt}(\bz) )\in\argmax_{(\bxi, \btheta)\in \Q} \langle -\bd, \bxi \rangle+\langle \bz - \bar{\bt}(\bz), \bt \rangle
	\]
	which is equivalent to $\left\langle-\bd, \bxi\right\rangle+\left\langle\bz-\bar{\bt}(\bz), \bt\right\rangle\leq\left\langle-\bd, \bar{\bxi}(\bz)\right\rangle+\left\langle\bz-\bar{\bt}(\bz), \bar{\bt}(\bz)\right\rangle,$
	for any $(\bxi,\bt)\in \Q$. This  implies $\left\langle\bd, \bxi-\bar{\bxi}(\bz)\right\rangle+\left\langle\bar{\bt}(\bz)-\bz, \bt-\bar{\bt}(\bz)\right\rangle\geq0,
	$ for any $(\bxi,\bt)\in \Q$, which, by the optimality conditions (see e.g.,~\citet[Proposition 4.7.1]{bertsekas03}), shows that $(\bar{\bt}(\bz), \bar{\bxi}(\bz))$ is an optimal solution of \eqref{eq:LSE} with $\lambda=0$.
	
	Due to the uniqueness of the $\bt$-component of the optimal solution of \eqref{eq:LSE}, we have $\hbt(\bz)=\bar{\bt}(\bz)\in\Proj_{\bt}(F_{\by})$ and we can set $\hbxi(\bz)=\bar{\bxi}(\bz)$ as well. Recall the facts that $(\hbxi(\by),\hbt(\by))\in \text{relint}(F_{\by})$, $(\hbxi(\bz),\hbt(\bz))=(\bar{\bxi}(\bz),\bar{\bt}(\bz))\in F_{\by}$, and $\bar{\bxi}(\bz)$ minimizes $\|\bar{\bxi}(\bz)-\hbxi(\by)\|_2^2$ among all qualified $\bar{\bxi}(\bz)$'s. By the continuity of $\bar{\bxi}(\cdot)$ and $\bar{\bt}(\cdot)$,  we can guarantee that $(\hbxi(\bz),\hbt(\bz))\in\text{relint}(F_{\by})$ for any $\bz \in U$, if $U$ is small enough.

	Next, we show that, for all $\bz \in U$,
	\begin{eqnarray}\label{eq:key_lem2_gen}
	\argmin_{(\bxi,\bt)\in \Q } \frac{1}{2}\|\bt-\bz\|_2^2+\bd^\top\bxi & \supseteq & \argmin_{(\bxi,\bt)\in F_{\by} } \frac{1}{2}\|\bt-\bz\|_2^2+\bd^\top\bxi \nonumber\\
	& =& \argmin_{(\bxi,\bt)\in \text{aff}(F_{\by}) } \frac{1}{2}\|\bt-\bz\|_2^2+\bd^\top\bxi.
	\end{eqnarray}
	The first equality of the above display follows from the fact that $(\hbxi(\bz),\hbt(\bz))=(\bar{\bxi}(\bz),\bar{\bt}(\bz))\in F_{\by}$ for any $\bz\in U$. 
	We prove the second equality by contradiction. Suppose that the equality does not hold for some $\bz \in U$. Then, there must exist $(\bxi',\bt')\in  \text{aff}(F_{\by})\backslash F_{\by}$ such that $\frac{1}{2}\|\bt'-\bz\|_2^2+\bd^\top\bxi' <\frac{1}{2}\|\hbt(\bz)-\bz\|_2^2+\bd^\top\hbxi(\bz)$. Because $(\hbxi(\bz),\hbt(\bz))\in \text{relint}(F_{\by})$, there exists a small enough $\alpha>0$ such that $\alpha(\bt', \bxi')+(1-\alpha)(\hbt(\bz), \hbxi(\bz))\in F_{\by}$ and, by convexity,
	\begin{eqnarray*}
		&&\frac{1}{2}\|\alpha\bt'+(1-\alpha)\hbt(\bz)-\bz\|_2^2+\bd^\top(\alpha\bxi'+(1-\alpha)\hbxi(\bz))\\
		&\leq&\alpha\left[\frac{1}{2}\|\bt'-\bz\|_2^2+\bd^\top\bxi'\right]+(1-\alpha)\left[\frac{1}{2}\|\hbt(\bz)-\bz\|_2^2+\bd^\top\hbxi(\bz)\right]\\
		&<&\frac{1}{2}\|\hbt(\bz)-\bz\|_2^2+\bd^\top\hbxi(\bz),
	\end{eqnarray*}
	which leads to a contradiction to the optimality of $(\hbxi(\bz),\hbt(\bz))$ in the first equality in \eqref{eq:key_lem2_gen}. Therefore, we must have $(\hbxi(\bz),\hbt(\bz))\in\argmin_{(\bxi,\bt)\in \text{aff}(F_{\by}) } \frac{1}{2}\|\bt-\bz\|_2^2+\bd^\top\bxi$. Since $\text{aff}(F_{\by})=\{(\bxi,\bt)\in\mathbb{R}^{p+n}:A_{J_{\by}}\bxi+B_{J_{\by}}\bt=\bc_{J_{\by}}\}$ due to Lemma~\ref{lem:aff_hull}, Lemma \ref{lem:div} follows when $\lambda=0$. \newline
	
	\noindent \textbf{Case 2: $\lambda>0$.} Note that it suffices to prove Lemma~\ref{lem:div} in the special case where $\lambda=1$ and $\bd=\mathbf{0}$. The case where $\lambda\neq1$ or $\bd\neq\mathbf{0}$ can be reduced to the case with $\lambda=1$ by letting $\bg= \sqrt{\lambda} \bxi+\bd/\sqrt{\lambda}$ and reformulating the problem~\eqref{eq:LSE} as
	\begin{eqnarray}\label{eq:reform_Lip_multi_conv}
	(\hbt(\by), \hat{\bg}(\by)) & = &  \argmin_{\bt, \bg}\frac{1}{2} \|\bt-\by\|_2^2+\frac{1}{2} \|\bg\|_2^2 \\
	& \text{s.t.} &   \frac{1}{\sqrt{\lambda}}A\bg + B\bt \leq \bc+\frac{1}{\lambda}A\bd. \nonumber
	\end{eqnarray}

	Given any face $F$ of $\Q$, $R(F+N(F))$ is itself a polyhedron in $\mathbb{R}^n$ so that its boundary $\text{bd}(R(F+N(F)))$ is a measure zero set in $\mathbb{R}^n$. Since $\Q$ has finitely many faces, the set
	\begin{equation}\label{eq:y_bd_lem3}
	\underset{F\text{ is a face of }\Q}{\bigcup}\text{bd}\bigg(R(F+N(F))\bigg)
	\end{equation}
	is a measure zero set in $\mathbb{R}^n$. Therefore, to prove Lemma \ref{lem:div} when $\lambda=1$ and $\bd=\mathbf{0}$, it suffices to prove that, for any $\by\in\mathbb{R}^n$ not in the set \eqref{eq:y_bd_lem3}, there is an associated neighborhood $U$ of $\by$ such that for every $\bz \in U$, $ \hbt(\bz) = \widetilde{\bt}(\bz)$.
	
	For $\by$ not in the set \eqref{eq:y_bd_lem3}, let $\hbt(\by)$ and $\hbxi(\by)$ be defined as in \eqref{eq:LSE} and $J_{\by}$ be defined as in \eqref{eq:J_y_2}. We consider a face $F_{\by}$ of $\Q$ defined as in \eqref{eq:Fby}.
	When $\lambda=1$ and $\bd=\mathbf{0}$, \eqref{eq:LSE} represents a projection of $(\textbf{0},\by)$ onto $\Q$. By a similar argument to Case 1 based on the KKT conditions \eqref{eq:lang_system} of \eqref{eq:LSE}, we can show $(\hbxi(\by),\hbt(\by))\in F_{\by}$ and $(-\hbxi(\by),\by-\hbt(\by))\in N(F_{\by})$, which further implies $(\textbf{0},\by)\in F_{\by}+N(F_{\by})$ and $\by\in R(F_{\by}+N(F_{\by}))$.
	

	Because $\by$ is not in \eqref{eq:y_bd_lem3}, $R(F_{\by}+N(F_{\by}))$ must have a full dimension and contain $\by$ in its interior. Therefore, there exists a neighborhood $U$ of $\by$ such that, for every $\bz \in U$, we have $(\textbf{0},\bz)\in F_{\by}+N(F_{\by})$, $(\hbxi(\bz),\hbt(\bz))\in F_{\by}$ and  $(-\hbxi(\bz),\bz-\hbt(\bz))\in N(F_{\by})$.
	
	We claim that $U$ above can be further chosen such that, for every $\bz \in U$, $(\hbxi(\bz),\hbt(\bz))\in \text{relint}(F_{\by})$. If not, there exists a sequence of $\{\bz_k\}_{k \ge 1} \subseteq R(F_{\by}+N(F_{\by}))$ converging to $\by$ but $(\hbxi(\bz_k),\hbt(\bz_k))\in \text{relbd}(F_{\by})$ for all $k$. Because $(\hbxi(\cdot),\hbt(\cdot))$ is a continuous mapping and $\text{relbd}(F_{\by})$ is a closed set, we have $(\hbxi(\by),\hbt(\by))\in \text{relbd}(F_{\by})$, contradicting with the fact that $(\hbxi(\by),\hbt(\by))\in \text{relint}(F_{\by})$. Thus, $(\hbxi(\bz),\hbt(\bz))\in \text{relint}(F_{\by})$ for all $\bz\in U$.
	
	Next we show that for all $\bz\in U$,
	\begin{eqnarray*}
		\argmin_{(\bt, \bxi)\in \Q}\frac{1}{2}\|\bt-\bz\|_2^2 + \frac{1}{2}\|\bxi\|_2^2
		&= &\argmin_{(\bt, \bxi)\in F_{\by} } \frac{1}{2}\|\bt-\bz\|_2^2 + \frac{1}{2}\|\bxi\|_2^2 \\
		&= &\argmin_{(\bt, \bxi)\in \text{aff}(F_{\by}) } \frac{1}{2}\|\bt-\bz\|_2^2 + \frac{1}{2}\|\bxi\|_2^2.
	\end{eqnarray*}
	The first equality holds because $(\hbxi(\bz),\hbt(\bz))\in F_{\by}\subseteq \Q$. Suppose that the second equality does not hold. Then there must exist $(\bt', \bxi')\in  \text{aff}(F_{\by})\backslash F_{\by}$ such that  $\|\bt'-\bz\|_2^2+\|\bxi'\|_2^2 <\|\hbt(\bz)-\bz\|_2^2+\|\hbxi(\bz)\|_2^2$. However, since $(\hbt(\bz), \hbxi(\bz))$ is an interior point of $F_{\by}$, there exists a small enough $\alpha>0$ such that $\alpha(\bt', \bxi')+(1-\alpha)(\hbt(\bz), \hbxi(\bz))\in F_{\by}$  and \[\|\alpha\bt'+(1-\alpha)\hbt(\bz)-\bz\|_2^2+\|\alpha\bxi'+(1-\alpha)\hbxi(\bz)\|_2^2 <\|\hbt(\bz)-\bz\|_2^2+\|\hbxi(\bz)\|_2^2,
	\]
	which leads to a contradiction.
	According to Lemma~\ref{lem:aff_hull}, $\text{aff}(F_{\by})=\{(\bxi,\bt)\in\mathbb{R}^{p+n}:A_{J_\by} \bxi + B_{J_\by} \bt = \bc_{J_\by} \}$,  which means that $(\hbt(\bz), \hbxi(\bz))$ is an optimal solution of \eqref{eq:aff_LSE} when $\lambda=1$ and $\bd=\mathbf{0}$. As a result, $\hbt(\bz)=\widetilde{\bt}(\bz)$ for each $\bz\in U$, by the uniqueness of the optimal solution of \eqref{eq:aff_LSE}. Then Lemma \ref{lem:div} has been proved $\lambda >0$.
\end{proof}

\subsection{Proof of  Theorem~\ref{thm:div}}
\label{sec:supp_theorem_div}

\begin{reptheorem}{thm:div}
	Suppose $-\bd=A^\top\bu$ for some $\bu\geq\mathbf{0}$ whenever $\lambda=0$ in \eqref{eq:LSE}. For any
	$\by \in \R^n$, let $(\hbt(\by), \hbxi(\by)) $ be any solution for \eqref{eq:LSE} and let
	\begin{eqnarray}\label{eq:J_y_2}
	J_{\by}:=\{1 \leq i \leq m: \langle \ba_i, \hat{\bxi}(\by) \rangle +  \langle \bb_i, \hat{\bt}(\by) \rangle =c_i \},
	\end{eqnarray}
	and $A_{J_\by}$ and $B_{J_\by}$ be  the submatrices of $A$ and $B$ with rows in the set $J_\by$. Let $I_{\by} \subseteq J_\by$ be the index set of maximal independent rows of the matrix $[A_{J_\by}, B_{J_\by}]$, i.e., the set of vectors $\{[\ba_i^\top, \bb_i^\top], i \in I_\by\}$ are independent.
	Then, the following statements hold:
	\begin{enumerate}
		\item[(i)] The optimal solution $(\hbt(\by), \hbxi(\by))$ of~\eqref{eq:LSE} has unique components $\hbt(\by)$. The components of $\hat \bt(\by)$ are almost differentiable in $\by$ and  $\nabla \hat \theta_i(\by)$  is an essentially bounded function for each $i=1,\ldots, n$.
		\item[(ii)] For a.e. $\by$,
		{\small \begin{equation}\label{eq:LSE_div}
			D(\by)=\begin{cases} n - \mathrm{trace}\left(B_{I_\by}^\top \left(B_{I_\by}B_{I_\by}^\top+ \frac{1}{\lambda} A_{I_\by} A_{I_\by}^\top\right)^{-1}B_{I_\by} \right), &\text{ if }\lambda>0,\\
			n- |I_\by| + \mathrm{rank}(A_{I_\by}),&\text{ if }\lambda=0,
			\end{cases}
			\end{equation}}
		and $\df(\hat\bt(\by))=\E[D(\by)]$ (note that the index set $I_\by$ is random).
	\end{enumerate}
\end{reptheorem}

For the ease of presentation, we provide the proofs for part (i) and part (ii) of Theorem \ref{thm:div} separately.

\begin{proof}[Proof of Part (i) of  Theorem~\ref{thm:div}]
	Since  $-\bd=A^\top\bu$ for some $\bu\geq\mathbf{0}$ whenever $\lambda=0$ in \eqref{eq:supp_LSE}, the optimization problem in~\eqref{eq:supp_LSE} has a bounded optimal value for any $\by$ according to Lemma~\ref{lem:bounded} so that $(\hbt(\by), \hbxi(\by)) $ is well defined.
	
	The uniqueness of $\hbt(\by)$ can be easily shown via a strong convexity argument. For the simplicity of notations, we define
	$$
	g(\bxi)=\bd^\top\bxi+ \frac{\lambda}{2}\|\bxi\|_2^2.
	$$
	Assume that there are two distinct optimal solutions to~\eqref{eq:supp_LSE},  $(\bt_1(\by), \bxi_1(\by))$ and $(\bt_2(\by), \bxi_2(\by))$. Then, the solution $((\bt_1(\by)+\bt_2(\by))/2$, $(\bxi_1(\by)+\bxi_2(\by))/2)$ is a feasible solution with strictly smaller objective value, i.e.,
	\begin{eqnarray*}
		&&\frac{1}{2}\left\|\frac{\bt_1(\by)+\bt_2(\by)}{2} -\by \right\|_2^2+g\left(\frac{\bxi_1(\by)+\bxi_2(\by)}{2}\right) \\
		&<& \frac{1}{4} \left\|\bt_1(\by)-\by \right\|_2^2+ \frac{1}{2}g(\bxi_1(\by))+\frac{1}{4} \left\|\bt_2(\by)-\by \right\|_2^2+\frac{1}{2}g(\bxi_2(\by)),
	\end{eqnarray*}
	which contradicts the optimality of $(\bt_1(\by), \bxi_1(\by))$ and $(\bt_2(\by), \bxi_2(\by))$.
	
	The almost differentiability of $\hat \bt(\by)$ and the essential boundedness of $\nabla \hat \theta_i$ can be proved by a scheme similar to the proof of Proposition 1 in \citet{MW00}. In particular, it suffices to prove that $\hbt(\by)$ is Lipschitz continuous, namely, $\|\hbt(\by_1)-\hbt(\by_2)\|_2\leq\|\by_1-\by_2\|_2$, which further implies the almost differentiability of $\hat \bt(\by)$ by Rademacher's theorem (\citet{Herbertbook}). According to the optimality condition of~\eqref{eq:supp_LSE}, we have
	\begin{eqnarray*}
		\left\langle\by_1-\hbt(\by_1),\hbt(\by_2)-\hbt(\by_1)\right\rangle-\left\langle\nabla g(\hbxi(\by_1)),\hbxi(\by_2)-\hbxi(\by_1)\right\rangle\leq0,\\
		\left\langle\by_2-\hbt(\by_2),\hbt(\by_1)-\hbt(\by_2)\right\rangle-\left\langle\nabla g(\hbxi(\by_2)),\hbxi(\by_1)-\hbxi(\by_2)\right\rangle\leq0.
	\end{eqnarray*}
	Adding these two inequalities leads to
	\begin{eqnarray*}
		&&\left\langle\by_1-\by_2-(\hbt(\by_1)-\hbt(\by_2)),\hbt(\by_2)-\hbt(\by_1)\right\rangle\\
		&&+\left\langle\nabla g(\hbxi(\by_2))-\nabla g(\hbxi(\by_1)),\hbxi(\by_2)-\hbxi(\by_1)\right\rangle\leq0.
	\end{eqnarray*}
	Since $g(\cdot)$ is convex so that $\nabla g(\cdot)$ is monotone, we have
	$$
	\left\langle\nabla g(\hbxi(\by_2))-\nabla g(\hbxi(\by_1)),\hbxi(\by_2)-\hbxi(\by_1)\right\rangle\geq0
	$$
	which implies
	\begin{eqnarray*}
		\|\hbt(\by_1)-\hbt(\by_2)\|_2^2&\leq&\left\langle\by_2-\by_1,\hbt(\by_2)-\hbt(\by_1)\right\rangle\\
		&\leq&\|\by_2-\by_1\|_2\|\hbt(\by_2)-\hbt(\by_1)\|_2,
	\end{eqnarray*}
	and thus $\|\hbt(\by_1)-\hbt(\by_2)\|_2\leq\|\by_1-\by_2\|$.
\end{proof}

\begin{proof}[Proof of Part (ii) of Theorem~\ref{thm:div}]
	Lemma \ref{lem:div} implies that for a.e.~$\by \in \R^n$, $D(\by)= \nabla_\by \hbt(\by) =  \nabla_\bz  \widetilde{\bt}(\bz)\Big|_{\bz=\by},$
	where $\widetilde{\bt}(\bz)$ is defined in \eqref{eq:aff_LSE}. By the definition of $I_{\by}$, we have $$\{(\bxi,\bt)\in\mathbb{R}^{p+n}:A_{J_{\by}}\bxi+B_{J_{\by}}\bt=\bc_{J_{\by}}\}
	=\{(\bxi,\bt)\in\mathbb{R}^{p+n}:A_{I_{\by}}\bxi+B_{I_{\by}}\bt=\bc_{I_{\by}}\}$$
	so that $(\widetilde{\bt}(\bz), \widetilde{\bxi}(\bz))$ in \eqref{eq:aff_LSE} can be equivalently defined as
	\begin{eqnarray}\label{eq:aff_LSE_I}
	(\widetilde{\bt}(\bz),\widetilde{\bxi}(\bz))  & = & \argmin_{\bt, \bxi}\frac{1}{2} \|\bt-\bz\|_2^2 +\bd^\top\bxi+ \frac{\lambda}{2}\|\bxi\|_2^2\\
	& &\;\;\mathrm{s.t.} \; A_{I_\by} \bxi + B_{I_\by} \bt = \bc_{I_\by}. \nonumber
	\end{eqnarray}
	
	According to the optimality conditions of~\eqref{eq:aff_LSE_I}, there exists a Lagrange multiplier $\widetilde\bu(\bz)\in \mathbb{R}^{|I_{\by}|}$ such that,
	\begin{eqnarray}
	\label{eq:lang_system_eq1}
	\widetilde{\bt}(\bz)-\bz+B_{I_{\by}}^\top\widetilde\bu(\bz)&=&\mathbf{0},\\
	\label{eq:lang_system_eq2}
	\lambda\widetilde{\bxi}(\bz)+\bd+A_{I_{\by}}^\top\widetilde\bu(\bz)&=&\mathbf{0},\\
	\label{eq:lang_system_eq3}
	A_{I_{\by}}\widetilde{\bxi}(\bz)+B_{I_{\by}}\widetilde{\bt}(\bz)&=&\bc_{I_{\by}}.
	\end{eqnarray}
	We then prove the result in two cases: $\lambda=0$ and $\lambda>0$.
	
	\textbf{Case 1: $\lambda=0$.} We define $K$ as a matrix whose columns form a set of basis for the linear space $\mathrm{ker}(A_{I_{\by}}^\top)$ in $\mathbb{R}^{|I_{\by}|}$. Hence, $K$ is a matrix of order $|I_{\by}|\times (|I_{\by}|-\mathrm{rank}(A_{I_{\by}}^\top))$. Because, when $\bz\in U$ (the neighborhood of $\by$), \eqref{eq:aff_LSE_I} has the same objective value as \eqref{eq:LSE} which has a bounded value (according to Lemma~\ref{lem:bounded}), we have $-\bd=A_{I_{\by}}^\top\bar\bu$ for some $\bar\bu$. Note that \eqref{eq:lang_system_eq2} shows that $-\bd=A_{I_{\by}}^\top\widetilde\bu(\bz)$, which implies that $\widetilde\bu(\bz)-\bar\bu\in\mathrm{ker}(A_{I_{\by}}^\top)$. Therefore, there exists $\bv(\bz)\in\mathbb{R}^{|I_{\by}|-\mathrm{rank}(A_{I_{\by}}^\top)}$ such that $\widetilde\bu(\bz)=\bar\bu+K\bv(\bz)$.
	Then, using~\eqref{eq:lang_system_eq1}, we have
	\begin{eqnarray}
	\label{eq:lang_system_eq4}
	\widetilde{\bt}(\bz)=\bz-B_{I_{\by}}^\top(\bar\bu+K\bv(\bz)).
	\end{eqnarray}
	
	From the definition of $K$, multiplying $K^\top$ to both sides of \eqref{eq:lang_system_eq3}, and using the previous display, we have
	\begin{eqnarray}
	\label{eq:lang_system_eq5}
	K^\top\bc_{I_{\by}}&=&K^\top A_{I_{\by}}\widetilde{\bxi}(\bz)+K^\top B_{I_{\by}}\widetilde{\bt}(\bz)\\\nonumber
	&=&K^\top B_{I_{\by}}\widetilde{\bt}(\bz)\\\nonumber
	&=&K^\top B_{I_{\by}}(\bz-B_{I_{\by}}^\top(\bar\bu+K\bv(\bz)))\\\nonumber
	&=&K^\top B_{I_{\by}}\bz-K^\top B_{I_{\by}}B_{I_{\by}}^\top \bar\bu-K^\top B_{I_{\by}}B_{I_{\by}}^\top K\bv(\bz).
	\end{eqnarray}
	
	We claim that $K^\top B_{I_{\by}}B_{I_{\by}}^\top K$ is invertible. Suppose otherwise. Then there exists a non-zero vector $\bar\bv\in\mathbb{R}^{|I_{\by}|-\mathrm{rank}(A_{I_{\by}}^\top)}$ such that $\bar\bv^\top K^\top B_{I_{\by}}B_{I_{\by}}^\top K\bar\bv=0$, which implies $B_{I_{\by}}^\top K\bar\bv=\mathbf{0}$. By the definition of $K$, $A_{I_{\by}}^\top K\bar\bv=\mathbf{0}$ also. Note that $K\bar\bv$ must be non-zero as the columns of $K$ are linearly independent. However, this means that $\bar\bv^\top K^\top[A_{I_\by}, B_{I_\by}]=\mathbf{0}$, contradicting  the fact that $I_{\by}$ is chosen so that the rows of the matrix $[A_{I_\by}, B_{I_\by}]$ are independent. Therefore, $K^\top B_{I_{\by}}B_{I_{\by}}^\top K$ must be invertible so that \eqref{eq:lang_system_eq5} implies
	\begin{align*}
	\bv(\bz)
	=&\left(K^\top B_{I_{\by}}B_{I_{\by}}^\top K\right)^{-1} [K^\top B_{I_{\by}}\by  - K^\top\bc_{I_{\by}}
	- K^\top B_{I_{\by}}B_{I_{\by}}^\top\bar\bu]\label{eq:lang_system_eq6}.
	\end{align*}
	Plugging in $\bv(\bz)$ into \eqref{eq:lang_system_eq4}, we have
	\begin{eqnarray}
	\label{eq:lang_system_eq7}
	\widetilde{\bt}(\bz)=\by-B_{I_{\by}}^\top K\left(K^\top B_{I_{\by}}B_{I_{\by}}^\top K\right)^{-1}K^\top B_{I_{\by}}\bz+\bc',
	\end{eqnarray}
	where $\bc'$ is a constant vector not depending on $\bz$. Therefore,
	\begin{eqnarray*}
		D(\by)= \nabla_\by \hbt(\by) = \nabla_\bz  \widetilde{\bt}(\bz)\Big|_{\bz=\by}
		&=&\text{trace}\left(I_n-B_{I_{\by}}^\top K\left(K^\top B_{I_{\by}}B_{I_{\by}}^\top K\right)^{-1}K^\top B_{I_{\by}}\right)\\
		&=&n-(|I_{\by}|-\mathrm{rank}(A_{I_{\by}}^\top)),
	\end{eqnarray*}
	which completes the proof in this case.
	
	\textbf{Case 2: $\lambda>0$.} In this case, as opposed to the proof of Case 1, 
	we will directly characterize $\nabla_\by \hbt_\lambda(\by)$ by applying the implicit function theorem to the equality system of KKT conditions \eqref{eq:lang_system_eq1}, \eqref{eq:lang_system_eq2} and \eqref{eq:lang_system_eq3}. For the purpose of completeness, we show the implicit function theorem here.

	\begin{lemma}[Implicit function theorem]\label{thm:implicit}
		Let $F:U\rightarrow\mathbb{R}^{n_2}$ be defined in a neighborhood $U\subseteq\mathbb{R}^{n_1+n_2}$ of $(\bu_0,\bv_0)\in\mathbb{R}^{n_1+n_2}$. Suppose that $F$ is continuously differentiable, satisfies $F(\bu_0,\bv_0)=0$, and $\nabla_{\bv} F(\bu_0,\bv_0)$ is an $n_2\times n_2$ invertible matrix. Then there exists a neighborhood $U_{\bu_0}\subseteq\mathbb{R}^{n_1}$ of $\bu_0$ and a continuously differentiable function $f(\bu): \R^{n_1} \rightarrow \R^{n_2}$ such that
		$
		F(\bu,\bv)=0\Longleftrightarrow \bv=f(\bu),
		$
		for any $\bu \in U_{\bu_0}$ and
		\begin{equation}
		\label{eq:der_implicit}
		\nabla f(\bu)=-\left[\nabla_{\bv}F(\bu,f(\bu))\right]^{-1}[\nabla_{\bu}F(\bu,f(\bu))].
		\end{equation}
	\end{lemma}
	
	To characterize the divergence of $\widetilde{\bt}_\lambda(\by)$  we view $(\widetilde{\bt}_\lambda(\by),\widetilde{\bxi}_\lambda(\by))$ and $\by$ in \eqref{eq:aff_LSE} as $\bu$ and $\bv$ in Lemma~\ref{thm:implicit}, respectively, and let $F(\bt,\bxi,\bv)=F(\bu,\bv)=0$ be the KKT conditions of~\eqref{eq:aff_LSE}. Hence, $\widetilde{\bt}_\lambda(\by)$ can be viewed as the implicit function induced by this KKT system whose derivative can be characterized by \eqref{eq:der_implicit}. Note that, we cannot directly apply the implicit function theorem to  the KKT conditions of \eqref{eq:LSE_Lip_multi_conv} because the corresponding KKT conditions involve inequalities and cannot be represented as a system of equalities of the form $F(\bu,\bv)=0$. This shows the necessity of Lemma \ref{lem:div} which establishes the local equivalence between~\eqref{eq:aff_LSE} and~\eqref{eq:supp_LSE}. It is worthy to note that our proof technique of using implicit function theorem   to derive DF can be a general tool with potential applications to other (shape-restricted) regression problems.

	Now, we formally present the proof using Lemma \ref{thm:implicit}. We use $J_1$ and $J_2$ respectively to represent the Jacobian matrices of the equations in the KKT conditions \eqref{eq:lang_system_eq1}, \eqref{eq:lang_system_eq2} and \eqref{eq:lang_system_eq3} with respect to $(\widetilde{\bt}(\bz),\widetilde{\bxi}(\bz),\widetilde\bu(\bz))$  and  with respect to $\bz$. Then, $J_1$ and $J_2$
	have the following forms:
	\begin{eqnarray}\label{eq:Jacob}
	J_1 = \begin{pmatrix}
	I_n & 0  & B_{I_\by}^\top \\
	0   & \lambda I_{p}  & A_{I_\by}^\top \\
	B_{I_\by}  & A_{I_\by}            & 0
	\end{pmatrix},
	\quad
	J_2 = \begin{pmatrix}
	-I_n  \\
	0  \\
	0
	\end{pmatrix}.
	\end{eqnarray}
	Let $\bw=(\widetilde{\bt}(\bz),\widetilde{\bxi}(\bz),\widetilde\bu(\bz)) \in \R^{n+nd+|I|}$. The implicit function theorem implies that
	$$
	\left[\frac{\partial w_i}{\partial z_j}\right]_{ij}=-J_1^{-1}J_2,
	$$
	which further implies that the Jacobian matrix of $\widetilde{\bt}(\by)$ is $-\left( [J_1^{-1}J_2](1:n,1:n) \right)$  and
	\begin{eqnarray}\label{eq:implict_func}
	D(\by)= \nabla_\by \hbt_\lambda(\by) =\nabla_{\bz}\tbt_\lambda(\bz)\Big|_{\bz=\by}=- \mathrm{tr}\left( [J_1^{-1}J_2](1:n,1:n) \right),
	\end{eqnarray}
	where $[J_1^{-1}J_2](1:n,1:n)$ denotes the top-left $n \times n$ sub-matrix of $J_1^{-1}J_2$.
	
	Due to the special structure of $J_1$ in \eqref{eq:Jacob}, its inversion can be computed analytically. In particular, let $D_{I_\by}=B_{I_\by}B_{I_\by}^\top + \frac{1}{\lambda}A_{I_\by}A_{I_\by}^\top$. We note that $D_{I_\by}$ is an invertible matrix since the matrix $[A_{I_\by}, B_{I_\by}]$ has full row rank. The inversion of  $J_1$ takes the following form:
	{\small \begin{eqnarray*}\label{eq:inv_J}
			J_1^{-1}= \begin{pmatrix}
				\begin{pmatrix}
					I_n & 0\\
					0   &  I_{p}/\lambda
				\end{pmatrix}-\begin{pmatrix}
					B_{I_\by}^\top \\ A_{I_\by}^\top /\lambda
				\end{pmatrix}D_I^{-1}\begin{pmatrix}
					B_{I_\by}, A_{I_\by}/\lambda
				\end{pmatrix}  & \begin{pmatrix}
					B_{I_\by}^\top \\ A_{I_\by}^\top /\lambda
				\end{pmatrix} D_I^{-1} \\
				D_I^{-1} \begin{pmatrix}
					B_{I_\by}^\top \\ A_{I_\by}^\top /\lambda
				\end{pmatrix} & 0
			\end{pmatrix}.
	\end{eqnarray*}}
	By plugging in the above formula for the inverse of $J_1$ in~\eqref{eq:implict_func}, we obtain the
	Jacobian matrix of $\widetilde{\bt}(\by)$, which is
	\begin{eqnarray}\label{eq:Jacobianthetay}
	-\left( [J_1^{-1}J_2](1:n,1:n) \right)=I_n-B_{I_\by}^\top \left(B_{I_\by}B_{I_\by}^\top+ \frac{1}{\lambda} A_{I_\by} A_{I_\by}^\top\right)^{-1}B_{I_\by},
	\end{eqnarray}
	and the divergence in \eqref{eq:LSE_div} when $\lambda>0$, which completes the proof.
	
\end{proof}

\subsection{A sanity check for Theorem~\ref{thm:div}}
\label{sec:SanityCheck}
\begin{lemma}\label{lem:SanityCheck}
	$D(\by) \geq 0$ where $D(\by)$ is defined in~\eqref{eq:LSE_div}.
\end{lemma}
\begin{proof}
	Recall the equation \eqref{eq:LSE_div}:
	{\small \begin{equation*}
		D(\by)=\begin{cases} n - \mathrm{trace}\left(B_{I_\by}^\top \left(B_{I_\by}B_{I_\by}^\top+ \frac{1}{\lambda} A_{I_\by} A_{I_\by}^\top\right)^{-1}B_{I_\by} \right), &\text{ if }\lambda>0,\\
		n- |I_\by| + \mathrm{rank}(A_{I_\by}),&\text{ if }\lambda=0.
		\end{cases}
		\end{equation*}
	}
	
	When $\lambda=0$, since $B_{I_\by}$ only has $n$ columns, we have
	$
	|I_\by|=\mathrm{rank}([A_{I_\by}, B_{I_\by}]) \leq n + \mathrm{rank}(A_{I_\by}),
	$
	which implies that $D(\by) \geq 0$. 

	When $\lambda>0$, for any vector $\bx$,
	\begin{eqnarray*}
		&&\bx^\top B_{I_\by}^\top \left(B_{I_\by}B_{I_\by}^\top+ \frac{1}{\lambda} A_{I_\by} A_{I_\by}^\top\right)^{-1}B_{I_\by}\bx\\
		&=&[\mathbf{0}^\top~\bx^\top]\underbrace{\left[
			\begin{array}{c}
				\frac{1}{\sqrt{\lambda}}A_{I_\by}^\top\\
				B_{I_\by}^\top
			\end{array}
			\right] \left(B_{I_\by}B_{I_\by}^\top+ \frac{1}{\lambda} A_{I_\by} A_{I_\by}^\top\right)^{-1}\left[\frac{1}{\sqrt{\lambda}}A_{I_\by}, B_{I_\by}\right]}_{P}\left[
		\begin{array}{c}
			\mathbf{0}\\
			\bx
		\end{array}
		\right]\\
		&\leq&\|\bx\|_2^2,
	\end{eqnarray*}
	where the last inequality holds as $P$ is a projection matrix. This indicates that all the eigenvalues of the matrix $\widetilde{P} := B_{I_\by}^\top \left(B_{I_\by}B_{I_\by}^\top+ \frac{1}{\lambda} A_{I_\by} A_{I_\by}^\top\right)^{-1}B_{I_\by}$ are between 0 and 1, which further implies
	\begin{equation*}
	\mathrm{trace}(\widetilde{P}) \leq \mathrm{rank}(\widetilde{P}) \leq \mathrm{rank}(B_{I_\by})\leq n,
	\end{equation*}
	where the second inequality is due to the well-known fact that for any two matrices $B_1$ and $B_2$, $\mathrm{rank}(B_1 B_2) \leq \min(\mathrm{rank}(B_1), \mathrm{rank}(B_2))$ (note that here $B_1=B_{I_\by}^\top \left(B_{I_\by}B_{I_\by}^\top+ \frac{1}{\lambda} A_{I_\by} A_{I_\by}^\top\right)^{-1}$ and $B_2=B_{I_\by}$). Hence, we have $D(\by)=n-\mathrm{trace}(\widetilde{P}) \geq 0$.
\end{proof}

\subsection{Illustration for Remark \ref{rem:nontrivial}}
\label{sec:supp_nontrivial}

\begin{figure}[!t]
	\centering
	\includegraphics[width=0.8\textwidth, height=6cm]{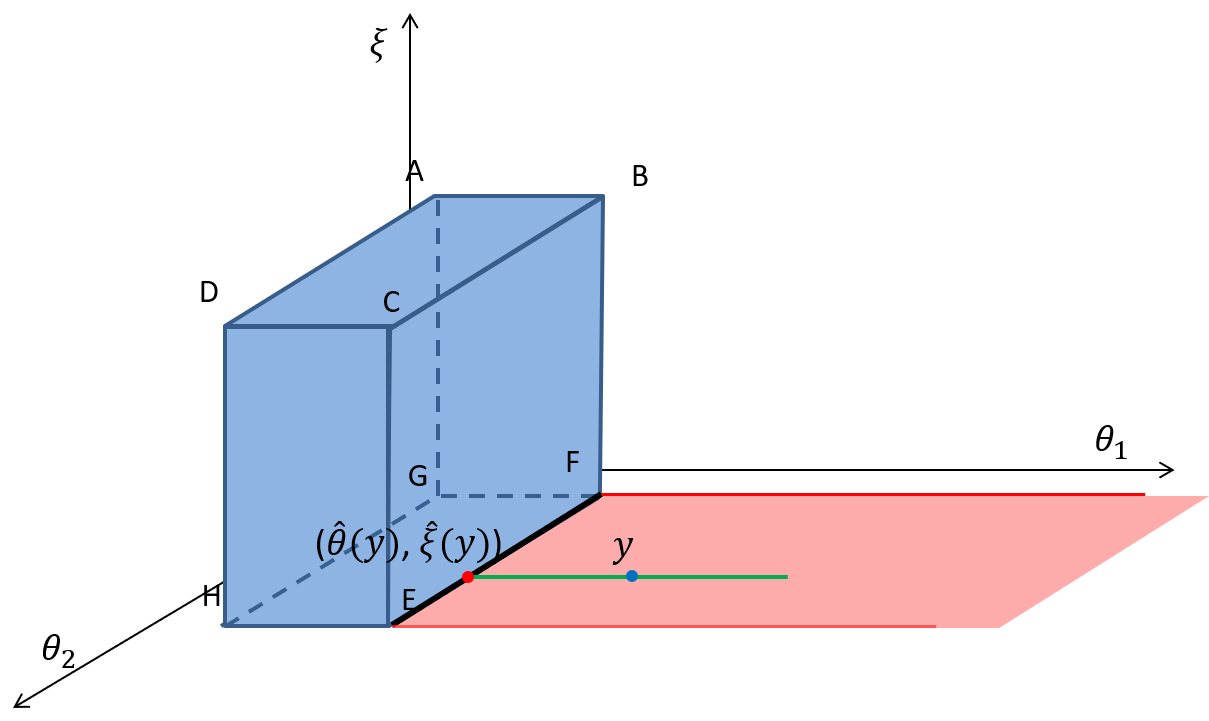}
	\hspace{-3mm}
	\caption{An illustration of the fact that Proposition~\ref{prop:div_poly} does not imply Theorem~\ref{thm:div} with $\lambda>0$. Choose $\lambda=1$ and $\bd=\mathbf{0}$ in~\eqref{eq:LSE} as an example.}
	\label{fig:Lip}
\end{figure}

To better illustrate Remark \ref{rem:nontrivial}, we consider a special case of \eqref{eq:LSE} where $n=2$, $p=1$, $\lambda=1$, $\bd=\mathbf{0}$ and the domain set $\Q=\{(\bt,\bxi) \in \R^{p+n} : A \bxi + B \bt \leq \bc\}$ is the three-dimensional cube $ABCDEFGH$ as illustrated in Figure~\ref{fig:Lip}. In this case, \eqref{eq:LSE} is equivalent to projecting $(\by,-\bd) \in \R^3$ to $\Q$ and the projected point is $(\hbt(\by,\bd), \hbxi(\by,\bd))$. For instance, if $(\by, \mathbf{0})$ is the blue point in Figure~\ref{fig:Lip}, its projection onto $\Q$, $(\hbt(\by, \mathbf{0}), \hbxi(\by, \mathbf{0}))$, is the red point.
According to Lemma 3.2 in \cite{Kato09DF} or the proof of Lemma 2 in \cite{TT12}, $(\hbt(\by,\bd), \hbxi(\by,\bd))$ is a projection onto an affine space in the neighborhood of every $(\by,\bd)$ except a measure-zero set (in $\R^3$).
This measure-zero set consists of the boundary of each subset of $\mathbb{R}^{3}$ that project onto the same face of $\Q$.  For example, the pink area in Figure~\ref{fig:Lip} belongs to the boundary of the set whose projection onto $\Q$ is on the face $BCEF$ so that the pink area belongs to this measure-zero set. Therefore, the mapping $(\hbt(\by,\bd), \hbxi(\by,\bd))$ is no longer a projection onto the same affine space near any point like $(\by, \mathbf{0})$ in this pink area,  which has a positive measure in the space of $\by$ (i.e., $\mathbb{R}^2$). 
In fact, it is easy to verify that $(\hbt(\by,\bd), \hbxi(\by,\bd))$ is not even a differentiable mapping of $(\by,\bd)$ at any point in the pink area (i.e. at the point like $(\by, \mathbf{0})$). As a result, the Jacobian matrix of the estimator $(\hbt, \hbxi)$ is not well-defined and cannot be used to derive the divergence of $\hbt$ with respective to $\by$. On the contrary, when $\bd=\mathbf{0}$ and $(\hbt,\hbxi)$ is viewed as a mapping of only $\by$, it is differentiable at $\by$ in the interior of the pink area in Figure~\ref{fig:Lip}. Hence, the Jacobian matrix of $\hbt$ with respect to $\by$ is well-defined almost everywhere (see \eqref{eq:Jacobianthetay}). Based on this property, we show that \eqref{eq:LSE_div} holds for almost every $\by$ for any given $\bd$.

\section{Proof of Results and Additional Material for Section \ref{sec:DFIso}}
\label{sec:supp_DFIso}

\subsection{Proof of Proposition \ref{prop:uni_bound_iso}}\label{sec:proof_uni_bound_iso}

\begin{repproposition}{prop:uni_bound_iso}
	The bounded isotonic constraint set $\C$ defined in~\eqref{eq:partial_iso_bounded} is a convex polyhedron in the form of \eqref{eq:CvxPoly} where $m=|E|$ and $B\in\mathbb{R}^{|E|\times n}$ is defined as (the rows of $B$ are indexed by the edge set) $\vspace{-0.05in}$
	\begin{eqnarray}\label{eq:A_iso_bounded_partial_order}
	B_{e,i}=\left\{\begin{array}{ll}
	1&\text{ if }e=(i,j)\in E\text{ for some } j\neq i\\
	-1&\text{ if }e=(j,i)\in E\text{ for some } j\neq i\\
	0&\text{ otherwise } \vspace{-0.05in}
	\end{array}
	\right.
	\end{eqnarray}
	and $\bc = (c_e)_{e=1}^{|E|} \in\mathbb{R}^{|E|}$ is defined as
	\begin{eqnarray}\label{eq:b_iso_bounded_partial_order}
	c_e=\left\{\begin{array}{ll}
	\gamma&\text{ if }e=(i,j)\in E\text{ for } i \in \max(V), j \in \min(V)\\
	0&\text{ otherwise}.
	\end{array}
	\right.
	\end{eqnarray}
	Let $B_e$ be the $e$-th row of $B$ and $J_\by:=\{e\in E: B_e \hat \bt_\gamma(\by) = c_e\}$. Further,  let $G_{J_\by}$  be the subgraph of $G$ with the edge set $J_\by$. The divergence of $\hat \bt_\gamma (\by)$ is  the number of connected components of $G_{J_\by}$ for a.e. $\by$, i.e.,
	$D(\by)=\omega(G_{J_\by})$, and therefore $\df(\hat \bt_\gamma(\by))= \E[\omega(G_{J_\by})]$.
\end{repproposition}

\begin{proof}[Proof of Proposition \ref{prop:uni_bound_iso}]
	One key observation is that the matrix $B$  used to define the bounded isotonic constraint set $\C$ in \eqref{eq:A_iso_bounded_partial_order} is the \emph{incidence matrix} of the graph $G$.  Recall that the incidence matrix of a directed graph has one column corresponding to each node  and one row for each edge. If an edge runs from  node $i$ to node $j$, the row corresponding to that edge has $+1$ in column $i$ and $-1$ in column $j$. And it is also straightforward to see that $B_{J_\by}$ is the incidence matrix of the subgraph $G_{J_\by}$.
	
	By Theorem \ref{thm:div} (note $A=0$), $D(\by)=n- |I_\by|= n-\mathrm{rank}(B_{J_\by})$. Since $B_{J_\by}$ is the incidence matrix of the graph $G_{J_\by}$, by a fundamental result from algebraic graph theory (see e.g., Proposition 4.3 from \citet{Biggs94AlgeGraph}), we have $\mathrm{rank}(B_{J_\by})=n-\omega(G_{J_\by})$, where $\omega(G_{J_\by})$ is the number of connected components of $G_{J_\by}$. Therefore, we have $D(\by)=n- \mathrm{rank}(B_{J_\by})= \omega(G_{J_\by}), $ which completes the proof of the proposition.
\end{proof}

\subsection{Proof of Proposition \ref{prop:threshold_gen}}\label{sec:proof_DFIso}

\begin{repproposition}{prop:threshold_gen}
	Let $|U_s|=k_s$ for $s=1,\dots,r$ and $H(L,\gamma)$ be a function on $\mathbb{R}^2$ defined as
	\begin{eqnarray}
	\label{eq:HL_gen}
	H(L,\gamma) :=\sum_{s=1}^rk_s\left(L-\bar\theta_{ s}\right)_+ +
	\sum_{s=1}^rk_s\left(L+\gamma-\bar\theta_{ s}\right)_-,
	\end{eqnarray}
	where $(x)_+=\max\{x,0\}$ and $(x)_-=\min\{x,0\}$. For any given $\gamma$ with $\bar\theta_{ r}-\bar\theta_{ 1}\geq\gamma\geq0$,  $H(L,\gamma)$ is a continuous and strictly increasing function of $L$. Moreover,  $\lim_{L\rightarrow-\infty}H(L,\gamma)=-\infty$ and $\lim_{L\rightarrow+\infty}H(L,\gamma)=+\infty$ so that there exists a unique $L_{\gamma}$ satisfying $H(L_{\gamma},\gamma)=0$. Then, we have
	\begin{equation}\label{eq:thresh}
	\hat{\theta}_{\gamma, i}=\max(L_{\gamma},\min(L_{\gamma}+\gamma,\bar\theta_{ s})),\text{ for all }i\in U_s.
	\end{equation}
	Moreover, $L_{\gamma}$ is non-increasing in $\gamma$.
\end{repproposition}

\begin{proof}[Proof of Proposition \ref{prop:threshold_gen}]
	For the given partial ordered set $\mathcal{X}$ with $n$ elements, the graph induced from the isotonic constraints is denoted by $\widetilde{G}=(V, \widetilde{E})$ where $V=\{1, \ldots, n\}$ and the set of directed edges is $\widetilde{E}=\{(i,j):  x_i \lesssim x_j\}$.  Recall that,  the projection estimator for unbounded isotonic regression, denoted by $\hat \bt=(\hat{\theta}_{1},\dots, \hat{\theta}_{n})^\top$, is obtained by projecting $\by$ onto $\{\bt \in \R^n: \theta_i \leq \theta_j, \; \forall (i,j) \in \widetilde{E}\}$, and the projection estimator for bounded isotonic regression, denoted by $\hat{\bt}_\gamma=(\hat{\theta}_{\gamma,1},\dots, \hat{\theta}_{\gamma,n})^\top$, is obtained by projecting $\by$ onto 
	\[
	\C =  \{\btheta \in \R^n: \theta_i \leq \theta_j  \, \forall \, (i,j)\in \widetilde{E},  \theta_i\leq \theta_j+\gamma,
	i\in \max(V),j\in \min(V)\},
	\]
	where $\max(V)$ and $\min(V)$ are the sets of maximal and minimal elements with respect to the partial order, respectively. It is well known that $\hat\bt$ has a group-constant structure, i.e., there exist disjoint subsets $U_1,U_2,\dots,U_r$ of $V=\{1,\ldots, n\}$ with $|U_s|=k_s$ such that $V=\bigcup_{s=1}^r U_s$
	and $\hat{\theta}_i=\bar\theta_s$ for each $i\in U_s$. Moreover, we assume, without loss of generality, that $r>1$ and   $\bar\theta_1<\bar\theta_2<\dots<\bar\theta_r$.

	Let $(x)_+=\max\{x,0\}$ and $(x)_-=\min\{x,0\}$. We define
	\begin{eqnarray}
	\label{eq:HL_gen_supp}
	H(L,\gamma) :=\sum_{s=1}^rk_s\left(L-\bar\theta_{ s}\right)_+ +
	\sum_{s=1}^rk_s\left(L+\gamma-\bar\theta_{ s}\right)_-.
	\end{eqnarray}
	We first show that, for any $\gamma$ such that $\bar\theta_{ r}-\bar\theta_{ 1}\geq\gamma\geq 0$, there exists an unique $L_{\gamma}$ such that
	\begin{eqnarray}
	\label{eq:threshold1_gen}
	H(L_{\gamma},\gamma)
	=0.
	\end{eqnarray}
	For any $\gamma \le \bar\theta_{ r}-\bar\theta_{ 1}$, it is easy to see that $H(L,\gamma)$ is a continuous, non-decreasing and piecewise linear function of $L$. If $H(L,\gamma)$ is not strictly increasing, there must exist $L^1<L^2$ such that $H(L^1,\gamma)=H(L^2,\gamma)$. This means that $H(L,\gamma)$ is a constant on the interval $[L^1, L^2]$, which further implies from the definition of the function in~\eqref{eq:HL_gen_supp} that
	$$
	\bar \theta_r-\gamma\leq L^1< L^2 \leq \bar \theta_1.
	$$
	This contradicts with the fact that $\bar\theta_{ r}-\bar\theta_{ 1}\geq\gamma$. Hence, $H(L,\gamma)$ is strictly increasing in function of $L$. Since $\lim_{L\rightarrow-\infty}H(L,\gamma)=-\infty$ and $\lim_{L\rightarrow+\infty}H(L,\gamma)=+\infty$, there exists an unique $L_{\gamma}$ satisfies $H(L_{\gamma},\gamma)=0$.
	
	Next we show that this $ L_{\gamma}$ is a non-increasing function of $\gamma$. If not, there exist $\gamma_1$ and $\gamma_2$ such that $\bar\theta_{ r}-\bar\theta_{ 1}\geq\gamma_2>\gamma_1\geq 0$ and $L_{\gamma_2}>L_{\gamma_1}$. By the definitions of $L_{\gamma_1}$ and $L_{\gamma_2}$, we have
	\begin{eqnarray*}
		0&=&H(L_{\gamma_2},\gamma_2)\\
		&=&\sum_{s=1}^rk_s\left(L_{\gamma_2}-\bar\theta_{ s}\right)_+ +
		\sum_{s=1}^rk_s\left(L_{\gamma_2}+\gamma_2-\bar\theta_{ s}\right)_-\\
		&>&\sum_{s=1}^rk_s\left(L_{\gamma_1}-\bar\theta_{ s}\right)_+ +
		\sum_{s=1}^rk_s\left(L_{\gamma_1}+\gamma_2-\bar\theta_{ s}\right)_-\\
		&\geq&\sum_{s=1}^rk_s\left(L_{\gamma_1}-\bar\theta_{ s}\right)_+ +
		\sum_{s=1}^rk_s\left(L_{\gamma_1}+\gamma_1-\bar\theta_{ s}\right)_-\\
		&=&0,
	\end{eqnarray*}
	where the first inequality holds because $H(L,\gamma)$ is strictly increasing in $L$ and the second inequality holds because $H(L,\gamma)$ is non-decreasing in $\gamma$. This contradiction indicates that $ L_{\gamma}$ is a non-increasing function of $\gamma$.
	
	For each node $i\in V$, we denote the set of  successors and the set of predecessors  of $i$ in the partial order by
	$$
	\widetilde n^+(i) :=\{j\in V : (i,j)\in \widetilde E)\}\quad\text{ and }\quad
	\widetilde n^-(i):=\{j\in V : (j,i)\in \widetilde E\}.
	$$

	According to the KKT conditions of isotonic regression, for $e=(i,j)\in \widetilde E$, there exists a dual variable $u_{ij}\geq0$ for the constraint $\theta_i\leq\theta_j$ such that
	\begin{eqnarray}\label{eq:iso_first_order_gen}
	\hat\theta_{ i}-y_i+\sum_{j\in \widetilde n^+(i)}u_{ij}-\sum_{j\in \widetilde n^-(i)}u_{ji}=0,\quad\forall \;  i \in V,
	\end{eqnarray}
	and
	\begin{eqnarray}\label{eq:iso_complementary}
	u_{ij}(\hat\theta_{ i}-\hat\theta_{ j})=0,\quad\forall\; (i,j)\in \widetilde E.
	\end{eqnarray}
	Moreover, for any $(i,j)\in \widetilde E$ such that $i\in U_t$ and $j\in U_s$ and $\bar\theta_{ t}<\bar\theta_{ s}$, we have $\hat\theta_{ i}<\hat\theta_{ j}$, and thus, $u_{ij}=0$.
	
	We expand the graph $\widetilde G$ to $G=(V, E)$ where $ E=\widetilde E\cup \{(i,j): i\in \max(V),j\in \min(V)\}$ and define
	$$
	n^+(i) :=\{j\in V : (i,j)\in  E\}\quad\text{ and }\quad
	n^-(i) :=\{j\in V : (j,i)\in  E\}.
	$$
	
	Similarly, according to the KKT conditions of bounded isotonic regression, for $e=(i,j)\in E$, there exists a dual variable $u_{\gamma,ij}\geq0$ for the constraint either $\theta_i\leq\theta_j$ or $\theta_i\leq\theta_j+\gamma$ such that
	\begin{eqnarray}\label{eq:iso_bounded_first_order_gen}
	\qquad \hat\theta_{\gamma, i}-y_i+\sum_{j\in n^+(i)}u_{\gamma,ij}-\sum_{j\in n^-(i)}u_{\gamma,ji} &=&0,\quad\forall \; i\in V, \\\label{eq:iso_bounded_gen}
	u_{\gamma,ij}(\hat\theta_{\gamma, i}-\hat\theta_{\gamma, j})&=&0,\;\forall \; (i,j)\in \widetilde E,\\
	\label{eq:iso_bounded_complementary2}
	\qquad u_{\gamma,ij}(\hat\theta_{\gamma,i}-\hat\theta_{\gamma,j}-\gamma)&=&0,\;\forall \; i\in \max(V), j\in \min(V).
	\end{eqnarray}
	To show that $\hat\btheta_{\gamma}$ defined by
	\begin{equation}\label{eq:thresh_supp}
	\hat{\theta}_{\gamma, i}=\max(L_{\gamma},\min(L_{\gamma}+\gamma,\bar\theta_{ s})),\text{ for }i\in U_s
	\end{equation}
	is the optimal solution for bounded isotonic regression, it suffices to construct a non-negative value for each dual variables $u_{\gamma,ij}$ for $e=(i,j)\in E$, which satisfy the conditions \eqref{eq:iso_bounded_first_order_gen}, \eqref{eq:iso_bounded_gen}, and \eqref{eq:iso_bounded_complementary2} together with $\hat\btheta_{\gamma}$ defined by \eqref{eq:thresh_supp}.
	
	We will do this by solving a \emph{transportation problem}, which is a classical problem in operations research (see, e.g., \citep[Chapter 14]{Dantzig:59}). In a transportation problem,  some demands and supplies of a product are located in different nodes of a (directed) graph and we need to determine a transportation plan that sends products from the supply nodes along the arcs to meet the demands in the demand nodes.
	
	To construct the transportation problem, we consider a directed graph $\hat G=(\hat V, \hat E)$ with
	\begin{eqnarray*}
		\hat V:=\{i\in V: \hat\theta_i\leq L_\gamma\text{ or }\hat\theta_i\geq L_\gamma+\gamma \},
	\end{eqnarray*}
	and
	\begin{eqnarray*}
		\hat E:=\{(i,j)\in E: i\in \hat V\text{ and }j\in\hat V\} \backslash \{(i,j)\in E:  \hat\theta_i\leq L_\gamma \text{ and  } \hat\theta_j\geq L_\gamma+\gamma \}
	\end{eqnarray*}
	where $L_\gamma$ is the unique value satisfying \eqref{eq:threshold1_gen}. Note that $\hat G$ is a subgraph of $G$ containing the arcs in $E$ whose both ends are in $\hat V\subset V$. We also
	define
	$$
	\hat n^+(i):=\{j\in V : (i,j)\in  \hat E\}\quad\text{ and }\quad
	\hat n^-(i):=\{j\in V : (j,i)\in  \hat E\}.
	$$

	We claim there is at least one node $i\in \hat V$ with $\hat\theta_i\leq L_\gamma$. If not, since $\bar\theta_{ r}-\bar\theta_{ 1}\geq\gamma$, we will have $L_\gamma<\bar\theta_1\leq\bar\theta_r-\gamma$ so that $L_\gamma+\gamma<\bar\theta_r$. As a result, we have $H(L_\gamma,\gamma)<0$ contradicting~\eqref{eq:threshold1_gen}. Similarly, we can show there is at least one node $i\in \hat V$ with $\hat\theta_i\geq L_\gamma+\gamma$.
	
	Then, to each node $i\in \hat V$ with $\hat\theta_i\leq L_\gamma$, we assign a \emph{demand} of $L_\gamma-\hat\theta_i\geq0$. To each node $i\in \hat V$ with $\hat\theta_i\geq L_\gamma+\gamma$, we assign a \emph{supply} of $\hat\theta_i-L_\gamma-\gamma\geq0$. The decision variables of the transportation problem is denoted by $\delta_{ij}\geq 0$, for each $(i,j)\in\hat E$, which represents the amount of products shipped from node $i$ to node $j$ along arc $(i,j)$. To find a shipping plan so that the demands are satisfied by the supplies, we want to find $\delta_{ij}$'s to satisfy the following \emph{flow-balance} constraints
	\begin{eqnarray}\label{eq:flowbalance1}
	&&L_\gamma-\hat\theta_i+\sum_{j\in \hat n^+(i)}\delta_{ij}-\sum_{j\in \hat n^-(i)}\delta_{ji}=0,\quad\text{ for } i\in \hat V, ~\hat\theta_i\leq L_\gamma\\\label{eq:flowbalance2}
	&&L_\gamma+\gamma-\hat\theta_i+\sum_{j\in \hat n^+(i)}\delta_{ij}-\sum_{j\in \hat n^-(i)}\delta_{ji}=0,\quad\text{ for } i\in \hat V, ~\hat\theta_i\geq L_\gamma+\gamma.
	\end{eqnarray}
	The constraint \eqref{eq:flowbalance1} means, for a demand node, the total amount of in-flow minus the total amount of out-flow must equal its demand. The constraint \eqref{eq:flowbalance2} means, for a supply node, the total amount of out-flow minus the total amount of in-flow must equal its supply.
	
	Then, we show that there exist $\delta_{ij}\geq0$ such that \eqref{eq:flowbalance1} and \eqref{eq:flowbalance2} hold by the following three observations.
	
	\begin{itemize}
		\item First, we note that the total demand is
		$$
		\sum_{i\in \hat V,~\hat\theta_i\leq L_\gamma }\left(L_\gamma-\hat\theta_{ i}\right)
		=\sum_{s=1}^rk_s\left(L_\gamma-\bar\theta_{ s}\right)_+
		$$
		and the total supply is
		$$
		\sum_{i\in \hat V,~\hat\theta_i\geq L_\gamma+\gamma}\left(\hat\theta_{ i}-L_\gamma-\gamma\right)
		=-\sum_{s=1}^rk_s\left(L_\gamma+\gamma-\bar\theta_{ s}\right)_-.
		$$
		Since  $L_\gamma$ satisfies \eqref{eq:threshold1_gen}, the total demand above equals the total supply.
		
		\item Second, given any $i\in \hat V$ with $\hat\theta_i\geq L_\gamma+\gamma$, let $j$ be a successor of $i$ in the partial order. Then, $j$ must be in $\hat V$ also because $\hat\theta_j\geq\hat\theta_i\geq L_\gamma+\gamma$. As a result, $\max(V)\bigcap\hat V\neq\emptyset$ and there must exist a directed path in $\hat G$ from each node $i\in \hat V$ with $\hat\theta_i\geq L_\gamma+\gamma$ to a maximal node in $\max(V)\bigcap\hat V$. Similarly, we can show that  $\min(V)\bigcap\hat V\neq\emptyset$ and there must exist a directed path in $\hat G$ from a minimal node in $\min(V)\bigcap\hat V$ to each node $i\in \hat V$ with $\hat\theta_i\leq L_\gamma$.
		
		\item Third, by definition, $\hat G$ contains every arc from a node in $\max(V)\bigcap\hat V$ to a node in $\min(V)\bigcap\hat V$.
	\end{itemize}
	
	By these three observations above, there always exist a shipping plan that exactly matches supplies to demands in all nodes in $\hat G$. Therefore, there exist $\delta_{ij}\geq0$ satisfying \eqref{eq:flowbalance1} and \eqref{eq:flowbalance2}.
	
	Next we construct dual variables $u_{\gamma,ij}$ for $e=(i,j)\in E$ that satisfy the conditions \eqref{eq:iso_bounded_first_order_gen}, \eqref{eq:iso_bounded_gen}, and \eqref{eq:iso_bounded_complementary2} together with $\hat\btheta_{\gamma}$ defined by \eqref{eq:thresh_supp} as follows:
	\begin{eqnarray}\label{eq:dual_iso_bounded1}
	&&u_{\gamma,ij}=u_{ij},\quad\text{ for } (i,j)\in \widetilde E\backslash \hat E, \\\label{eq:dual_iso_bounded2}
	&&u_{\gamma,ij}=0,\quad\text{ for }  i\in \max(V),~j\in \min(V),~ (i,j)\notin \hat E,\\\label{eq:dual_iso_bounded3}
	&&u_{\gamma,ij}=u_{ij}+\delta_{ij},\quad\text{ for } (i,j)\in \widetilde E\cap\hat E,\\\label{eq:dual_iso_bounded4}
	&&u_{\gamma,ij}=\delta_{ij},\quad\text{ for }  i\in \max(V),~j\in \min(V),~ (i,j)\in \hat E.
	\end{eqnarray}
	
	We can easily see that all $u_{\gamma,ij}$'s defined as above are non-negative.
	
	First, we show that \eqref{eq:iso_bounded_first_order_gen} holds. For $i\in V\backslash\hat V$, we have $\hat\theta_{\gamma,i}=\hat\theta_{i}$ according to \eqref{eq:thresh_supp}, which further implies \eqref{eq:iso_bounded_first_order_gen} together with \eqref{eq:dual_iso_bounded1}, \eqref{eq:dual_iso_bounded2} and \eqref{eq:iso_first_order_gen}. For $i\in \hat V$ with $\hat\theta_i\leq L_\gamma$, we have $\hat\theta_{\gamma, i}=L_\gamma$ and summing \eqref{eq:iso_first_order_gen} and \eqref{eq:flowbalance1} yields \eqref{eq:iso_bounded_first_order_gen}. For $i\in \hat V$ with $\hat\theta_i\geq L_\gamma+\gamma$, we have $\hat\theta_{\gamma, i}=L_\gamma+\gamma$  and  summing \eqref{eq:iso_first_order_gen} and \eqref{eq:flowbalance2} yields \eqref{eq:iso_bounded_first_order_gen}.
	
	Second, we show that \eqref{eq:iso_bounded_gen} holds. It suffices to prove that $u_{\gamma,ij}=0$ for $(i,j)\in \widetilde E$ such that $\hat\theta_{\gamma, i}<\hat\theta_{\gamma, j}$, which can only happen when $(i,j)\in \widetilde E\backslash \hat E$ (note that when $(i,j) \in \widetilde{E}\cap \hat E$, we must have either $\hat\theta_{\gamma, i}=\hat\theta_{\gamma, j}=L_\gamma$ or  $\hat\theta_{\gamma, i}=\hat\theta_{\gamma, j}=L_\gamma+\gamma$). In this case, we have $\hat\theta_{\gamma, i}=\hat\theta_{i}<\hat\theta_{j}=\hat\theta_{\gamma, j}$. By \eqref{eq:dual_iso_bounded1} and \eqref{eq:iso_complementary},  \eqref{eq:iso_bounded_gen} holds.
	
	Third, we show that \eqref{eq:iso_bounded_complementary2} holds. It suffices to prove that $u_{\gamma,ij}=0$ for $i\in \max(V)$ and $j\in \min(V)$ such that $\hat\theta_{\gamma, i}<\hat\theta_{\gamma, j}+\gamma$, which can only happen when  $i\in \max(V)$, $j\in \min(V)$ and $(i,j)\notin \hat E$. In this case, \eqref{eq:iso_bounded_complementary2} is implied by \eqref{eq:dual_iso_bounded2}.
	
	Then, all the KKT conditions are satisfied by $\hat\btheta_{\gamma}$ given in \eqref{eq:thresh_supp} and the dual variables defined in \eqref{eq:dual_iso_bounded1}, \eqref{eq:dual_iso_bounded2}, \eqref{eq:dual_iso_bounded3} and \eqref{eq:dual_iso_bounded4}. Hence, such a $\hat\btheta_{\gamma}$ is an optimal solution for bounded isotonic regression.
\end{proof}

\subsection{Proof of  Theorem \ref{thm:iso_bounded_monotone}} \label{sec:supp_monotone}

\begin{reptheorem}{thm:iso_bounded_monotone}
	For any given $\by \in \R^n$ the divergence of $\hat \bt_{\gamma} (\by)$ is nondecreasing in $\gamma$. This implies that $\df(\hat\bt_\gamma(\by))$ is nondecreasing in $\gamma$.
\end{reptheorem}

\begin{proof}[Proof of Theorem \ref{thm:iso_bounded_monotone}]
	According to Proposition \ref{prop:threshold_gen}, when $\bar\theta_{ r}-\bar\theta_{ 1}\geq\gamma\geq0$, we have
	$$
	\hat\theta_{\gamma, i}=\max(L_{\gamma},\min(L_{\gamma}+\gamma,\bar\theta_{ s})),\text{ for } i \in U_s
	$$
	where $L_{\gamma}$ is non-increasing in $\gamma$. Therefore, the number of connected component is  non-decreasing in $\gamma$; so is the divergence of $\hat \bt_{\gamma} (\by)$. For $\gamma\geq\bar\theta_{ r}-\bar\theta_{ 1}$, we have $\hat \bt_{\gamma} (\by)=\hat \bt(\by)$, i.e., the solution of the unbounded isotonic regression and the bounded isotonic regression are identical. Therefore, the number of connected component and the divergence of $\hat \bt_{\gamma} (\by)$ is a constant when $\gamma\geq\bar\theta_{ r}-\bar\theta_{ 1}$. Combining the above two cases on $\gamma$ completes the proof of the theorem. 
\end{proof}

\section{Proof of Results in Section \ref{sec:other_app}}
\label{sec:proof_other_app}

\subsection{DF for additive models}
\label{sec:proof_additive}
\begin{repproposition}{prop:fused_add}
	For the estimator $\hbt(\by)=\sum_{j=1}^d \hbt_j(\by) + \hat\theta_0(\by) \mathbf{1}$ in  \eqref{eq:fused_add}, the divergence of $\hbt(\by)$ is,
	\[
	D(\by)=\mathrm{dim}(\text{span}\{\mathbf{1}_{n\times1},\mathrm{ker}(K_1),\dots,\mathrm{ker}(K_d)\}),
	\]
	where, for $j=1,\dots,d$,  $K_j=\begin{pmatrix}Q_0^j\\\mathbf{1}_{1\times n}\end{pmatrix}$,
	$Q_0^j$ is the sub-matrix of $Q_j$ consisting of each row $\mathbf{q}_{ji}$ ($1\leq i\leq n_j$) of $D_jP_j$ such that $\mathbf{q}_{ji}^\top\hbt_j(\by)=0$ and
	$\mathrm{ker}(K_j) :=\{\bx\in\mathbb{R}^n : Q_0^j\bx=\mathbf{0}\text{ and }\mathbf{1}_{1\times n}\bx=0\}$ is the kernel of $K_j=\begin{pmatrix}Q_0^j\\\mathbf{1}_{1\times n}\end{pmatrix}$. The DF $\mathrm{df} (\hbt(\by))=\E(D(\by))$.
\end{repproposition}

\begin{proof}[Proof of Proposition~\ref{prop:fused_add}]
	Letting $\bt=\sum_{j=1}^d \bt_j + \theta_0 \mathbf{1} $ and $Q_j=D_jP_j\in\mathbb{R}^{n_j\times n}$ for $j=1,\dots,d$,
	the formulation in \eqref{eq:fused_add} is further equivalent to
	\begin{eqnarray}\label{eq:lasso_gen_new}
	(\hbt(\by), \{\hbt_j(\by)\}_{j=1}^d,\hat\theta_0(\by),\{\hat\bg_j(\by)\}_{j=1}^d) & \in &\argmin_{\bt, \bt_j,\theta_0,\bg_j } \frac{1}{2}\|\bt-\by\|_2^2+\sum_{j=1}^d\tau\mathbf{1}^\top\bg_j \\
	& & \;\;\;  \text{s.t.}  \;\; \bt-\sum_{j=1}^d \bt_j - \theta_0 \mathbf{1} \leq \mathbf{0}, \;\; -\bt+\sum_{j=1}^d \bt_j + \theta_0 \mathbf{1}  \leq \mathbf{0} \nonumber\\
	& & \;\;\;  \;\;\;\;   \;\; \;Q_j\bt_j-\bg_j  \leq \mathbf{0}, \;\; -Q_j\bt_j-\bg_j \leq \mathbf{0}\nonumber\\
	& & \;\;\;  \;\;\;\;   \;\; \;\mathbf{1}^T  \bt_j\leq0, \;\; -\mathbf{1}^T  \bt_j\leq0,\quad 1 \leq j \leq d. \nonumber
	\end{eqnarray}
	To facilitate our reformulation, we denote by $\otimes$ the Kronecker product between two matrices and let $N=\sum_{j=1}^dn_j$ and  $Q\in\mathbb{R}^{N\times nd}$ defined as
	{\small
		\begin{eqnarray}
		\label{eq:Qdiag}
		Q= \begin{pmatrix}
		Q_1& & &\\
		&Q_2 & &\\
		&  &\ddots &\\
		&  &&Q_d
		\end{pmatrix}.
		\end{eqnarray}
	}
	By setting $\bxi=(\bt_1^\top,\dots,\bt_d^\top,\theta_0,\bg_1^\top,\dots,\bg_d^\top)^\top$, the optimization problem in~\eqref{eq:lasso_gen_new} is a special case of \eqref{eq:LSE} with
	{\small
		\begin{equation}\label{eq:genLasso_A_B1}
		\bd=(\mathbf{0}_{1 \times (nd+1)}, \tau \mathbf{1}_{1 \times N})^T, \; \lambda=0,
		\end{equation}
		\begin{equation}\label{eq:genLasso_A_B2}
		A=\begin{pmatrix}
		\mathbf{1}_{1\times d}\otimes I_{n}&  \mathbf{1}_{n\times1} &\mathbf{0}_{n \times N}\\
		-\mathbf{1}_{1\times d}\otimes I_{n}&  -\mathbf{1}_{n\times1} &\mathbf{0}_{n \times N}\\
		Q&\mathbf{0}_{N\times 1}&-I_{N}\\
		-Q&\mathbf{0}_{N\times 1}&-I_{N}\\
		I_{d}\otimes\mathbf{1}_{1\times n}&\mathbf{0}_{d\times 1}&\mathbf{0}_{d \times N}\\
		-I_{d}\otimes\mathbf{1}_{1\times n}&\mathbf{0}_{d\times 1}&\mathbf{0}_{d \times N}
		\end{pmatrix},\;
		B=\begin{pmatrix}
		-I_{n}  \\
		I_{n}  \\
		\mathbf{0}_{N \times n}   \\
		\mathbf{0}_{N \times n}\\
		\mathbf{0}_{d \times n}\\
		\mathbf{0}_{d \times n}
		\end{pmatrix},\;
		\bc=\mathbf{0}.
		\end{equation}
	}
	For each $j$, let $\{1,2,\dots,n_j\}$ be the sets of indexes of the rows of $Q_j$ and $\mathbf{q}_{ji}^\top$ be the $i$-th row of $Q_j$ for $i=1,\dots,n_j$. In addition, let $\gamma_{ji}$ be the $i$-th component of $\bg_j$ for $i=1,\dots,n_j$.
	We partition the set $\{1,2,\dots,n_j\}$ into three sets of indexes as:
	\begin{eqnarray}
	\label{eq:Iindex}
	I_+^j := \{i:\mathbf{q}_{ji}^\top\hbt_j(\by)>0\}, \quad I_-^j :=  \{i:\mathbf{q}_{ji}^\top\hbt_j(\by)<0\}, \quad I_0^j := \{i:\mathbf{q}_{ji}^\top\hbt_j(\by)=0\}.
	\end{eqnarray}
	According to the constraints $Q_j\bt_j-\bg_j  \leq \mathbf{0}$ and $-Q_j\bt_j-\bg_j \leq \mathbf{0}$ in \eqref{eq:lasso_gen_new}, the optimality of $\hat\gamma_{ji}(\by)$ will ensure $\hat\gamma_{ji}(\by)=\max(\mathbf{q}_{ji}^\top\hbt_j(\by),-\mathbf{q}_{ji}^\top\hbt_j(\by))$, which implies that $\mathbf{q}_{ji}^\top\hbt_j(\by)-\hat\gamma_{ji}(\by)=0$ for $i\in I_+^j\cup I_0^j$ and $-\mathbf{q}_{ji}^\top\hbt_j(\by)-\hat\gamma_{ji}(\by)=0$ for $i\in I_-^j\cup I_0^j$.
	
	We define $Q_+^j$, $Q_-^j$ and $Q_0^j$ as the sub-matrices of $Q_j$ consisting of the rows of $Q_j$ indexed by $I_+^j$, $I_-^j$ and $I_0^j$, respectively. By ordering
	$$
	\bxi=(\bt_1^\top,\dots,\bt_d^\top,\theta_0,\bg_1^\top,\dots,\bg_d^\top)^\top
	=(\bt_1^\top,\dots,\bt_d^\top,\theta_0,\bg_{1I_+^1}^\top,\dots,\bg_{dI_+^d}^\top,\bg_{1I_-^1}^\top,\dots,\bg_{dI_-^d}^\top,\bg_{1I_0^1}^\top,\dots,\bg_{dI_0^d}^\top)^\top,
	$$
	we can represent the matrices $A_{J_\by}$ and $B_{J_\by}$ as
	{\small\begin{equation*}
		A_{J_\by}=\begin{pmatrix}
		\mathbf{1}_{1\times d}\otimes I_{n}&  \mathbf{1}_{n\times1} &\mathbf{0}&\mathbf{0}&\mathbf{0}\\
		-\mathbf{1}_{1\times d}\otimes I_{n}&  -\mathbf{1}_{n\times1} &\mathbf{0}&\mathbf{0}&\mathbf{0}\\
		Q_+&\mathbf{0}&-I &\mathbf{0}&\mathbf{0} \\
		-Q_-&\mathbf{0}&\mathbf{0}&-I&\mathbf{0}\\
		Q_0&\mathbf{0}&\mathbf{0}&\mathbf{0}&-I\\
		-Q_0&\mathbf{0}&\mathbf{0}&\mathbf{0}&-I\\
		I_{d}\otimes\mathbf{1}_{1\times n}&\mathbf{0}&\mathbf{0}&\mathbf{0}&\mathbf{0}\\
		-I_{d}\otimes\mathbf{1}_{1\times n}&\mathbf{0}&\mathbf{0}&\mathbf{0}&\mathbf{0}
		\end{pmatrix},\;
		B_{J_\by}=\begin{pmatrix}
		-I_{n}  \\
		I_{n}  \\
		\mathbf{0}\\
		\mathbf{0}\\
		\mathbf{0}\\
		\mathbf{0}\\
		\mathbf{0}\\
		\mathbf{0}
		\end{pmatrix},
		\end{equation*}
	}
	where
	{\small
		\begin{eqnarray}
		\label{eq:threeQmatrix}
		Q_+=
		\begin{pmatrix}
		Q_+^1& & &\\
		&Q_+^2 & &\\
		&  &\ddots &\\
		&  &&Q_+^d
		\end{pmatrix},\;
		Q_-=
		\begin{pmatrix}
		Q_-^1& & &\\
		&Q_-^2 & &\\
		&  &\ddots &\\
		&  &&Q_-^d
		\end{pmatrix},\;
		Q_0=
		\begin{pmatrix}
		Q_0^1& & &\\
		&Q_0^2 & &\\
		&  &\ddots &\\
		&  &&Q_0^d
		\end{pmatrix}.
		\end{eqnarray}
	}
	Let $\hat Q_0^j$ for $j=1,\dots,d$ be the sub-matrix of {\small$\begin{pmatrix}Q_0^j\\\mathbf{1}_{n\times1} \end{pmatrix}$} that contains the maximum number of linearly independent rows of {\small$\begin{pmatrix}Q_0^j\\\mathbf{1}_{n\times1} \end{pmatrix}$}. Analyzing the maximum independent rows of $[A_{J_{\by}}\; B_{J_{\by}}]$, we have
	{\small\begin{equation*}
		A_{I_\by}=\begin{pmatrix}
		\mathbf{1}_{1\times d}\otimes I_{n}&  \mathbf{1}_{n\times1} &\mathbf{0}&\mathbf{0}&\mathbf{0}\\
		Q_+&\mathbf{0}&-I &\mathbf{0}&\mathbf{0} \\
		-Q_-&\mathbf{0}&\mathbf{0}&-I&\mathbf{0}\\
		Q_0&\mathbf{0}&\mathbf{0}&\mathbf{0}&-I\\
		-\hat Q_0&\mathbf{0}&\mathbf{0}&\mathbf{0}&E
		\end{pmatrix},\;
		B_{I_\by}=\begin{pmatrix}
		-I  \\
		\mathbf{0}\\
		\mathbf{0}\\
		\mathbf{0}\\
		\mathbf{0}
		\end{pmatrix},
		\end{equation*}}
	where {\small
		$$
		\hat Q_0=
		\begin{pmatrix}
		\hat Q_0^1& & &\\
		&\hat Q_0^2 & &\\
		&  &\ddots &\\
		&  &&\hat Q_0^d
		\end{pmatrix}
		$$}
	is a full-rank matrix and, after appropriately re-ordering, the block of rows
	{\small
		$\begin{pmatrix}
		-\hat Q_0&\mathbf{0}_{N\times 1}&\mathbf{0}&\mathbf{0}&E
		\end{pmatrix}$}
	is the sub-matrix of
	{\small\begin{equation*}
		\begin{pmatrix}
		-Q_0&\mathbf{0}&\mathbf{0}&\mathbf{0}&-I\\
		-I_{d}\otimes\mathbf{1}_{1\times n}&\mathbf{0}&\mathbf{0}&\mathbf{0}&\mathbf{0}
		\end{pmatrix}
		\end{equation*}}
	contained in $A_{J_\by}$ and $E$ is the corresponding sub-matrix of {\small$\begin{pmatrix}-I\\\mathbf{0} \end{pmatrix}$} after the same re-ordering.
	
	Therefore, $|I_\by|=n+\sum_{j=1}^d|I_+^j|+\sum_{j=1}^d|I_-^j|+\sum_{j=1}^d|I_0^j|+\mathrm{rank}(\hat Q_0)$ and
	$$
	\mathrm{rank}(A_{I_\by})=\sum_{j=1}^d|I_+^j|+\sum_{j=1}^d|I_-^j|+\sum_{j=1}^d|I_0^j|+\mathrm{rank}\left(
	\begin{pmatrix}
	\mathbf{1}_{1\times d}\otimes I_{n}&  \mathbf{1}_{n\times1} \\
	\hat Q_0&\mathbf{0}
	\end{pmatrix}
	\right).
	$$
	Let $\widetilde Q_0^j$ for $j=1,\dots,d$ be the matrix whose rows form a basis of the linear space $\mathrm{ker}(\hat Q_0^j)$ and {\small
		$$
		\widetilde Q_0=
		\begin{pmatrix}
		\widetilde Q_0^1& & &\\
		&\widetilde Q_0^2 & &\\
		&  &\ddots &\\
		&  &&\widetilde Q_0^d
		\end{pmatrix}
		$$}
	As a result, the following matrix
	$$
	\begin{pmatrix}
	\mathbf{0}_{1\times nd}&1\\
	\widetilde Q_0&\mathbf{0}\\
	\hat Q_0&\mathbf{0}
	\end{pmatrix}
	$$
	should be a $(nd+1)\times (nd+1)$ invertible matrix. Hence,
	\begin{eqnarray*}
		\mathrm{rank}\left(
		\begin{pmatrix}
			\mathbf{1}_{1\times d}\otimes I_{n}&  \mathbf{1}_{n\times1} \\
			\hat Q_0&\mathbf{0}_{N\times 1}
		\end{pmatrix}
		\right)
		&=&
		\mathrm{rank}\left(
		\begin{pmatrix}
			\mathbf{1}_{1\times d}\otimes I_{n}&  \mathbf{1}_{n\times1} \\
			\hat Q_0&\mathbf{0}
		\end{pmatrix}\cdot
		\begin{pmatrix}
			\mathbf{0}_{nd\times 1}&\widetilde Q_0^\top&\hat Q_0^\top\\
			1&\mathbf{0}&\mathbf{0}
		\end{pmatrix}
		\right)\\
		&=&
		\mathrm{rank}\left(
		\begin{pmatrix}
			\mathbf{1}_{n\times1}&\widetilde Q_0^{1\top}&\cdots&\widetilde Q_0^{d\top}&  \hat Q_0^{1\top}&\cdots&\hat Q_0^{d\top} \\
			\mathbf{0}&\mathbf{0}&\cdots&\mathbf{0}& \hat Q_0^1\hat Q_0^{1\top}&\cdots&\mathbf{0}\\
			\vdots&\vdots&\ddots&\vdots& \vdots&\ddots&\vdots\\
			\mathbf{0}&\mathbf{0}&\cdots&\mathbf{0}& \mathbf{0}&\cdots&\hat Q_0^d\hat Q_0^{d\top}
		\end{pmatrix}
		\right)\\
		&=&
		\mathrm{rank}\left(\begin{pmatrix}
			\mathbf{1}_{n\times1}&\widetilde Q_0^{1\top}&\cdots&\widetilde Q_0^{d\top}
		\end{pmatrix}\right)
		+\mathrm{rank}(\hat Q_0\hat Q_0^\top)\\
		&=&
		\mathrm{rank}\left(\begin{pmatrix}
			\mathbf{1}_{n\times1}&\widetilde Q_0^{1\top}&\cdots&\widetilde Q_0^{d\top}
		\end{pmatrix}\right)
		+\mathrm{rank}(\hat Q_0).
	\end{eqnarray*}
	It is easy to verify that $-\bd=A^\top\bu$ for $\bu=\begin{pmatrix}
	\mathbf{0}_{1\times 2n}& \frac{\tau}{2}\mathbf{1}_{1\times 2N}& \mathbf{0}_{1\times 2d}
	\end{pmatrix}$.
	Hence, according to Theorem \ref{thm:div}, for a.e.~$\by$, we have
	\begin{eqnarray*}
		\mathrm{df} (\hbt(\by))&=& n- |I_\by| +\E[\mathrm{rank}(A_{I_\by})]\\
		&=&\E\left[\mathrm{rank}\left(\begin{pmatrix}
			\mathbf{1}_{n\times1}&\widetilde Q_0^{1\top}&\cdots&\widetilde Q_0^{d\top}
		\end{pmatrix}\right)\right]\\
		&=&\E[\mathrm{dim}(\text{span}\{\mathbf{1}_{n\times1},\mathrm{ker}(K_1),\dots,\mathrm{ker}(K_d)\})].
	\end{eqnarray*}
\end{proof}

\subsection{DF for the $\ell_\infty$-regularized  group Lasso problem}
\label{sec:proof_group}

\begin{repcorollary}{thm:grouplasso_gen}
	In the $\ell_\infty$-regularized group Lasso problem in \eqref{eq:grouplasso_gen} and \eqref{eq:grouplasso_gen_new}, for a.e.~$\by\in\mathbb{R}^n$,
	\[
	\mathrm{df}(\hat\bt(\by))=\mathrm{df} (X\hat\bbt(\by))=\E[\mathrm{rank}(X_{J_0^c})],
	\]
	where
	\[
	J_0=\{i \in \{1,\ldots, d\}:i\in\mathcal{G}_j, \hat\beta_i(\by)=\|\hat\bbt_{\mathcal{G}_j}(\by)\|_\infty \; \text{ for some } j\in\{1,2,\dots,l\}\},
	\]
	and $J_0^c$ is the complement set of $J_0$ and $X_{J_0^c}$ consists of columns of $X$ indexed by $J_0^c$.
\end{repcorollary}

\begin{proof}[Proof of Corollary~\ref{thm:grouplasso_gen}]
	Letting $\bxi=(\bbt^\top,\bg^\top)^\top$ and $\bt=X\bbt$ in \eqref{eq:grouplasso_gen} and defining $E$ as a $d\times l$ matrix with $E_{ij}=1$ if $i\in\mathcal{G}_j$ and $E_{ij}=0$ otherwise, the $\ell_\infty$-group Lasso problem can be reformulated as a special case of \eqref{eq:LSE_div} as shown in \eqref{eq:grouplasso_gen_new} with
	with
	\begin{equation*}
	\bd=(\mathbf{0}_{1 \times d}, \tau \mathbf{1}_{1 \times l})^T, \; \lambda=0, \; A=\begin{pmatrix}
	X & \mathbf{0}_{n \times l} \\
	-X & \mathbf{0}_{n \times l} \\
	I_d  & - E  \\
	-I_d  & - E
	\end{pmatrix}, \; B=\begin{pmatrix}
	-I_{n}  \\
	I_{n}  \\
	\mathbf{0}_{d \times n}   \\
	\mathbf{0}_{d \times n}
	\end{pmatrix}; \;\; \bc=\mathbf{0}.
	\end{equation*}
	
	We define three mutually disjoint sets of indexes as:
	\begin{eqnarray*}
		S_+ &:=& \{i:0<\hat\beta_i(\by)=\|\hat\bbt_{\mathcal{G}_j}(\by)\|_\infty, i\in\mathcal{G}_j\}, \\
		S_- &:=& \{i:0<-\hat\beta_i(\by)=\|\hat\bbt_{\mathcal{G}_j}(\by)\|_\infty, i\in\mathcal{G}_j\}, \\
		S_0 &:=& \{i:0=\hat\beta_i(\by)=\|\hat\bbt_{\mathcal{G}_j}(\by)\|_\infty, i\in\mathcal{G}_j\}.
	\end{eqnarray*}
	According to the definition of $J_0$, we can show that $J_0=S_+\cup S_-\cup S_0$.
	According to the constraints $ \bbt_{\mathcal{G}_j}-\gamma_j\mathbf{1}_{|\mathcal{G}_j|}  \leq \mathbf{0}$ and $ -\bbt_{\mathcal{G}_j}-\gamma_j\mathbf{1}_{|\mathcal{G}_j|} \leq \mathbf{0}$ in \eqref{eq:grouplasso_gen_new}, the optimality of $\hat\gamma_i(\by)$ will ensure $\hat \gamma_i(\by)=\|\hat\bbt_{\mathcal{G}_j}(\by)\|_\infty$, which implies that $\hat\beta_i(\by)-\hat\gamma_j(\by)= 0$ for $i\in S_+\cup S_0$ and $i\in\mathcal{G}_j$ and $-\hat\beta_i(\by)-\hat\gamma_j(\by)= 0$ for $i\in S_-\cup S_0$ and $i\in\mathcal{G}_j$.
	
	We define $E_+$, $E_-$ and $E_0$ as the sub-matrices of $E$ consisting of the rows of $E$ indexed by $S_+$, $S_-$ and $S_0$, respectively. By ordering $\bxi=(\bbt^\top,\bg^\top)^\top=(\bbt^\top_{J_0^c},\bbt_{I_+}^\top,\bbt_{I_-}^\top,\bbt_{I_0}^\top,\bg^\top)^\top$, we can represent the matrices $A_{J_\by}$ and $B_{J_\by}$ as
	\begin{eqnarray*}
		A_{J_\by}=
		\left(
		\begin{array}{ccccc}
			X^c & X_+ & X_- & X_0&  0\\
			-X^c & -X_+ & -X_- & -X_0&  0\\
			0 & I_{|S_+|} & 0 & 0 & -E_+ \\
			0 & 0 & -I_{|S_-|} & 0 & -E_- \\
			0 & 0 & 0 & I_{|S_0|} & -E_0\\
			0 & 0 & 0 & -I_{|S_0|} & -E_0\\
		\end{array}
		\right) \quad \mbox{ and } \quad
		B_{J_\by}=\left(
		\begin{array}{c}
			-I \\
			I \\
			0 \\
			0 \\
			0 \\
			0 \\
		\end{array}
		\right).
	\end{eqnarray*}
	Let $\hat S_0$ be the subset of $S_0$ and let the sub-matrix $\hat E_0$ of $E_0$ consist of the rows indexed by $\hat S_0$. We choose $\hat S_0$ so that $\hat E_0$ actually consists of the maximum number of linearly independent rows of $E_0$. Suppose $\hat E_0$ has $\hat s$ rows.  We have
	\begin{eqnarray*}
		A_{I_\by}=
		\left(
		\begin{array}{ccccc}
			X^c & X_+ & X_- & X_0&  0\\
			0 & I_{|S_+|} & 0 & 0 & -E_+ \\
			0 & 0 & -I_{|S_-|} & 0 & -E_- \\
			0 & 0 & 0 & I_{|S_0|} & -E_0\\
			0 & 0 & 0 & -I_{\hat s} & -\hat E_0\\
		\end{array}
		\right) \quad  \mbox{ and } \quad
		B_{I_\by}=\left(
		\begin{array}{c}
			-I \\
			0 \\
			0 \\
			0 \\
			0 \\
		\end{array}
		\right).
	\end{eqnarray*}
	Therefore, $|I_\by|=n+|S_+|+|S_-|+|S_0|+\hat s$ and $\mathrm{rank}(A_{I_\by})=|S_+|+|S_-|+|S_0|+\mathrm{rank}(X^c)+\hat s$.
	It is easy to verify that $-\bd=A^\top\bu$ for $\bu=\begin{pmatrix}
	\mathbf{0}_{1\times 2n}& \frac{\tau}{2|\G_1|}\mathbf{1}_{1\times |\G_1|}&\cdots& \frac{\tau}{2|\G_l|}\mathbf{1}_{1\times |\G_l|}& \frac{\tau}{2|\G_1|}\mathbf{1}_{1\times |\G_1|}&\cdots& \frac{\tau}{2|\G_l|}\mathbf{1}_{1\times |\G_l|}
	\end{pmatrix}$.
	Hence, according to Theorem \ref{thm:div}, for a.e.~$\by$, we have
	\begin{eqnarray*}
		\mathrm{df} (X\bbt(\by))&=&\mathrm{df} (\hbt(\by))\\
		&=& n- |I_\by| +\E[\mathrm{rank}(A_{I_\by})]\\
		&=&\E[\mathrm{rank}(X^c)].
	\end{eqnarray*}
\end{proof}

\subsection{Recovering existing results: Lasso, generalized Lasso, linear,  and ridge regression}
\label{sec:recover_other}

The generalized Lasso can be formulated as the following optimization problem \citep{Tib11Gen,TT12}:
\begin{equation}\label{eq:lasso_gen}
\hbbt(\by)\in\argmin_{\bbt\in\mathbb{R}^d}\frac{1}{2}\|\by-X\bbt\|_2^2+\tau\|D\bbt\|_1,
\end{equation}
where $D$ is a given $l\times d$ matrix. When $D=I_d$ (and $l=d$), it reduces to the standard Lasso problem. To see why \eqref{eq:lasso_gen} is a special case of our general optimization formulation in \eqref{eq:LSE}, note that \eqref{eq:lasso_gen} can be re-written as
\begin{equation}\label{eq:lasso_gen_equiv}
(\hat\bbt(\by),\hat\bg(\by))\in\argmin_{-\bg\leq D\bbt\leq\bg}\frac{1}{2}\|\by-X\bbt\|_2^2
+\tau\mathbf{1}^\top\bg.
\end{equation}
Letting $\bt=X\bbt$, the formulation in \eqref{eq:lasso_gen_equiv} is further equivalent to
\begin{eqnarray}\label{eq:lasso_gen_new}
(\hbt(\by), \hat\bbt(\by),\hat\bg(\by)) & \in &\argmin_{\bt, \bbt,\bg } \frac{1}{2}\|\bt-\by\|_2^2+\tau\mathbf{1}^\top\bg \\
& & \;\;\;  \text{s.t.}  \;\; X\bbt - \bt \leq \mathbf{0}, \;\; -X\bbt + \bt \leq \mathbf{0} \nonumber\\
& & \;\;\;  \;\;\;\;   \;\; \;D\bbt-\bg  \leq \mathbf{0}, \;\; -D\bbt-\bg  \leq \mathbf{0}.\nonumber
\end{eqnarray}
By setting $\bxi=(\bbt^\top,\bg^\top)^\top$, the optimization problem in~\eqref{eq:lasso_gen_new} is a special case of \eqref{eq:LSE} with
{\small \begin{equation}\label{eq:genLasso_A_B}
	\bd=(\mathbf{0}_{1 \times d}, \tau \mathbf{1}_{1 \times l})^T, \; \lambda=0, \; A=\begin{pmatrix}
	X & \mathbf{0}_{n \times l} \\
	-X & \mathbf{0}_{n \times l} \\
	D  & - I_l  \\
	-D  & - I_l
	\end{pmatrix}, \; B=\begin{pmatrix}
	-I_{n}  \\
	I_{n}  \\
	\mathbf{0}_{l \times n}   \\
	\mathbf{0}_{l \times n}
	\end{pmatrix}, \bc=\mathbf{0}.
	\end{equation}}
\citet{TT12} computed the DF of $\hat{\bt}(\by)=X\hbbt(\by)$ for generalized Lasso (see Theorem 3 of \citet{TT12}). In the next corollary, we show that the result of~\citet{TT12} can be obtained as a direct consequence of our general theory (Theorem \ref{thm:div}).

\begin{corollary}\label{thm:div_lasso_gen}
	In the generalized Lasso problem in \eqref{eq:lasso_gen} and \eqref{eq:lasso_gen_new}, for a.e.~$\by\in\mathbb{R}^n$, $\mathrm{df}(\hat\bt(\by))= \mathrm{df} (X\hat\bbt(\by))=\E[\mathrm{dim}(X\mathrm{ker}(D_0))],$
	where  $D_0 \in \R^{l_0 \times d}$  is the sub-matrix of $D$ consisting of rows $\mathbf{d}_i$'s of $D$ such that $\mathbf{d}_i^\top\hat\bbt(\by)=0$ and
	$\mathrm{ker}(D_0) :=\{\bx\in\mathbb{R}^d : D_0\bx=\mathbf{0}\}$ is the kernel of $D_0$.
\end{corollary}

\begin{proof}[Proof of Corollary \ref{thm:div_lasso_gen}]
	Letting $\bxi=(\bbt^\top,\bg^\top)^\top$ and $\bt=X\bbt$ in \eqref{eq:lasso_gen}, the generalized Lasso problem can be reformulated as a special case of \eqref{eq:supp_LSE} as shown in \eqref{eq:lasso_gen_new}.
	We partition $\{1,2,\dots,l\}$ into three sets of indexes as:
	\begin{eqnarray*}
		I_+ := \{i:\mathbf{d}_i^\top\hat\bbt(\by)>0\}, \quad I_- := \{i:\mathbf{d}_i^\top\hat\bbt(\by)<0\}, \quad I_0 := \{i:\mathbf{d}_i^\top\hat\bbt(\by)=0\}.
	\end{eqnarray*}
	According to the constraints $ D\bbt-\bg \leq \mathbf{0}$ and $-D\bbt-\bg  \leq \mathbf{0}$ in \eqref{eq:lasso_gen_new}, the optimality of $\hat\gamma_i(\by)$ will ensure $\hat \gamma_i(\by)=\max(\mathbf{d}_i^\top\hat\bbt(\by),-\mathbf{d}_i^\top\hat\bbt(\by))$, which implies that $\mathbf{d}_i^\top\hat\bbt(\by)-\hat\gamma_i(\by)=0$ for $i\in I_+\cup I_0$ and $-\mathbf{d}_i^\top\hat\bbt(\by)-\hat\gamma_i(\by)=0$ for $i\in I_-\cup I_0$.
	
	We define $D_+$, $D_-$ and $D_0$ as the sub-matrices of $D$ consisting of the rows of $D$ indexed by $I_+$, $I_-$ and $I_0$, respectively. By ordering $\bxi=(\bbt^\top,\bg^\top)^\top=(\bbt^\top,\bg_{I_+}^\top,\bg_{I_-}^\top,\bg_{I_0}^\top)^\top$, we can represent the matrices $A_{J_\by}$ and $B_{J_\by}$ as
	\begin{eqnarray*}
		A_{J_\by}=
		\left(
		\begin{array}{cccc}
			X & 0 & 0 & 0 \\
			-X & 0 & 0 & 0 \\
			D_+ & -I & 0 & 0 \\
			-D_- & 0 & -I & 0 \\
			D_0 & 0 & 0 & -I \\
			-D_0 & 0 & 0 & -I \\
		\end{array}
		\right) \quad \mbox{ and } \quad
		B_{J_\by}=\left(
		\begin{array}{c}
			-I \\
			I \\
			0 \\
			0 \\
			0 \\
			0 \\
		\end{array}
		\right).
	\end{eqnarray*}
	Let $\hat D_0$ be the sub-matrix of $D_0$ that contains the maximum number of linearly independent rows of $D_0$. Suppose $\hat D_0$ has $\hat l$ rows.  We have
	\begin{eqnarray*}
		A_{I_\by}=
		\left(
		\begin{array}{cccc}
			X & 0 & 0 & 0 \\
			D_+ & -I & 0 & 0 \\
			-D_- & 0 & -I & 0 \\
			D_0 & 0 & 0 & -I \\
			-\hat D_0 & 0 & 0 & [-I_{\hat l}\;\; 0] \\
		\end{array}
		\right) \quad  \mbox{ and } \quad
		B_{I_\by}=\left(
		\begin{array}{c}
			-I \\
			0 \\
			0 \\
			0 \\
			0 \\
		\end{array}
		\right).
	\end{eqnarray*}
	Therefore, $|I_\by|=n+|I_+|+|I_-|+|I_0|+\mathrm{rank}(\hat D_0)$ and $\mathrm{rank}(A_{I_\by})=|I_+|+|I_-|+|I_0|+\mathrm{rank}([X^\top,\hat D_0^\top])$. Let $\hat D_0^c$ be an $(d-\hat l)\times d$ matrix whose rows form a basis of the linear space $\text{ker}(\hat D_0)$. Then $[(\hat D_0^c)^\top,\hat D_0^\top]$ becomes a $d\times d$ invertible matrix. It is easy to verify that $-\bd=A^\top\bu$ for $\bu=\begin{pmatrix}
	\mathbf{0}_{1\times 2n}& \frac{\tau}{2}\mathbf{1}_{1\times 2l}
	\end{pmatrix}$. Hence,
	\begin{eqnarray*}
		\mathrm{rank}\left(
		\left[
		\begin{array}{c}
			X \\
			\hat D_0 \\
		\end{array}
		\right]
		\right)&=&
		\mathrm{rank}\left(
		\left[
		\begin{array}{c}
			X \\
			\hat D_0 \\
		\end{array}
		\right]\cdot[(\hat D_0^c)^\top,\hat D_0^\top]
		\right)\\
		&=&
		\mathrm{rank}\left(
		\left[
		\begin{array}{cc}
			X(\hat D_0^c)^\top & X\hat D_0^\top \\
			0 & \hat D_0\hat D_0^\top \\
		\end{array}
		\right]
		\right)\\
		&=&
		\mathrm{rank}(X(\hat D_0^c)^\top)+\mathrm{rank}(\hat D_0\hat D_0^\top)\\
		&=&\mathrm{rank}(X(\hat D_0^c)^\top)+\mathrm{rank}(\hat D_0).
	\end{eqnarray*}
	According to Theorem \ref{thm:div}, for a.e.~$\by$, we have
	\begin{eqnarray*}
		\mathrm{df} (X\bbt(\by))&=&\mathrm{df} (\hbt(\by))\\
		&=& n- |I_\by| +\E[\mathrm{rank}(A_{I_\by})]\\
		&=&\E[\mathrm{rank}(X(\hat D_0^c)^\top)]\\
		&=&\E[\mathrm{dim}(X\mathrm{ker}(D_0))].
	\end{eqnarray*}
\end{proof}

Corollary \ref{thm:div_lasso_gen} is true even when \eqref{eq:lasso_gen} have multiple optimal solutions $\hat\bbt(\by)$s and the matrix $D_0$ can be different for each optimal solution. In fact, even if different optimal solutions $\hat\bbt(\by)$s correspond to different $D_0$s, the divergence of $\hat\bt(\by)=X\hat\bbt(\by)$ will always be the same for a.e. $\by$.

Note that the standard Lasso is a special case of generalized Lasso (see~\eqref{eq:lasso_gen}) with $D=I_d$. In the next corollary we provide the DF of $X\hat\bbt(\by)$ for the Lasso estimator $\hat\bbt(\by)$. It recovers the result in Theorem 1 in \citet{Zou07} and Theorem 2 in \citet{TT12}.


\begin{corollary}\label{thm:div_lasso}
	In the Lasso problem \eqref{eq:lasso_gen} with $D=I_d$, for a.e.~$\by\in\mathbb{R}^n$, $ \mathrm{df}(\hat\bt(\by))=\mathrm{df} (X\hat\bbt(\by))=\E[\mathrm{rank}(X_{J_0^c})],
	$ where  $J_0=\{1\le i \le d:\hat\beta_i(\by)=0\}$, $J_0^c$ is the complement of $J_0$ and $X_{J_0^c}$ contains columns of $X$ indexed by $J_0^c$.
\end{corollary}

It is worthwhile to note that Corollary \ref{thm:div_lasso} is true when $X$ does not have rank $p$. Note that when $X$ doesn't have rank $p$,  \eqref{eq:lasso_gen} with $D=I_d$ can have multiple optimal solutions $\hbbt(\by)$s and the inactive set $J_0$ can be different for each optimal solution. However, Corollary \ref{thm:div_lasso} does not require the inactive set $J_0$ to be unique and holds for any optimal solution  $\hbbt(\by)$ in \eqref{eq:lasso_gen}.
In fact, for different optimal solutions $\bbt(\by)$s with different $J_0$s, the divergence of $\hat\bt(\by)=X\hat\bbt(\by)$ will always be the same for a.e. $\by$.

\begin{proof}[Proof of Corollary~\ref{thm:div_lasso}]
	In the special case of~\eqref{eq:lasso_gen} with $D=I_d$, the matrix $D_0$ in Corollary~\ref{thm:div_lasso_gen} consists of the rows of $I_d$ indexed by $J_0$, which is essentially a projection matrix from $\mathbb{R}^d$ to the coordinates indexed by $J_0$. Therefore, $\mathrm{ker}(D_0)=\{\bx\in\mathbb{R}^d:\bx_i=0, \forall i\in J_0\}$ so that $\mathrm{dim}(X\mathrm{ker}(D_0))$=$\mathrm{rank}(X_{J_0^c})$ and the conclusion follows.
\end{proof}


The classical results on the DF of linear and ridge regression (see \citet{Li86}) can also be readily obtained as simple consequences of Theorem \ref{thm:div}. 

For linear regression, given the response vector $\by\in\mathbb{R}^n$ and the design matrix $X\in\mathbb{R}^{n\times p}$, the ordinary LSE is defined as
\begin{equation}\label{eq:lr}
\hbbt(\by)\in\argmin_{\bbt \in\mathbb{R}^p}\frac{1}{2}\|\by-X\bbt\|_2^2.
\end{equation}
By setting $\bxi=\bbt$ and $\bt=X\bbt$, \eqref{eq:lr} can be reformulated as a special case of \eqref{eq:LSE_mul_cvx}, i.e.,
\begin{equation}\label{eq:lr_new}
\quad (\hbt(\by), \hbxi(\by)) \; \in \; \argmin_{\bt, \bxi } \frac{1}{2}\|\bt-\by\|_2^2 \qquad \textrm{s.t.}  \; \begin{pmatrix} X \\ -X \end{pmatrix} \bxi + \begin{pmatrix} -I_n \\  I_n \end{pmatrix} \bt \leq \mathbf{0},
\end{equation}
which is in the form of \eqref{eq:supp_LSE} with $A=\begin{pmatrix} X \\ -X \end{pmatrix}$, $B=\begin{pmatrix} -I_n \\  I_n \end{pmatrix}$, $\bc=\mathbf{0}$, $\bd=\mathbf{0}$ and $\lambda=0$. Theorem \ref{thm:div} directly implies the following corollary, which establishes the well-known result that for the LSE $ \mathrm{df} (X\hbbt(\by))= \mathrm{rank}(X)$.


\begin{corollary}\label{prop:linearregression}
	Let $\hbbt(\by)$ be the ordinary LSE (i.e., $\hbbt(\by)\in\argmin_{\bbt \in\mathbb{R}^p}\frac{1}{2}\|\by-X\bbt\|_2^2$). The divergence of $\hbt(\by)=X\hbbt(\by)$ equals $\mathrm{rank}(X)$ a.s. Thus, $\mathrm{df}(X \hbbt(\by))=\mathrm{rank}(X)$.
\end{corollary}

\begin{proof}[Proof of Corollary \ref{prop:linearregression}]
	Note that, an equivalent formulation of the LSE given in \eqref{eq:lr_new} is a special case of \eqref{eq:supp_LSE} when $\bd=\mathbf{0}$.
	Since each feasible solution of \eqref{eq:lr_new} must satisfy $X \bxi - \bt=0$, $J_\by$, as defined in~\eqref{eq:J_y_2}, includes all the constraints of \eqref{eq:lr_new} and $A_{J_\by}=\left[X^\top, -X^\top\right]^\top$ and $B_{J_\by}=\left[ -I_n, I_n \right]^\top$. Since $B_{J_\by}$ contains $I_n$,  all the rows of $[A_{J_\by}, B_{J_\by}]$ are linear independent and thus  $I_\by=J_\by$ with $|I_\by|=n$.
	According to Theorem \ref{thm:div}, for a.e.~$\by$, we have
	\begin{eqnarray*}
		\mathrm{df} (X\hbbt(\by)) =\mathrm{df} (\hbt(\by))
		= n- |I_\by| + \E[\mathrm{rank}(A_{I_\by})]
		=\mathrm{rank}(X).
	\end{eqnarray*}
\end{proof}


Ridge regression, described as
\begin{equation}\label{eq:ridge}
\hat{\bbt}_\lambda(\by)=\argmin_{\bbt\in\mathbb{R}^d}\frac{1}{2}\|\by-X\bbt\|_2^2+\frac{\lambda}{2}\|\bbt\|_2^2,
\end{equation}
can also be shown to be a special case of the general optimization problem~\eqref{eq:supp_LSE} by letting and $\bxi=\bbt$  and $\bt= X \bbt$. In particular, using the same reformulation as in \eqref{eq:lr_new}, the ridge estimator  in  \eqref{eq:ridge} is a special case of \eqref{eq:supp_LSE} with $A$ and $B$ as in~\eqref{eq:lr_new}, $\bc=\mathbf{0}$, $\bd=\mathbf{0}$ and $\lambda>0$. Theorem~\ref{thm:div} can be applied to \eqref{eq:ridge} to obtain $\mathrm{df} (X\hat\bbt(\by))$.

\begin{corollary}\label{thm:div_ridge}
	In ridge regression $\hat{\bbt}_\lambda(\by)=\argmin_{\bbt\in\mathbb{R}^d}\frac{1}{2}\|\by-X\bbt\|_2^2+\frac{\lambda}{2}\|\bbt\|_2^2$. For a.e.~$\by\in\mathbb{R}^n$,
	$\mathrm{df} (X\hat\bbt_\lambda(\by))=\mathrm{trace}\left(X\left( \lambda I_d+X^\top X\right)^{-1}X^\top\right)$.
\end{corollary}

\begin{proof}[Proof of Corollary \ref{thm:div_ridge}]
	By setting $\bxi=\bbt$ and $\bt=X\bbt$, \eqref{eq:ridge} can be reformulated as a special case of \eqref{eq:LSE_Lip_multi_conv}, i.e.,
	\begin{eqnarray}\label{eq:ridge_new}
	(\hbt_\lambda(\by), \hbxi_\lambda(\by)) & = &\argmin_{\bt, \bxi } \frac{1}{2}\|\bt-\by\|_2^2+\frac{\lambda}{2}\|\bxi\|_2^2 \\
	& & \;\;\;  \text{s.t.}  \;\; X \bxi - \bt \leq \mathbf{0} \nonumber\\
	& & \;\;\;  \;\;  \;\; -X \bxi + \bt \leq \mathbf{0}. \nonumber
	\end{eqnarray}
	Since each feasible solution of \eqref{eq:ridge_new} must satisfy $X \bxi - \bt=0$, $J_\by$ includes all the constraints of \eqref{eq:ridge_new} and thus $A_{J_\by}=\left[X^\top, -X^\top\right]^\top$ and $B_{J_\by}=\left[ -I_n, I_n \right]^\top.$
	It is easy to see that $A_{I_\by}=X$ and $B_{I_\by}=-I_n
	$. According to Theorem \ref{thm:div}, for a.e.~$\by \in \R^n$, we have
	\begin{eqnarray*}
		\mathrm{df} (X\hat\bbt_\lambda(\by))&=&\mathrm{df} (\hbt_\lambda(\by))\\
		&=&n - \mathrm{trace}\left(I_n+ \frac{1}{\lambda} XX^\top\right)^{-1}\\
		&=&n -\mathrm{trace}\left(I_n \right)+\mathrm{trace}\left(X\left( \lambda I_d+X^\top X\right)^{-1}X^\top\right)\\
		&=&\mathrm{trace}\left(X\left( \lambda I_d+X^\top X\right)^{-1}X^\top\right),
	\end{eqnarray*}
	where the third equality is due to the Sherman-Morrison-Woodbury formula.
\end{proof}

\bibliographystyle{Chicago}
\bibliography{AG}

\end{document}